\documentclass{cpamart1}     %% Input Standard CPAM class.
\volume{000}
\startingpage{1}                      

\authorheadline{T.\ Berry and D.\ Giannakis}
\titleheadline{Spectral Exterior Calculus}

\usepackage{amssymb}
\usepackage{graphicx}
\usepackage{bm}
\usepackage{tabularx}
\usepackage{tikz-cd}
\usepackage{enumerate}
\usepackage{rotating}
\usepackage{mathptmx}
\newtheorem{theorem}{Theorem}[section]
\newtheorem{lemma}[theorem]{Lemma}
\newtheorem{corollary}[theorem]{Corollary}
\newtheorem{proposition}[theorem]{Proposition}
\newtheorem{conj}[theorem]{Conjecture}
\theoremstyle{definition}
\newtheorem{definition}[theorem]{Definition}
\newtheorem{exmp}[theorem]{Example}
\theoremstyle{remark}
\newtheorem{remark}[theorem]{Remark}

\DeclareMathOperator{\grad}{grad}
\DeclareMathOperator{\wgrad}{\overline{grad}}
\DeclareMathOperator{\divr}{div}
\DeclareMathOperator{\wdivr}{\overline{div}}

\DeclareMathOperator{\spn}{span}
\DeclareMathOperator{\rank}{rank}

\DeclareMathOperator{\ran}{ran}
\DeclareMathOperator{\diag}{diag}
\DeclareMathOperator{\tr}{tr}
\DeclareMathOperator*{\esssup}{ess\,sup}

\newcommand{\widecheck}{\check}

\allowdisplaybreaks

\begin{document}                        %% Standard LaTeX command

%%      -----------------------------------------------------------------------
%%      -------------------------------- TITLE -----------------------------
%%      -----------------------------------------------------------------------

\title{Spectral Exterior Calculus}

%%      -----------------------------------------------------------------------------
%%      ------------------------------- AUTHORS -----------------------------
%%      ----------------------------------------------------------------------------
\author{Tyrus Berry}{Department of Mathematics, George Mason University}
% EXAMPLE: \author{Bart Simpson}{Universit‰ Paris-Sorbonne (Paris IV)}
% Uncomment and fill in the following lines as needed
\author{Dimitrios Giannakis}{Courant Institute of Mathematical Sciences, New York University}
%\author{*** THIRD AUTHOR'S NAME ***}{*** THIRD AUTHOR'S AFFILIATION WHEN ARTICLE WAS WRITTEN ***}
%\author{*** FOURTH AUTHOR'S NAME ***}{*** FOURTH AUTHOR'S AFFILIATION WHEN ARTICLE WAS WRITTEN ***}
%\author{*** FIFTH AUTHOR'S NAME ***}{*** FIFTH AUTHOR'S AFFILIATION WHEN ARTICLE WAS WRITTEN ***}
% Add additional names and affiliations as necessary using above format
%%      ---------------------------------------------------------------------
%%      --------------------------- DEDICATION  (OPTIONAL)------------------- 
%%      ---------------------------------------------------------------------

%       Uncomment the following line to insert a dedication.

%\dedication{ *** DEDICATION *** }        %% Enter dedication between braces.

%%      ------------------------------------------------------------------------------------
%%      --------------------------- ABSTRACT (OPTIONAL)----------------------
%%      ------------------------------------------------------------------------------------

%% ***** UNCOMMENT THE FOLLOWING TO INSERT AN ABSTRACT *****

\begin{abstract}
    A spectral approach to building the exterior calculus in manifold learning problems is developed.  The spectral approach is shown to converge to the true exterior calculus in the limit of large data.  Simultaneously, the spectral approach decouples the memory requirements from the amount of data points and ambient space dimension.  To achieve this, the exterior calculus is reformulated entirely in terms of the eigenvalues and eigenfunctions of the Laplacian operator on functions.  The exterior derivatives of these eigenfunctions (and their wedge products) are shown to form a frame (a type of spanning set) for appropriate $L^2$ spaces of $k$-forms, as well as higher-order Sobolev spaces.  Formulas are derived to express the Laplace-de Rham operators on forms in terms of the eigenfunctions and eigenvalues of the Laplacian on functions.  By representing the Laplace-de Rham operators in this frame, spectral convergence results are obtained via Galerkin approximation techniques.  Numerical examples demonstrate accurate recovery of eigenvalues and eigenforms of the Laplace-de Rham operator on 1-forms.  The correct Betti numbers are obtained from the kernel of this operator approximated from data sampled on several orientable and non-orientable manifolds, and the eigenforms are visualized via their corresponding vector fields.  These vector fields form a natural orthonormal basis for the space of square-integrable vector fields, and are ordered by a Dirichlet energy functional which measures oscillatory behavior. The spectral framework also shows promising results on a non-smooth example (the Lorenz 63 attractor), suggesting that a spectral formulation of exterior calculus may be feasible in spaces with no differentiable structure.  
\end{abstract}

% With AMS-LaTeX, \maketitle follows the abstract
\maketitle   

%%      ---------------------------------------------------------------------
%%      ------------------- TABLE OF CONTENTS (OPTIONAL) --------------------
%%      ---------------------------------------------------------------------

%% ***** IF YOUR PAPER IS OVER 40 PAGES AND YOU WISH TO HAVE A TABLE
%% ***** OF CONTENTS, PLEASE UNCOMMENT THE FOLLOWING LINE

 \tableofcontents

%%      ---------------------------------------------------------------------
%%      ---------------------------- BODY OF PAPER --------------------------
%%      ---------------------------------------------------------------------

%%      Please input or insert the body of your paper here.

\section{Introduction}

The field of manifold learning has focused significant attention recently on consistently estimating the Laplacian operator on a manifold, 
\begin{equation}
    \label{0Laplacian} \Delta = -\divr  \grad = \delta d 
\end{equation}
(in this paper we use the positive definite Laplacian, which we also refer to as the 0-Laplacian) \cite{BelkinNiyogi03,Singer06,CoifmanLafon06,VonLuxburgEtAl08,BerrySauer16b,HeinEtAl05,BerryHarlim16,belkin2007convergence,shi2015convergence,trillos2018error,trillos2018variational,BerrySauer16}.  Given data $\{x_i\} $ sampled from a manifold $\mathcal{M} \subset \mathbb{R}^n$, these methods build a graph with weights given by a kernel function $k(x_i,x_j)$, and then approximate the Laplacian operator with the graph Laplacian 
\begin{equation}
    \label{graphLaplacian} 
    \bm L= \bm D - \bm K, 
\end{equation}
where $\bm K$ and $\bm D$ are the kernel and degree matrices associated with $k $, respectively. In Table~\ref{fig0}, we briefly summarize the current state-of-the-art results.

\begin{table}
\caption{\label{fig0} Summary of results on manifold learning}
\hrulefill
\begin{enumerate}
    \item For uniform sampling density (with respect to the Riemannian volume measure) on a compact manifold, the Gaussian kernel provides a consistent pointwise estimator of the Laplace-Beltrami operator \cite{BelkinNiyogi03}.
    \item For nonuniform sampling density on a compact manifold, any isotropic kernel with exponential decay can be normalized to give a consistent pointwise estimator \cite{CoifmanLafon06}.
    \item For nonuniform sampling density on a compact manifold, any symmetric kernel with super-polynomial decay can be normalized to give a consistent pointwise estimator with respect to a geometry determined by the kernel function \cite{BerrySauer16b}.
    \item The bias-variance tradeoff implies error which is exponential in the dimension of the manifold \cite{Singer06,BerryHarlim16}.
    \item The above results can be generalized to non-compact manifolds by assuming appropriate lower bounds on the injectivity radius and either the curvature or the ratio between the intrinsic and extrinsic distances \cite{HeinEtAl05,BerryHarlim16,BerrySauer16}.    
    \item For data sampled on a compact subset of $\mathbb{R}^n$, not necessarily with manifold structure, the normalized and (under additional conditions) the unnormalized graph Laplacians converge spectrally to operators on continuous functions in the infinite-data limit \cite{VonLuxburgEtAl08}.
    \item For smooth compact manifolds without boundary and uniform sampling density, the graph Laplacian associated with Gaussian kernels converges spectrally to the manifold Laplacian along a decreasing sequence of kernel bandwidth parameters as the number of samples increases \cite{BelkinNiyogi07}. More recently, these results have been extended to allow spectral approximation of more general self-adjoint elliptic operators on bounded open subsets of $\mathbb{R}^n$, with specified relationships between the bandwidth and number of points, including error estimates  \cite{shi2015convergence,trillos2018error,trillos2018variational}.
    \item The bias and variance of the spectral estimator were computed in \cite{BerrySauer16}, who showed that the variance is dominated by two terms, one proportional to the eigenvalue $\lambda$ (linear) and another proportional to $\lambda^2$ (quadratic), explaining why an initial part of the spectrum (close to zero) can be significantly more accurate than larger eigenvalues.  For this initial part of the spectrum, the optimal bias-variance tradeoff results in a much smaller bandwidth than is optimal for pointwise estimation.
    \item A separate construction (closely related to kernel estimators) uses local estimators of the tangent space and orthogonal matrices that estimate the covariant derivative in order to construct an estimator of the connection Laplacian, which is closely related to the Hodge Laplacian on 1-forms \cite{singer2012vector,el2016graph,singer2016spectral}. 
\end{enumerate}
\hrulefill
\end{table}

The Laplacian-based approach to manifold learning is justified by the fact that the Laplace-Beltrami operator encodes all the geometric information about a Riemannian manifold.  A simple demonstration of this fact arises from the product formula for the Laplacian 
\begin{equation}
    \label{productrule} \Delta(fh) = f\Delta h + h\Delta f - 2\grad f \cdot \grad h, 
\end{equation}
where the dot-product above is actually the Riemannian inner product $g_x : T_x\mathcal{M} \times T_x\mathcal{M} \to \mathbb{R}$,
\begin{align}
    \nonumber g_x(\grad f(x),\grad h(x)) &= ( \grad f \cdot \grad h )(x ) \\
    \label{SECmetric} &= \frac{1}{2}(f(x)\Delta h(x) + h(x) \Delta f(x) - \Delta(fh)(x)). 
\end{align}
Specifically, given any vectors $v,w \in T_x\mathcal{M}$, there must exist functions $f,h$ with $\grad f(x)=v$ and $\grad h(x) = w$, and then the inner product 
\begin{displaymath}
    g_x(v,w) = g_x(\grad f(x),\grad h(x)) 
\end{displaymath}
can be computed as above.
 
Since the geometry of a Riemannian manifold is completely determined by the Riemannian metric, the above formulas show that metric is completely recoverable from Laplacian, so learning the Laplacian is sufficient for manifold learning.    Of course, this is a theoretical rather than pragmatic notion of sufficiency.  If one asks certain geometric questions, such as ``What is the 0-homology of the manifold?" (i.e., the number of connected components) this can be easily answered as the dimension of the kernel (nullspace) of the Laplacian.  However, if one asks for the higher homology of the manifold, or the harmonic vector fields, or the closed or exact forms, the above formulas do not suggest any practical approach.  What is needed is not merely the Laplacian, but a consistent representation of the entire exterior calculus on the manifold.

In this paper, we introduce the Spectral Exterior Calculus (SEC) as a consistent representation of the exterior calculus based entirely on the eigenfunctions and eigenvalues of the Laplacian on functions.  In essence, we will follow through on the above analysis and reformulate the entire exterior calculus in terms of these eigenfunctions and eigenvalues.  

Discrete formulations of the exterior calculus, utilizing a finite number of sampled points on or near the manifold, have been introduced at least as early as the mid 1970s with the work of Dodziuk \cite{Dodziuk76} and Dodziuk and Patodi \cite{DodziukPatodi76}. They constructed a combinatorial Laplacian on simplicial cochains of a smooth triangulated manifold, and showed that, under refinement of the triangulation, the combinatorial Laplacian on $k$-cochains converges in spectrum to the Laplace-de Rham operator on $ k$-forms (or $k$-Laplacian, denoted $\Delta_k$). A key element of this construction was the use of Whitney interpolating forms \cite{Whitney57}, mapping  $k$-cochains to $k$-forms on the manifold. More recently, two alternative methods of discretely representing the exterior calculus have been developed, namely the Discrete Exterior Calculus (DEC) by Hirani \cite{Hirani03} and Desbrun et al.\ \cite{DesbrunEtAl05}, and the Finite element Exterior Calculus (FEC) by Arnold et al.\ \cite{arnold2006finite,arnold2010finite}. Among these, the FEC includes the techniques of Dodziuk and Patodi, as well as subsequent generalizations by Baker \cite{Baker83} utilizing Sullivan-Whitney piecewise-polynomial forms \cite{Sullivan75}, as special cases. In Table~\ref{fig1}, we compare the features of the SEC to the DEC and FEC.  For manifold learning applications, we focus on three requirements: consistency, applicability to raw data, and amount of data and memory required.  

\begin{table}
    \caption{\label{fig1} Comparison of the Spectral Exterior Calculus (SEC) introduced in this manuscript to the Discrete Exterior Calculus (DEC) and the Finite element Exterior Calculus (FEC).}
\begin{center}
\begin{tabular*}{.75\linewidth}{@{\extracolsep{\fill}}lccc}
\hline
Feature & DEC & FEC & SEC \\
\hline\hline
Pointwise  consistency & Yes & Yes & Yes \\
Spectral  consistency & Unknown & Yes & Yes \\
Works on raw data & No & No & Yes \\
Decouples memory from data & No & N/A & Yes \\
Exterior Calculus structure & Partial & Partial & Partial \\
\hline
\end{tabular*}
\end{center}
\end{table}

The first requirement is that the method should be consistent, meaning that discrete analogs of objects from the exterior calculus should converge to their continuous counterparts in the limit of large data.  In this paper we will focus on the pointwise and spectral consistency, in appropriate Hilbert spaces, of representations of vector fields and the Laplace-de Rham operators on $k$-forms.  We chose to focus on Laplace-de Rham operators because their eigenforms form a natural ordered basis for the space of $k$-forms.  Moreover, in the case of the eigenforms of the $1$-Laplacian, the Riemannian duals form a natural basis for smooth vector fields.  While the DEC formulates a discrete analog to $\Delta_k$, currently it has not been proven to be pointwise consistent, except for $\Delta_0$ using the cotangent formula for the special case of surfaces in $\mathbb{R}^3$. A recent preprint by Schulz and Tsogtgerel \cite{schulz2016convergence} shows consistency of the DEC when used to solve Poisson problems for the $1$-Laplacian, but spectral convergence is not addressed.  A recent thesis  of Rufat \cite{SPEX} considers a collocation-based variant of the DEC, also called  ``Spectral Exterior Calculus" but  abbreviated SPEX, and shows numerical examples suggesting consistency of the kernel of the $1$-Laplacian. The FEC has convergence results for estimating Laplace-de Rham and related operators, including eigenvalue problems.  

In Section, \ref{secGalerkin}, we prove spectral convergence results for the SEC-approximated 1-Laplacian using a Galerkin technique. More generally, many of the operators encountered in exterior calculus, including the $k$-Laplacians, are unbounded, and the requirement of consistency must necessarily address domain issues for such operators. One of the advantageous aspects of the SEC is that the Sobolev regularity appropriate for differential operators such as $k$-Laplacians can be naturally enforced using the eigenvalues of the 0-Laplacian. This allows us to construct spectrally convergent Galerkin schemes using classical results from the spectral approximation theory for linear operators. This approach generalizes Galerkin approximation schemes for a class of unbounded operators on functions (generators of measure-preserving dynamical systems) \cite{GiannakisEtAl15,Giannakis19,DasGiannakis19} to operators acting on vector fields and higher-order $k$-forms.  

Our second requirement is that the method should only require raw data, as the assumption of an auxiliary structure such as a simplicial complex is unrealistic for many data science applications.  The FEC makes strong use of the known structure of the manifold to build their finite element constructions, which makes the FEC inappropriate for manifold learning. Indeed, it is instead targeted at solving PDEs on manifolds where the manifold structure is given explicitly.  Based on this requirement we will not consider further comparison to the FEC.  The DEC also makes strong use of a simplicial complex in their formulation and in the consistency results.  It is conceivable that one could apply the DEC to an abstract simplicial complex based on an $\epsilon$-ball or $k$-nearest neighbor construction, however there are no consistency results for such constructions.  

Our third requirement is that the memory requirements should be decoupled from the data requirements, since data sets may be very large, rendering any method requiring memory that is even quadratic in the data impractical.  In the DEC, discrete $k$-forms are encoded as weights on all combinations of $k$-neighbors of each data point.  For a data set with $N$ data points, each having $\ell$ neighbors, functions would be represented as $N\times 1$ vectors, $1$-forms as $N\times \ell$ matrices, and general $k$-forms as $N\times \ell^k$ matrices.  Thus, operators such as the $k$-Laplacian are represented as $N\ell^k \times N\ell^k$ matrices.  The SEC provides an alternative which is much more memory efficient.  

It has been shown that the error in the leading eigenfunctions of the $0$-Laplacian is proportional to the eigenvalues \cite{BerrySauer16}, which by Weyl's law grow according to $\lambda_j \propto n_j^{2/d}$ where $d$ is the dimension of the manifold.  Moreover, for larger eigenvalues and eigenfunctions the error ultimately becomes quadratic in the eigenvalue.  The idea of the SEC is to formulate the exterior calculus entirely in terms of the eigenfunctions of the $0$-Laplacian, approximated through graph-theoretic kernel methods, and to discretize the exterior calculus by projecting onto the first $M \ll N$ eigenfunctions.  Thus, functions would be represented as $M\times 1$ vectors, $1$-forms as $M \times J$ matrices, and $k$-forms as $M\times J^k$ matrices.  As we will explain in Sections \ref{vectorOverview} and \ref{secFrameDef}, $J$ is the number of eigenfunctions required to form an embedding of the manifold in $\mathbb{R}^J$.  Notice that highly redundant data sets may introduce extremely large values of $N$, but since $M$ is decoupled from $N$ this would not present a problem for the SEC. Also, for high-dimensional manifolds which require a large data set $N$ to obtain a small number $M$ of accurate eigenfunctions, the SEC could proceed using only these accurate eigenfunctions potentially yielding very efficient representations of higher-dimensional manifolds. Another advantageous aspect of SEC representations is that their memory cost is independent of the ambient data space dimension $n $ (which can be very large in real-world applications). In fact, the only parts of the SEC framework  with an $n$-dependent memory and computation cost are the initial graph-Laplacian construction and the spectral representation of the pushforward maps on vector fields, all of which depend linearly on $ n $. In contrast, the cost of building simplicial complexes and other auxiliary constructs required by DEC and FEC approaches can be very high in large ambient space dimensions. 

It is also desirable that a data-based exterior calculus should capture as much as possible of the structure of the exterior calculus from differential geometry, meaning that discrete analogs of continuous theorems should hold.  While no method captures discrete analogs of all the continuous theorems, each method has some partial results.  For example, the DEC beautifully captures a discrete analog of Stokes' theorem and the Leibniz rule holds exactly for closed forms, however the product rule for the Laplacian fails.  In the SEC, the product rule for the Laplacian will hold exactly, however this leads inevitably to the failure of the Leibniz rule as shown in \ref{productRules}.  

Finally, even though here we do not explicitly address this issue, an important consideration in data-driven techniques is robustness to noise. The simplicial complexes employed in DEC become increasingly sensitive to noise with increasing simplex dimension. On the other hand, the noise robustness of the SEC is limited by the noise robustness of the graph Laplacian algorithm employed to approximate the eigenfunctions of the 0-Laplacian. The latter problem has been studied from different perspectives in the literature \cite{ElKaroui10,MeyerShen14,YanSarkar16}, and it has been shown \cite{ElKaroui10} that for certain classes of kernels and i.i.d.\ noises (including Gaussian noise of arbitrary variance) the graph Laplacian computed from noisy data converges spectrally to noise-free case in the infinite-data limit.

The central challenge of the SEC approach is obtaining the representation of the exterior calculus in the spectral basis of eigenfunctions of the $0$-Laplacian.  In the next section we overview how vector fields, $k$-forms, and the central operators of the exterior calculus can all be represented spectrally. Since the gradients of these functions do not span the set of vector fields (otherwise every vector field would be a gradient field), we instead build a \emph{frame} (overcomplete set) \cite{DuffinSchaeffer52,Mallat09} consisting of products of Laplacian eigenfunctions and their gradient, and we represent vector fields in this frame. We proceed analogously for $ k $-forms, using products of Laplacian eigenfunctions and $ k $-fold wedge products of their differentials to construct frames. 

The plan of this paper is as follows. We begin in Section \ref{secOverview} with an overview of the SEC, including the fundamental idea of our approach and tables which overview key formulas.  Computation of the more complex formulas can be found in \ref{derivations}.  In Section \ref{secPrelim}, we briefly review the necessary background material and introduce our key definitions.  Our central contribution to the theory of the exterior calculus is proving that our construction yields frames for $L^2$ and Sobolev spaces of vector fields and $k$-forms in Section \ref{secSECSmooth}. In Section~\ref{secSECRep}, we discuss aspects of this frame representation for bounded vector fields, as well as associated representations of vector fields as operators on functions and the convergence properties of finite-rank approximations. Then, in Section \ref{secGalerkin}, we employ this frame representation to construct a Galerkin approximation scheme for the eigenvalue problem of the 1-Laplacian, which is shown to converge spectrally.   Section \ref{secDataDriven} establishes the consistency of the data-driven SEC representation of the exterior calculus.  In Section \ref{secNumerics}, we present numerical results demonstrating the consistency of the SEC on a suite of numerical examples involving orientable and non-orientable smooth manifolds, as well as the fractal attractor of the Lorenz~63 system. We conclude with a summary discussion and future perspectives in Section~\ref{secDiscussion}. A Matlab code reproducing the results in Section~\ref{secNumerics} is included as supplementary material. 

\section{\label{secOverview}Overview of the Spectral Exterior Calculus (SEC)}

As mentioned above, many manifold learning techniques are based on the ability to approximate the Laplacian operator on a manifold \eqref{0Laplacian} via a graph Laplacian \eqref{graphLaplacian}, defined on a graph of discrete data points sampled from the manifold.  When this convergence is spectral, we may use the eigendecomposition of an appropriately constructed graph Laplacian.  In Section \ref{dm} below we briefly summarize the Diffusion Maps approach to the construction \cite{CoifmanLafon06}.  The eigendecomposition from Diffusion Maps, or a comparable algorithm with spectral convergence guarantees (e.g., \cite{trillos2018variational}), is the only input required to generate the entire SEC construction.

The SEC represents vector fields and differential forms using Laplacian eigenfunctions, and then reformulates the exterior calculus of Riemannian geometry in terms of these representations.  This reformulation is described in Sections~\ref{functionOverview}--\ref{laplacianOverview}, and will be made rigorous in Sections \ref{secPrelim} and \ref{secSECSmooth}.  The motivation for this reformulation is that it allows us to define an exterior calculus using only the eigendecomposition of the 0-Laplacian.  In other words, in a manifold learning scenario, the eigendecomposition of the graph Laplacian provides all the necessary inputs to formulas which generate the entire exterior calculus formalism.  Of course, this implies a ``low-pass'' or truncated representation, and in Sections \ref{secGalerkin} and \ref{secDataDriven} we will prove that these truncated representations converge in the limit, as the number of eigenvectors increases.  

\subsection{The Diffusion Maps Algorithm for the Construction of the 0-Laplacian}\label{dm}

Following the Diffusion Maps approach, we define a kernel matrix 
\begin{displaymath}
    K_{ij} = k_\epsilon(x_i,x_j) :=  \exp\left(-\frac{\lVert x_i-x_j \rVert^2}{4\epsilon} \right),
\end{displaymath} 
where $\{x_i\}_{i=1}^N \subset \mathcal{M} \subset \mathbb{R}^n$ is a data set sampled from the embedded manifold $\mathcal{M}$, under a measure with a smooth, fully supported density relative to the volume measure associated with the Riemannian metric, $g$, induced by the embedding. We then normalize $\bm K = [ K_{ij} ] $ to a new kernel matrix $ \hat{ \bm K } $  to remove the sampling bias,

\[  \hat { \bm K } = \bm Q^{-1} \bm K \bm Q^{-1}, \quad \bm Q = \diag[ Q_{ii} ], \quad Q_{ii} = \sum_{j=1}^N K_{ij},   \]
and finally normalize $\hat{ \bm K }$ into a Markov matrix $ \bm P $,  
\begin{equation}
    \label{eqPDM}
    \bm P = \bm D^{-1}\hat{ \bm K }, \quad \bm D = \diag[ D_{ii} ], \quad D_{ii} = \sum_{j=1}^N \hat K_{ij},
\end{equation}
which approximates the heat semigroup $e^{-\epsilon \Delta_0}$; see \cite{CoifmanLafon06} for details on this procedure. The normalized kernel matrix has the generalized eigendecomposition
\[ \hat {\bm K} = \bm D \bm \Phi \bm \Lambda \bm \Phi^{\top}, \]
which can be computed by solving the fully symmetric generalized eigenvalue problem $\hat{ \bm K } \vec \phi = \Lambda \bm D \vec \phi$.  The eigenvalues satisfy $\Lambda = e^{-\epsilon\hat\lambda}$ where the values $\hat \lambda \geq 0 $ approximate the eigenvalues of the Laplacian, so we define $\hat \lambda = -\log(\Lambda)/\epsilon$  An asymptotically equivalent (in the limit $N\to \infty$ and $\epsilon \to 0$) approach is to form the graph Laplacian $\bm L = \bm D - \hat{ \bm K } ) $, and compute the generalized eigendecomposition $\bm L = \bm D \bm \Phi \tilde{ \bm\Lambda}\bm \Phi^\top$.  In either approach, we approximate the eigenfunctions $\Delta_0 \phi_i = \lambda_i \phi_i$ of the Laplacian operator, and sort the columns of $\bm\Phi$ so that the eigenvalues are increasing.  The diagonal matrix $\bm D$ represents the Riemannian $L^2$ inner product on the manifold, in the sense that if $\vec f_i = f(x_i)$ and $\vec h_i = h(x_i)$ are vector representations of (complex-valued) continuous functions, then
\[ {\vec f}^\dag \bm D \vec h \approx \int_{\mathcal{M}} f^*(x)h(x) \, d\mu(x) = \langle f,g\rangle _{L^2} \]
up to a constant proportionality factor, where $ \dag $ denotes complex-conjugate transpose, and $ \mu $ is the Riemannian measure of the manifold.  Thus, we can compute the generalized Fourier transform of the function $f$ by 
\[ \hat f = \bm \Phi^\top \bm D \vec f,  \qquad \hat f_j = \sum_{i=1}^N \phi_j(x_i)D_{ii}f(x_i) \approx \langle \phi_j,f\rangle _{L^2}. \]  
We can then reconstruct the values of the function $f$ on the data set by $\vec f = \bm \Phi \hat f$, which holds exactly since $\bm \Phi \bm \Phi^\top = \bm D^{-1}$. If a smaller number of eigenvectors are used, then $\bm \Phi \bm \Phi^\top$ is not full rank, and the result is a low-pass filter.

Note that in applications (including those presented in this paper), one is frequently interested in real-valued functions and self-adjoint operators, so complex conjugation is not included in $ L^2$ inner products as above. However, applications with complex-valued functions can also be of interest (e.g., in dynamical systems theory \cite{Giannakis19}), so in what follows we work with complex-valued functions to maintain generality. 
\subsection{Functions, multiplication, and the Riemannian metric}\label{functionOverview}

In Table~\ref{fig2}, we show the basic elements of the exterior calculus and their SEC formulations.  For example, complex-valued functions are represented in the SEC by their generalized Fourier transform $\hat f_i = \langle \phi_i,f\rangle _{L^2}$, which is justified since the Hodge theorem shows that the eigenfunctions $\phi_i$ form a smooth orthonormal basis for square-integrable functions on the manifold. It should be noted that, as with all $ L^2 $ expansions, $ f $ may differ from the reconstructed function $ \sum_i \hat f_i \phi_i $ on sets of measure zero. Similarly, all frame representations of vector fields and $ k $-forms in Table~\ref{fig2} and the ensuing discussion should be interpreted in an $ L^2 $ sense. 

\begin{table}
    \caption{\label{fig2} The SEC reformulation of the basic elements of the exterior calculus from Riemannian geometry.}
    \small
\begin{center}
{\renewcommand{\arraystretch}{1.5}
\begin{tabular*}{\linewidth}{@{\extracolsep{\fill}}lcc}
\hline
Object & Symbolic & Spectral  \\
\hline\hline
Function 				&	$f$ 					& $\hat f_i = \langle \phi_i,f\rangle _{L^2}$ \\
\hline
Laplacian 				&	$\Delta f$ 	& $\langle \phi_k, \Delta f\rangle _{L^2} = \lambda_k \hat f_k$ \\
\hline 
$L^2$ Inner Product & $\langle f,h\rangle _{L^2}$ & $\sum_i \hat f_i^*  \hat h_i $  \\
\hline
Dirichlet Energy 							& $\langle f,\Delta f\rangle _{L^2} = \int_\mathcal{M} \lVert \grad f \rVert^2 \,d\mu$ & $\sum_i \lambda_i \lvert \hat f_i\rvert^2 $  \\
\hline
Multiplication  	& $\phi_i \phi_j$ & $c_{ijk} = \langle \phi_i \phi_j, \phi_k\rangle _{L^2}$  \\
\hline
Function Product 	& $f h$ 				& $\sum_{ij} c_{kij}\hat f_i \hat h_j$ 	 \\
\hline
Riemannian Metric			& $\grad \phi_i \cdot \grad \phi_j$ & 
\begin{minipage}{.3\linewidth}
    \begin{displaymath}
        \begin{aligned}
        g_{kij}  & \equiv \langle  \grad \phi_i \cdot \grad \phi_j,\phi_k\rangle _{L^2} \\ 
        & = (\lambda_i + \lambda_j - \lambda_k)c_{kij} / 2
    \end{aligned}
    \end{displaymath}
\end{minipage} \\
\hline
Gradient Field			& $\grad f(h)=\grad f^* \cdot \grad h$ & $\langle \phi_k, \grad f( h ) \rangle _{L^2}= \sum_{ij} g_{kij}\hat f_i \hat h_j$ \\
\hline 
Exterior Derivative 		& $df(\grad h)=df^*\cdot dh$ & $\sum_{ij} g_{kij} \hat f_i \hat h_j$    \\
\hline
Vector Field	(basis)		&$v(f) = v^* \cdot \grad f$		& $\sum_j v_{ij}\hat f_j$	 \\
\hline
Divergence 		& $\divr v$ 	& $\langle \phi_i, \divr v\rangle _{L^2} = -v_{0i}$  \\
\hline
Frame Elements & $b_{ij}(\phi_l) = \phi_i \grad \phi_j(\phi_l)$ &
\begin{minipage}{.3\linewidth}
    \begin{displaymath}
    \begin{aligned}
        G_{ijkl} & \equiv \langle b_{ij}(\phi_l),\phi_k\rangle _{L^2} \\
        &=\mbox{$\sum_{m}$}c_{mik}g_{mjl}  
    \end{aligned}
    \end{displaymath}
\end{minipage}\\
\hline
Vector Field (frame) & $v(f) = \sum_{ij}v^{ij}b_{ij}(f)$ & $\langle \phi_k,v(f)\rangle _{L^2} = \sum_{ijl}G_{ijkl}v^{ij}\hat f_l$ \\
\hline
Frame Elements & $b^{ij}(v)=b^i\, db^j(v)$ & $ \langle \phi_k,b^{ij}(v)\rangle _{L^2} =\sum_{nlm}c_{kmi}G_{nlmj}v^{nl}$  \\
\hline
1-Forms (frame) & $\omega = \sum_{ij}\omega_{ij}b^{ij}$ & $\langle \phi_k,\omega(v)\rangle _{L^2} =\sum_{ij}\omega_{ij}\langle \phi_k,b^{ij}(v)\rangle _{L^2}$  \\
\hline
\end{tabular*}
}
\end{center}
\end{table}

The two key elements of Table~\ref{fig2} are the representation of function multiplication and the Riemannian metric. 

First, function multiplication will be represented by the fully symmetric three-index tensor 
\begin{equation}
    \label{multiplication} 
    c_{ijk} = \langle \phi_i \phi_j,\phi_k\rangle _{L^2}, 
\end{equation}
which will be a key building block of the SEC. Note that here we use the term ``tensor'' to represent a general multi-index object such as $ c_{ijk} $ derived from inner products of Laplacian eigenfunctions. While these objects are not geometrical tensors on the manifold, they nevertheless transform via familiar tensor laws under changes of $ L^2 $ basis preserving the Laplacian eigenspaces.  

Next, the Riemannian metric is represented based on the product formula \eqref{productrule}, and is given by \eqref{SECmetric}. The power of the SEC is that we will only need to represent the metric for gradients of (real) eigenfunctions $\grad \phi_i$ and $\grad \phi_j$, where we find that
\begin{equation}
    \label{eqGPointwise}
    g(\grad \phi_i,\grad \phi_j) = \frac{1}{2}(\phi_i \Delta \phi_j + \phi_j \Delta \phi_i - \Delta(\phi_i \phi_j)) = \frac{1}{2}((\lambda_i+\lambda_j)\phi_i\phi_j - \Delta(\phi_i \phi_j)).
\end{equation}
We can further reduce this by writing the product $\phi_i \phi_j = \sum_{k} c_{ijk} \phi_k$, so that
    \[ g(\grad \phi_i,\grad \phi_j) = \frac{1}{2} \sum_k (\lambda_i+\lambda_j - \lambda_k)c_{ijk}\phi_k, \]
meaning that the $k$-th Fourier coefficient of the Riemannian metric is 
\begin{equation}
    \label{eqGFourier}
    g_{kij} \equiv \langle g(\grad \phi_i,\grad \phi_j),\phi_k\rangle _{L^2} = \frac{1}{2}(\lambda_i +\lambda_j - \lambda_k)c_{ijk}.
\end{equation}
Notice that $g_{kij}$ is symmetric in $i$ and $j$ but not in $k$.  These first two simple formulas are the key to the SEC.

\subsection{Vector fields}\label{vectorOverview}

We will need two different ways of representing vector fields.  The first method is called the operator representation and is based on the interpretation of a vector field as a map from smooth functions to smooth functions, defined by 
\[ v(f) = v^* \cdot \grad f = g(v^*,\grad f), \]
where $v^* $ denotes the complex-conjugate vector field to $v$. Note that, as with functions, in SEC we consider vector fields, differential forms, and other tensors to be complex. This is because, ultimately, we will be concerned with spectral approximation of operators on these objects, and the complex formulation will allow us to take advantage of the full range of spectral approximation techniques for operators on Hilbert spaces over the complex numbers. Throughout, our convention will be that Riemannian inner products on complex tensors are conjugate-symmetric in their first argument, e.g., $g(f v, w ) = f^* g( v, w )$ for vector fields $v,w $ and function $f$. Since we have a smooth basis $\{\phi_i\}$ for functions, we can represent any vector field $v$ in this basis by an operator with matrix elements
\[ v_{ij} = \langle \phi_i,v(\phi_j) \rangle _{L^2} = \langle \phi_i,v^*\cdot \grad \phi_j\rangle _{L^2} = \langle \phi_i \grad \phi_j, v \rangle _{L^2_\mathfrak{X}} = \langle  b_{ij}, v \rangle _{L^2_\mathfrak{X}}, \]
where the first two inner products appearing above are the $L^2$ inner products on functions, the last two inner products are the Hodge inner products induced on vector fields, 
\[ \langle v,w\rangle _{L^2_{\mathfrak X}} = \int_{\mathcal{M}} g(v,w) \,d\mu,  \] 
and $ b_{ij} = \phi_i \grad \phi_j $ are smooth vector fields. Note that the Hodge inner product defines the space of square-integrable vector fields, denoted $L^2_{\mathfrak X}$.

The second method of representing a vector field will be as a linear combination of the vector fields $b_{ij} $ just introduced, with coefficients $v^{ij}$ so that
\[ v = \sum_{ij} v^{ij}b_{ij}. \]
As we will show in Section \ref{secSECSmooth}, the vector fields $\{b_{ij}\}$ where $i=1, \ldots,\infty$ and $j=1,\ldots,J < \infty $ spans $L^2_{\mathfrak{X}}$.  However, instead of a basis, this set is only a \emph{frame} for this space, i.e., a spanning set satisfying appropriate upper and lower bounds  for the $\ell^2$ norms of the sequence $ \langle b_{ij}, v \rangle_{L^2_\mathfrak{X}}$ for every $v \in L^2_\mathfrak{X}$ \cite{christensen2008frames}.  Since $\{b_{ij}\}$ is not a basis, this representation will generally not be unique, although frame theory ensures that there is a unique choice of coefficients $v^{ij}$ which minimizes the $\ell^2$ norm \cite[Lemma 5.3.6]{christensen2008frames}.

As we will see in the Section \ref{laplacianOverview}, there is a natural choice of basis for $L^2_{\mathfrak X}$, and constructing this basis will be a central goal of the SEC approach, however doing so requires using the frame $\{b_{ij}\}$.  To motivate this choice of frame elements, note that given a fixed point $x$ on the manifold and a sufficient (finite) number of eigenfunctions $ \phi_j$, the gradients of these eigenfunctions $\grad \phi_j(x)$ will span the tangent space $T_x\mathcal{M}$ (see Section \ref{secPrelim} for details).  In fact, for a $d$-dimensional compact manifold, we should be able to find $d$ eigenfunctions whose gradients form a basis for $T_x\mathcal{M}$ for a fixed $x$.  However, in general any choice of $d$ eigenfunctions will not span $T_x\mathcal{M}$ for every $x$ simultaneously, meaning that the choice of eigenfunctions which span depends on $x$.  This is easily demonstrated by the example of the sphere $S^2$. On $ S^2 $, every smooth vector field vanishes at some point $x$, so at that point the collection of $d$ gradient fields will at most span a $(d-1)$-dimensional subspace of $T_xS^2$.  Intuitively, given a collection of sufficiently many gradients of eigenfunctions $\{\grad \phi_j\}_{j=1}^J$, and if the manifold is not too ``large'' (i.e., it is compact), we can span all the tangent spaces simultaneously with $J<\infty$, but of course we no longer have a basis.  Given an arbitrary smooth vector field $v$, we can then represent $v$ at each point $x \in \mathcal{M}$ as a linear combination of gradients of eigenfunctions,  
\[ v_x = \sum_{j=1}^J c_{v,j}(x) \grad \phi_j(x). \]
If we can choose the coefficients $c_{v,j}(x)$ in this linear combination to be smooth functions on the manifold, then these functions can be represented in the basis $\{\phi_i\}$ of eigenfunctions, so that
\[ c_{v,j}(x) = \sum_{i=0}^{\infty} \langle  \phi_i, c_{v,j} \rangle _{L^2}\phi_i(x), \]
which means that we can represent the vector field $v$ as
\[ v = \sum_{j=1}^J \sum_{i=0}^{\infty} \langle \phi_i, c_{v,j} \rangle _{L^2}\phi_i \grad \phi_j. \]

We now consider how to move between the operator representation and frame representation of a vector field.  Substituting the frame representation into the operator representation, we find that
\begin{equation}
    \label{eqVOpRep}
    v_{ij} = \langle b_{ij}, v\rangle _{L^2_{\mathfrak{X}}} = \sum_{kl} v^{kl} \langle b_{kl},b_{ij}\rangle _{L^2_\mathfrak{X}} = \sum_{kl} G_{ijkl}v^{kl}, 
\end{equation}
where $G_{ijkl} = \langle b_{ij},b_{kl}\rangle _{L^2_{\mathfrak{X}}}$ is the Grammian matrix of the frame elements with respect to the Hodge inner product.  Thus, we see that the Hodge Grammian is the linear transformation which maps from the frame representation $v^{kl}$ to the matrix representation $v_{ij}$. Crucially, the quantities $G_{ijkl}$ can be computed in closed form from the spectral representation of the pointwise inner products in~\eqref{eqGPointwise}, viz.
    \begin{equation}
        \label{eqGIJKL_Hodge}
        G_{ijkl} = \langle b_{ij}, b_{kl} \rangle_{L^2_{\mathfrak{X}}} = \langle \phi_i \phi_k, \grad \phi_j \cdot \grad \phi_l \rangle_{L^2} = \frac{1}{2} \sum_{m=0}^\infty c_{ikm} c_{jlm} (\lambda_j + \lambda_l - \lambda_m).
    \end{equation}
Since the frame is overcomplete, the matrix $\bm G$ is necessarily rank deficient and thus there is no unique inverse transformation.  However, if we also specify the minimum $\ell^2$ norm then we can map from the matrix representation $v_{ij}$ to the frame coefficients (with minimum norm) $v^{kl}$ via the pseudo-inverse $\bm G^+$ of the Hodge Grammian.

\subsection{Differential forms}\label{formsOverview}

In order to build a formulation of the exterior calculus we need to first move from vector fields to differential $k$-forms.  First, $0$-forms are equivalent to $C^{\infty}$ functions defined on the manifold, which we represent in the basis of eigenfunctions of the Laplacian $\phi_i$.  Since each eigenfunction can also be thought of as a $0$-form, we will sometimes also denote the eigenfunctions by
\begin{displaymath} 
    b^i = \phi_i, 
\end{displaymath}
since the superscript notation is oftentimes used for basis elements of spaces of differential forms. Our primary focus in this paper will be $1$-forms, which are are duals to vector fields. That is, a 1-form takes in a vector field as its argument and returns a function. On a Riemannian manifold, we can move back and forth between vector fields and 1-forms with the $\sharp$ (sharp) and $\flat$ (flat) operators. Locally, these operators map 1-forms and vector fields, respectively, to their Riesz representatives with respect to the Riemannian inner product. In particular, if $\omega$ is a 1-form and $v$ is a vector field, then
\[ \omega(v) = g^{-1}(\omega^*,v^{\flat}) = g(v^*,\omega^\sharp), \] 
where $g^{-1}$ is the ``inverse'' metric on dual vectors. A fundamental operator on differential forms is the exterior derivative, $d$, which maps $k$-forms to $(k+1)$-forms, so that the exterior derivative of a 0-form $f$ is defined by the 1-form $df$, acting on a vector field $v$ by 
\[ df(v) = v(f) = g(v^*,\grad f). \]
We will sometimes use the notation $d_k$ to explicitly exhibit the order of differential forms on which a given exterior derivative acts. 

Since 1-forms are dual to vector fields, we will use a similar frame representation to that in Section~\ref{vectorOverview}, based on the eigenfunctions, $\{b^{ij} = b^i \, db^j\}$.  As we will show in Section \ref{secSECSmooth}, these 1-forms span the space $L^2_1 $ of square-integrable 1-forms.  We also note that the Riemannian metric lifts to $k$-forms (see Section \ref{secPrelim} for details), and takes two $k$-forms and returns a function. Integrating the Riemannian inner product of two $k$-forms, 
\[ \langle \psi,\omega\rangle _{L^2_k}  = \int_{\mathcal{M}} g^{-1}(\psi,\omega) \,d\mu, \]
defines the Hodge inner product, which then defines the Hilbert space of square integrable $k$-forms. Finally, since $df^\sharp = \grad f^*$, $\grad f^\flat = df^*$, and the $b_{ij}$ are real,  we  have $b_{ij}^\flat = b^{ij}$ and $(b^{ij})^\sharp=b_{ij}$, so the coefficients of a vector field in the frame representation can also be used to represent the corresponding 1-form and vice versa.
\subsection{The Laplacian on forms} \label{laplacianOverview}

The Laplacian on $k$-forms is defined via the exterior derivative  $ d $ and its Hodge dual, the codifferential $\delta$, by 
\[ \Delta_k = d_{k-1}\delta_k + \delta_{k+1} d_k \]
In order to represent the eigenvalue problem for the operator $\Delta_1$ in the frame $\{b^{ij}\}$, we need to compute the inner products
\begin{equation}
    \label{eqGE}
    G_{ijkl} = \langle b^{ij},b^{kl}\rangle _{L^2_1}, \qquad  E_{ijkl} = \langle b^{ij},\Delta_1 b^{kl}\rangle _{L^2_1} = \langle d b^{ij}, d b^{kl} \rangle_{L^2_2} + \langle \delta b^{ij}, \delta b^{kl} \rangle_{L^2}, 
\end{equation}
representing the Gramm matrix of Hodge inner products (which we will call the Hodge Grammian) and Dirichlet form matrix, respectively.  We derive the expressions for these tensors in \ref{derivations}  below, and the formulas are summarized in Table~\ref{fig3}.  Note that both $\bm G$ and $ \bm E$ can be written as symmetric matrices by numbering the frame elements.  Moreover, we can easily represent the Gramm matrix with respect to the Sobolev $H^1$ inner product on 1-forms, 
 \begin{displaymath}
     \langle \psi, \omega \rangle_{H^1_1} = \langle \psi, \omega \rangle_{L^2_1} + E_{1,1}( \psi, \omega ), \quad E_{1,1}( \psi, \omega ) = \langle d\omega, d\nu \rangle_{L^2_2} + \langle \delta \omega, \delta \nu \rangle_{L^2},
 \end{displaymath}
 as
 \[ \bm G^1 = \bm G + \bm E. \]
 The importance of the Sobolev Grammian $ \bm G^1 $ is that $H^1_1$ is a natural domain for weak (variational) formulations of the eigenvalue problem of the Laplacian on 1-forms.  
 
\begin{table}
    \caption{\label{fig3} The SEC representation of the Laplacian on 1-forms and the Dirichlet and Sobolev energy forms. Derivations can be found in \ref{derivations}. Pairs of integers in parentheses indicate symmetries under permutations of tensor indices; e.g., $(1,2) $ in $ c^n_{ijkl} $ indicates that $ c^n_{ijkl} = c^n_{jikl}$.}
    \small
\begin{center}
{\renewcommand{\arraystretch}{1.5}
\begin{tabular*}{\linewidth}{@{\extracolsep{\fill}}lcc}
\hline
Operator & Tensor	& Symmetries \\
\hline\hline
Quadruple Product		& 
\begin{minipage}{.4\linewidth}
    \begin{displaymath}
        \begin{aligned}
            c_{ijkl}^0 &= \langle \phi_i \phi_j,\phi_k \phi_l \rangle _{L^2} \\
            &= \mbox{$\sum_s$} c_{ijs}c_{skl}
        \end{aligned}
    \end{displaymath}
\end{minipage}
& Fully symmetric \\
\hline
Product Energy	& 
\begin{minipage}{.4\linewidth}
    \begin{displaymath}
        \begin{aligned}
            c_{ijkl}^{p} &= \langle \Delta^p(\phi_i \phi_j),\phi_k\phi_l \rangle _{L^2} \\
            &= \mbox{$\sum_s$} \lambda_s^p c_{ijs}c_{skl}
        \end{aligned}
    \end{displaymath}
\end{minipage}
& \begin{minipage}{.2\linewidth}
    \centering
    (1,2), (3,4), \\
    (1,3), (2,4) 
\end{minipage} \\
\hline
Hodge Grammian  & 
\begin{minipage}{.4\linewidth}
    \begin{displaymath}
        \begin{aligned}
            G_{ijkl} &= \langle b^{ij},b^{kl}\rangle_{L^2_1} \\
            &= [(\lambda_j + \lambda_l)c^0_{ijkl} - c_{ijkl}^1]/2 
        \end{aligned}
    \end{displaymath}
\end{minipage}
& (1,3), (2,4) \\
\hline
Antisymmetric & 
\begin{minipage}{.4\linewidth}
    \begin{displaymath}
        \begin{aligned}
            \hat G_{ijkl} & = \langle\hat b^{ij},\hat b^{kl}\rangle_{L^2_1} \\
            &= G_{ijkl}+G_{jilk}-G_{jikl}-G_{ijlk}
        \end{aligned}
    \end{displaymath}
\end{minipage}
& (1,3), (2,4) \\
\hline
Dirichlet Energy 
& 
\begin{minipage}{.4\linewidth}
    \begin{displaymath}
        \begin{aligned}
            E_{ijkl} &= E_{ijkl}=\langle b^{ij},\Delta_1(b^{kl})\rangle _{L^2_1} \\
            &= [(\lambda_i+\lambda_j+\lambda_k+\lambda_l)(c_{iljk}^1-c_{ikjl}^1) \\
            & \quad +(\lambda_j+\lambda_l-\lambda_i-\lambda_k)c^1_{ijkl} \\
            & \quad + (c^2_{ijkl}+c^2_{ikjl}-c^2_{iljk}) ]/4
        \end{aligned}
    \end{displaymath}
\end{minipage}
& (1,3), (2,4) \\
\hline
Antisymmetric & 
\begin{minipage}{.4\linewidth}
    \begin{displaymath}
        \begin{aligned}
            \hat E_{ijkl} & = \langle \hat b^{ij},\Delta_1 \hat b^{kl}\rangle_{L^2_1} \\
            &=(\lambda_i+\lambda_j+\lambda_k+\lambda_l)(c_{iljk}^{1} - c_{ikjl}^{1}) \\         
            & + (c_{ikjl}^{2}- c_{iljk}^{2})
         \end{aligned}
   \end{displaymath}
\end{minipage}
& (1,3), (2,4) \\
\hline
Sobolev $H^1$ Grammian & 
\begin{minipage}{.4\linewidth}
    \begin{displaymath}
        \begin{aligned}
            G^1_{ijkl} &= E_{ijkl} + G_{ijkl} \\
            \hat G^1_{ijkl} &=\hat E_{ijkl} + \hat G_{ijkl} 
        \end{aligned}
    \end{displaymath}
\end{minipage}
& (1,3), (2,4) \\
\hline
\end{tabular*}
}
\end{center}
\end{table}

As we will show in Section \ref{secGalerkin}, the key to solving the eigenproblem of the Laplacian on 1-forms is to first express the problem in a weak sense, i.e., replace $ \Delta_1 \varphi = \nu \varphi $ by
\begin{displaymath}
    E_{1,1}( \psi, \varphi ) = \nu \langle \psi, \varphi \rangle_{L^2_1}, \quad \forall \psi \in H^1_1,
\end{displaymath}
which is equivalent to the minimization problem
\[  \nu = \min_{\varphi \in H^1_1 \setminus \{ 0 \}} \left\{\frac{E_{1,1}( \varphi, \varphi )}{\langle \varphi,\varphi\rangle _{L^2_1}} \right\}. \]
Intuitively, the ratio $E_{1,1}(\varphi,\varphi) / \langle \varphi, \varphi \rangle_{H^1_1}$ is a measure of ``roughness'', or oscillatory behavior, of a given eigenform $ \varphi $, much like the eigenvalues of the 0-Laplacian measure the roughness of the corresponding eigenfunctions. Thus, ordering eigenforms in order of increasing eigenvalue, as we will always do by convention, is tantamount to ordering them in order of increasing complexity that they exhibit on the manifold. As with functions, given finite amounts of data, the approximation error for eigenforms increases with the corresponding eigenvalue.

In SEC, we represent the eigenform in the frame, $\varphi =\sum_{ij} \varphi_{ij} b^{ij}$. The above variational problem can then in principle be written in matrix form as 
\[ E_{1,1}( \psi, \varphi ) = \nu \langle \psi, \varphi \rangle_{L^2_1}, \quad \forall \psi \in H^1_1, \qquad \implies \qquad \bm E\vec \varphi = \nu \bm G \vec \varphi. \]
However,  the above eigenvector problem is not well-conditioned because $\bm G$ is not full-rank in general (since the frame is overcomplete, meaning there can be multiple representations of the same 1-form).  In order to find an appropriate basis, we first diagonalize the Sobolev Grammian,
\[ \bm G^1 = \bm U \bm H \bm U^\top, \qquad \bm H = \diag [ h_{ii} ], \]
and select the columns $\tilde{ \bm U }$ of the orthogonal matrix $\bm U$ corresponding to the largest eigenvalues.  For example, in our implementation we choose $h_{ii}>h_{11} \times 10^{-3}$.  Notice that the columns of $\tilde{\bm U}$ contain the frame coefficients of unique orthogonal 1-forms.  In other words, the matrix $\tilde{\bm U }$ is a choice of basis for $H^1_1$ represented in the frame.  Thus, we can project the eigenvalue problem onto this basis by writing
\begin{equation}
    \label{eqGEVIntro}
    \bm L \vec a  = \nu \bm B \vec a, \qquad \bm L = \tilde{ \bm U }^\top \bm E\tilde{ \bm U }, \qquad  \bm B = \tilde{\bm U }^\top \bm G \tilde {\bm U }. 
\end{equation}
An eigenvector $ \vec a $ of this generalized eigenvalue problem contains the coefficients of a frame representation for an eigenform $\varphi$ of $\Delta_1$.  

We should note that in practice we found somewhat better results using the antisymmetric elements  $\hat b^{ij} = b^i\,db^j - b^j\,db^i$, likely due to the fact that these forms are less redundant.  All of the formulas for the antisymmetric formulation of the $1$-Laplacian are given in Table~\ref{fig3}.  The one change in the antisymmetric formulation is that in order to move from the frame representation to the operator representation, we need the additional tensor
\begin{align*}
 H_{ijkl} &= \langle \phi_k,\hat b^{ij}(\grad \phi_l)\rangle _{L^2} = \langle \phi_k,b^i \,db^j(\grad \phi_l)\rangle _{L^2}-\langle \phi_k,b^j\,db^i(\grad \phi_l)\rangle _{L^2}\\
 &= G_{ijkl} - G_{jikl}.
 \end{align*}
With this tensor, given the frame representation of a $1$-form $\varphi = \sum_{ij} \hat \varphi_{ij}\hat b^{ij}$, the operator representation of the corresponding vector field $v = \varphi^{\sharp}$ becomes $v_{kl} = \sum_{ij}H_{ijkl}\varphi_{ij}$. 

In order to visualize an eigenform $\varphi$, we will visualize the corresponding vector field, $v = \varphi^{\sharp}$, which has the same frame coefficients as shown in Section~\ref{formsOverview}. In particular, it follows from~\eqref{eqVOpRep} that simply multiplying  $\vec a$ by the matrix $\bm G$, leads to $\bm G\vec a$, which contains the operator representation of $v$.  By reshaping $\bm G\vec a$ into a $M\times M$ matrix $\bm V$, we have $V_{ij} = \langle \phi_i,v(\phi_j)\rangle _{L^2}$.  To visualize this vector field, we need to map it back into the original data coordinates. This ``pushforward'' operation on vector fields can also be represented spectrally \cite[Proposition~6]{Giannakis19}.  In particular, let $\bm X$ be the $n \times N$ matrix of $N$ data points in $\mathbb{R}^n$.  We first compute the Fourier transform of these coordinates by computing the $\bm D$ inner product with the $N \times M$ matrix $\bm \Phi$ (see Section~\ref{dm}).  Thus, $\hat{ \bm X } = \bm X \bm D \bm \Phi$ is the $n \times M$ matrix containing the $M$ Fourier coefficients of each of the $n$ coordinate functions.  We can now apply the vector field to each of these functions by multiplying $\hat{ \bm X}\bm V^\top$, which now contains the Fourier coefficients of the pushforward of the coordinate functions.  Finally, we can reconstruct the coordinates of the arrows by computing the inverse Fourier transform $\tilde{ \bm V } = \hat { \bm X } \bm V^\top \bm \Phi^\top$, which is a $n\times N$ matrix containing the $n$-dimensional vectors which can plotted at each data point.  This method is used to visualize the SEC eigenforms in Section~\ref{secNumerics}.

\section{\label{secPrelim}Hilbert spaces and operators in the exterior calculus}

Consider a closed (compact and without boundary), smooth, orientable, $d$-dimensional manifold $ \mathcal{M} $, equipped with a smooth Riemannian metric $ g $.  Without loss of generality, we assume that $ g $ is normalized such that its associated Riemannian measure, $ \mu $, satisfies $ \mu( \mathcal{M} ) =1 $. As stated in Section~2, we will work with vectors in the complexified tangent spaces, $T_x^{\mathbb{C}} \mathcal{M} = T_x \mathcal{M} \otimes_{\mathbb{R}} \mathbb{C}$, $x \in \mathcal{M}$, treating by convention $g(\cdot, \cdot)$ as conjugate-linear in its first argument. We denote the associated metric tensor on dual vectors by $ \eta = g^{-1} $, and use the notation $ v^\flat = g( v, \cdot ) $, $ v \in T^{\mathbb{C}}_x \mathcal{M} $, and $ \alpha^\sharp = \eta( \alpha, \cdot ) $, $ \alpha \in T_x^{\mathbb{C}*} \mathcal{M} $, for the Riemannian duals of tangent vectors and dual vectors. In what follows, we introduce the spaces of functions, vector fields, and differential forms that will be employed in the SEC framework.  

\subsection{\label{secSpacesFunctions}Function spaces}

Let $\Delta : C^\infty(\mathcal{M}) \to C^\infty(\mathcal{M}) $ be the (positive-semidefinite) Laplace-Beltrami operator on smooth, complex-valued functions associated with the Riemannian metric $ g $. It is a fundamental result in analysis on closed Riemannian manifolds (e.g., \cite{Strichartz83,AbrahamEtAl88,Rosenberg97} that $ \Delta $ extends to a unique self-adjoint operator $ \bar \Delta : D( \bar \Delta ) \to L^2( \mathcal{M}, \mu ) $ with a dense domain $ D( \bar \Delta ) \subset L^2( \mathcal{M}, \mu ) $ in the $ L^2 $ space associated with the Riemannian measure, and a pure point spectrum of eigenvalues $ 0 = \lambda_0 < \lambda_1 \leq \lambda_2 \leq \cdots $ with no accumulation points, corresponding to a smooth orthonormal basis $ \{ \phi_j \}_{j=0}^\infty $ of eigenfunctions. By smoothness of the Laplace-Beltrami eigenfunctions, the products $ \phi_i \phi_j $ lie in $ L^2( \mathcal{M}, \mu )$; thus, we have 
\begin{equation}
  \label{eqCIJK}
  \phi_i \phi_j = \sum_{k=0}^\infty c_{ijk} \phi_k, \quad c_{ijk} = \langle \phi_k, \phi_i \phi_j \rangle_{L^2(\mathcal{M},\mu)}, 
\end{equation}
where the limit in the first equation is taken in the $ L^2 $ sense. As discussed in Section~\ref{secOverview}, our objective is to build a framework for tensor calculus on $ \mathcal{M} $ that is defined entirely through the spectral properties of the Laplacian on functions, encoded in the eigenvalues $ \lambda_j $, the corresponding eigenfunctions $ \phi_j $, and the coefficients $ c_{ijk} $ representing the algebraic relationships between the eigenfunctions. 

We use the notation $ L^p $, $ 1 \leq p \leq \infty $, to represent the standard Banach spaces of complex-valued functions on $ \mathcal{M} $ associated with the Riemannian measure $ \mu $, equipped with the standard norms, $ \lVert \cdot \rVert_{L^p} $. In the case $ p = 2 $, we use the shorthand notation $ H = L^2 $, and denote the corresponding Hilbert space inner product $ \langle \cdot, \cdot \rangle_H $, which we take to be conjugate-linear on its first argument. We also consider Sobolev spaces of higher regularity, defined for $ p \geq 0 $ by
\begin{equation}
  \label{eqHP}
  H^p = \left \{ \sum_{j=0}^\infty c_j \phi_j \in L^2 : \sum_{j=0}^\infty \lambda_j^p \lvert c_j \rvert^2 < \infty \right \}.
\end{equation}
These spaces are closed with respect to the norms $ \lVert f \rVert_{H^p} = \langle f, f \rangle_{H^p}^{1/2} $ associated with the inner products
\begin{equation}
  \label{eqHPInnerProd}
  \langle f, h \rangle_{H^p} = \sum_{q=0}^p \sum_{j=0}^{\infty}\lambda_j^q \hat f_j^* \hat h_j, \quad \hat f_j = \langle \phi_j, f \rangle_{L^2}, \quad \hat h_j = \langle \phi_j, h \rangle_{L^2}.
\end{equation}
Among these, the space $H^2$ is precisely the domain of the self-adjoint Laplacian $\bar \Delta$.

We equip each $ H^p $ space with a Dirichlet form $E_p : H^p \times H^p \to \mathbb{C}$, defined as the bounded sesquilinear form 
\begin{equation}
  \label{eqEP}
  E_p( f, h ) = \sum_{j=0}^\infty \lambda_j^p \hat f^*_j \hat h_j,
\end{equation}
with $ f $ and $h$ as in~\eqref{eqHPInnerProd}. This form induces a positive-semidefinite  Dirichlet energy functional $E_p(f)=E_p(f,f)$. Given $ f \in H^p \cap H $, the quantity $ E_p(f)/\lVert f \rVert_{H}^2$ can be thought of as a measure of roughness of $f$. If $f$ and $h $ are smooth, $E_p( f, h ) $ can be expressed in terms of the Laplace-Beltrami operator as $E_p(f,h) = \langle f, \Delta^p h \rangle_H$. Evidently, $ E_p( f ) = \langle f, \Delta^p f \rangle_{H} $ can  be arbitrarily large for highly oscillatory functions. 

In general, the $\{ \phi_j \} $ orthonormal basis of $H $ is not a Riesz  basis of $ H^p $, $ p\geq 1$; that is, it is not the case that given any $\ell^2$ sequence of expansion coefficients $c_j $, the vectors $f_l = \sum_{j=0}^{l-1} c_j \phi_j $ converge as $l\to \infty $ in $H^p$ norm. This issue is manifested from the fact that the Dirichlet energies $E_p(\phi_j) = \lambda_j^p$ of the basis elements are unbounded in $j$, making $ \{ \phi_j \} $ a poorly conditioned basis of $H^p $ for numerical calculations. On the other hand, the normalized  eigenfunctions $\phi^{(p)}_j $, defined by 
\begin{equation}
    \phi^{(p)}_j = \frac{ \phi_j }{ \lVert \phi_j \rVert_{H^p} },  
  \label{eqPhiHP}
\end{equation}
where $\lVert \phi^{(p)}_j\rVert_{H^p} = 1 $ by construction, form orthonormal bases of the respective $H^p$ spaces. 

A related, but stronger, notion of regularity of functions on $\mathcal{M}$ to that associated with the $H^p$ Sobolev spaces is provided by reproducing kernel Hilbert spaces (RKHSs) associated with the heat kernel on $(\mathcal{M},g)$. In particular, let  $\mathcal{H}$ be the RKHS of complex-valued functions on $\mathcal{M}$ associated with the time 1 heat kernel. We denote the inner product and norm by $\langle \cdot, \cdot \rangle_{\mathcal{H}}$ and $\lVert \cdot \rVert_{\mathcal{H}}$, respectively. A natural orthonormal basis of $\mathcal{H} $ consists of the exponentially scaled Laplace-Beltrami eigenfunctions (cf.~\eqref{eqPhiHP})
\begin{equation}
  \label{eqPhiRKHS}
  \tilde \phi_j = \frac{ \phi_j }{\rVert \phi_j \rVert_{\mathcal{H}} }= e^{-\lambda_j/2} \phi_j, \quad j \in \{ 0, 1, \ldots \}.
\end{equation}
After inclusion (which can be shown to be compact), $\mathcal{H} $ is a dense subspace of $H$ consisting of all equivalence classes of functions $\sum_{j=0}^\infty c_j \phi_j$ satisfying the inequality $ \sum_{j=0}^\infty e^{\lambda_j} \lvert c_j \rvert^2 < \infty$. In addition, $\mathcal{H} $ can be compactly embedded into every Sobolev space $ H^p $, $ p > 0 $. $ \mathcal{H}  $ is also a dense subspace of $ C^k(\mathcal{M}) $ for all $ k \geq 0 $ \cite{FerreiraMenegatto13}.   

In what follows, we will also be interested in spaces of bounded operators on the function spaces introduced above. Given two Banach spaces $V_1 $ and $ V_2$, $\mathcal{B}(V_1,V_2) $ will denote the Banach space of bounded operators mapping $V_1 $ to $V_2$, equipped with the operator norm, $ \lVert \cdot \rVert $. If $V_1 $ and $V_2$ are Hilbert spaces, $\mathcal{B}_2(V_1,V_2) \subseteq \mathcal{B}(V_1,V_2)$ will denote the Hilbert space of Hilbert-Schmidt operators from $V_1$ to $ V_2$, equipped with the inner product $ \langle A, B \rangle_\text{HS} = \tr( A^* B )$ and the corresponding norm, $\lVert A \rVert_\text{HS} = \sqrt{\langle A, A \rangle_\text{HS}} $. We will also use the abbreviations $ \mathcal{B}(V_1) = \mathcal{B}(V_1,V_1)$ and $ \mathcal{B}_2(V_1) = \mathcal{B}_2(V_1,V_1)$.

\subsection{Spaces of vector fields}

We consider the space $ \mathfrak{X} $ of $ C^\infty $ complex vector fields on $ \mathcal{M} $ (that is, the space of derivations on the ring of smooth, complex-valued functions on $\mathcal{M}$, or, equivalently, the space of smooth sections of $T^{\mathbb{C}} \mathcal{M}$), where we recall that $ \mathfrak{ X} $ can be viewed either as a vector space over the field of complex numbers, or as a $C^\infty(\mathcal{M})$-module. In the former case, it can be endowed with the structure of a Lie algebra with the vector field commutator, $ [ \cdot, \cdot ] : \mathfrak{ X } \times \mathfrak{ X} \to \mathfrak{ X} $ acting as the algebraic product. We denote the gradient and divergence operators associated with $ g $ by $ \grad : C^\infty( \mathcal{M} ) \to \mathfrak{ X} $ and $ \divr : \mathfrak{ X } \to C^\infty( \mathcal{M} ) $. Note that these operators are related to the positive-semidefinite Laplacian via $ \Delta = - \divr \circ \grad $. As with functions, we consider the standard Banach spaces $ L^p_\mathfrak{X} $, $ 1 \leq p \leq \infty $, of vector fields associated with the norms
\begin{displaymath}
  \lVert v \rVert_{L^p_\mathfrak{X}} = \left( \int_\mathcal{M} ( g( v, v ) )^{p/2} \, d\mu \right)^{1/p}, \quad 1 \leq p < \infty, \quad \text{and} \quad \lVert v \rVert_{L^\infty_\mathfrak{X}} = \esssup_{x\in \mathcal{M}} \sqrt{ g( v, v )_x }. 
\end{displaymath} 
In the case $ p = 2 $, we set $ H_\mathfrak{X} = L_\mathfrak{X}^2 $ and use the notation $ \langle v, w \rangle_{H_\mathfrak{X}} = \int_\mathcal{M} g( v^*, w ) \, d\mu $ for the Hodge inner product inducing the $ L_\mathfrak{X}^2 $ norm.  

With these definitions, the closure $ \wgrad : D( \wgrad ) \to H_\mathfrak{X} $ of the gradient operator has domain $ D( \wgrad ) = H^1 $, and is bounded as an operator from $ H^1 $ to $ H_\mathfrak{X} $. Similarly, the closure $ \wdivr : D( \wdivr ) \to H_\mathfrak{X} $ of the divergence operator has as its domain  $ D( \wdivr )   $ the Sobolev space  $ H_{\mathfrak{X},\divr}  \subset H_\mathfrak{X}$, defined as the closure of $ \mathfrak{X} $ with respect to the norm $ \lVert v \rVert_{H_{\mathfrak{X},\divr}} = \langle v, v \rangle_{H_{\mathfrak{X},\divr}}^{1/2}$  induced by the inner product 
\begin{displaymath}
    \langle v, w \rangle_{H_\mathfrak{X},\divr} = \langle v, w \rangle_{H_\mathfrak{X}} + \langle \divr v, \divr w \rangle_{H}.
\end{displaymath} 
Also, for $ p \geq 0 $, we introduce the Sobolev spaces 
\begin{displaymath}
  H^p_{\mathfrak X,\divr} = \left \{ v \in H_\mathfrak{X} : \wdivr v \in H^p \right \},
\end{displaymath} 
which are equipped with the inner products
\begin{displaymath}
  \langle v, w \rangle_{H^p_{\mathfrak{X},\divr}} = \langle v, w \rangle_{H_\mathfrak{X}} + \langle \wdivr v, \wdivr w \rangle_{H^p},
\end{displaymath} 
and the corresponding norms $\lVert v \rVert_{H^p_{\mathfrak{X},\divr}} = \langle v, v \rangle_{H^p_{\mathfrak{X},\divr}}^{1/2}$. As in the case of functions, we define the Dirichlet forms $E_{p,\mathfrak{X},\divr} : H^p_{\mathfrak{X},\divr} \times H^p_{\mathfrak{X},\divr} \to \mathbb{C}$ by
\begin{displaymath}
  E_{p,\mathfrak{X,\divr}}(u,v) = E_p(\wdivr u, \wdivr v).
\end{displaymath}
The corresponding energy functionals, $E_{p,\mathfrak{X},\divr}(f) = E_{p,\mathfrak{X,\divr}}(f,f)$ assign measures of roughness to vector fields in $H^p_{\mathfrak{X},\divr} \cap H_{\mathfrak{X}}$ by $E_{p,\mathfrak{X},\divr}(f)/\lVert f \rVert^2_{H_\mathfrak{X}}$. 

An important subspace of $H_\mathfrak{X} $ is the closed subspace of gradient vector fields, $ H_{\mathfrak{X},\grad} = \overline{\ran( \grad ) }$. This leads to the orthogonal decomposition $ H_\mathfrak{ X } = H_{\mathfrak{X},\grad} \oplus H^\perp_{\mathfrak{X},\grad}$, and it can be readily checked that any vector field in $  H^\perp_{\mathfrak{X},\grad} \cap H_{\mathfrak{X},\divr} $ has vanishing divergence. A natural smooth orthonormal basis $ \{ u_j \}_{j=1}^\infty $ for $ H_{\mathfrak{X},\grad} $ is given by the normalized gradients of the Laplace-Beltrami eigenfunctions, 
\begin{equation}
  \label{eqEJ}
  u_j = \grad \phi_j / \lambda_j^{1/2}.
\end{equation} 

The following two lemmas characterize the behavior of  vector fields as operators on functions.  

\begin{lemma}[vector fields as conjugate-antisymmetric operators] \label{lemmaAV} To every vector field $ v \in H^1_{\mathfrak{X},\divr} $ there corresponds a unique operator $ A_v : C^\infty( \mathcal{M} ) \to H $ with the property 
  \begin{equation}
    \label{eqAVSym}
    \langle f, A_v h \rangle_{H} = - \langle A_{v^*} f, h \rangle_{H}, \quad \forall f, h \in C^\infty(\mathcal{M}).
  \end{equation}
  This operator is given by $ A_v = v + D_v $, where $ D_v : C^\infty( \mathcal{M} ) \to H $ is defined as $ D_v f = \divr( f v ) $, and we also have  
  \begin{equation}
    \label{eqVA}
    v( f ) = \frac{ A_v f  - f A_v 1 }{ 2 }.
  \end{equation}   
\end{lemma}
\begin{proof}
  The Leibniz rule for the divergence, $ \divr( f v ) = v( f ) + f \divr v $, and the fact that $ \int_\mathcal{M} \divr( f v ) \, d\mu $ vanishes on closed manifolds lead to 
  \begin{displaymath}
      \langle f, v( h ) \rangle_{H} = - \langle D_{ v^*} f, h \rangle_{H}.
  \end{displaymath}
  The claim in~\eqref{eqAVSym} follows from the definition of $ A_v $ and the last equation. Note that the restriction $ v \in H^1_{\mathfrak{X},\divr} $ is important in order for $ D_v $ and $ A_v $ to be well defined. To show that $ A_v $ is unique, suppose that $ u \in H^1_{\mathfrak{X},\divr}$, $ u \neq  v $, and consider $  ( A_u  - A_v ) f = 2 ( u - v ) f  + f  \divr( u - v )  $. If $ \divr u = \divr v $, then $ ( A_u - A_v) f = 2 ( u - v ) f $, which is non-vanishing for some $ f \in C^\infty(\mathcal{M})$. On the other hand, if $ \divr u \neq \divr v $, we have $ ( A_u - A_v ) 1 = \divr( u - v ) $, which is again non-vanishing. Equation~\eqref{eqVA} follows from the definition of $ A_v $ and the fact that $ A_v 1 =  \divr v $.     
\end{proof}

\begin{lemma}[vector fields as bounded operators] \label{lemmaVBound}Let $ v $ be a bounded vector field in $L^\infty_\mathfrak{X} $.  Then:
  \begin{enumerate}[(i)]
    \item $v$ extends uniquely to a bounded operator $  L_v \in \mathcal{B}(H^1,H) $ with operator norm $ \lVert L_v \rVert \leq \lVert v \rVert_{L^\infty_\mathfrak{X}} $.  
    \item The restriction of $v$ to $ \mathcal{H} $ is a Hilbert-Schmidt operator $ \tilde L_v \in \mathcal{B}_2(\mathcal{H}, H) $ with norm $ \lVert \tilde L_v \rVert_\text{HS} \leq C \lVert v \rVert_{L^\infty_\mathfrak{X}} $, where $C $ is a constant independent of  $v $.  
    \end{enumerate}
    As a result, the maps $ \iota : L^\infty_\mathfrak{X} \to \mathcal{B}(H^1,H)$ and $ \iota_2 : L^\infty_\mathfrak{X} \to \mathcal{B}_2(\mathcal{H},H)$ with $\iota v = L_v$ and $\iota_2 v = \tilde L_v$ are continuous embeddings.
\end{lemma}

\begin{proof}
  (i) Consider a vector field $ v \in L^\infty_\mathfrak{X} $, and let $ f $ be a $ C^\infty $ function. Then, we have 
\begin{align*}
  \lVert v( f ) \rVert^2_H &= \int_\mathcal{M} \lvert g( v, \grad f ) \rvert^2  \, d\mu \leq \int_\mathcal{M} g( v, v ) g( \grad f, \grad f ) \, d\mu \\
  & \leq \lVert v \rVert^2_{L^\infty_\mathfrak{X}} \lVert \grad f \rVert^2_{H_\mathfrak{X}} \leq \lVert v \rVert^2_{L^\infty_\mathfrak{X}} \lVert f \rVert^2_{H^1},
\end{align*} 
where the first inequality in the above follows from the Cauchy-Schwarz inequality. Thus, $ v $ is a densely-defined, bounded operator from $ C^\infty(\mathcal{M}) \cap H^1 $ to $ H$, and can be uniquely extended to $ L_v \in \mathcal{ B}(H^1,H) $ by the bounded linear transformation theorem. The fact  that $ \lVert L_v \rVert \leq \lVert v \rVert_{L^\infty_\mathfrak{X}} $ follows from the inequality $ \lVert v( f ) \rVert_H / \lVert f \rVert_{H^1} \leq \lVert v \rVert_{L^\infty_\mathfrak{X}}$.

(ii) Since $ \mathcal{H} \subset C^\infty(\mathcal{M}) $, proceeding as above we find that for any $f \in \mathcal{H}$,  $ \lVert \tilde L_v f \rVert_H \leq \lVert v \rVert_{L^\infty_\mathfrak{X}} \lVert f \rVert_{H^1} $. Moreover, since $ \mathcal{H} $ continuously embeds into $H^1 $, there exists a constant $ \tilde C $, independent of $ f$, such that $ \lVert f \rVert_{H^1} \leq \tilde C \lVert f \rVert_\mathcal{H} $. This shows that $ \tilde L_v $ lies in $ \mathcal{B}(\mathcal{H}, H )$. To establish that $ \tilde L_v $ lies in $ \mathcal{B}_2(\mathcal{H}, H )$, we compute
\begin{align*}
  \sum_{j=0}^\infty \langle \tilde \phi_j, \tilde L_v^* \tilde L_v \tilde \phi_j \rangle_\mathcal{H} & = \sum_{j=0}^\infty \langle \tilde L_v \tilde \phi_j, \tilde L_v \tilde \phi_j \rangle_{H} =  \sum_{j=0}^\infty \lVert \tilde L_v \tilde \phi_j \rVert_H^2  \leq \lVert v \rVert_{L^\infty_\mathfrak{X}}^2 \sum_{j=0}^\infty \lVert \tilde \phi_j \rVert^2_{H^1} \\
  &= \lVert v \rVert^2_{L^\infty_\mathfrak{X}} \sum_{j=0}^\infty e^{-\lambda_j} ( 1 + \lambda_j ),   
\end{align*}
where $ \{ \tilde \phi_j \}_{j=0}^\infty $ is the orthonormal basis of $\mathcal{H}$ from~\eqref{eqPhiRKHS}. It then follows from the Weyl estimate for Laplacian eigenvalues (see~\eqref{eqWeyl} ahead and the proofs of Theorems~\ref{thmFrame} and~\ref{thmFrameSob} in Section~\ref{secFrameProof}) that the quantity $ C^2 = \sum_{j=0}^\infty e^{-\lambda_j} ( 1 + \lambda_j )  $
is finite, and we conclude that 
\begin{displaymath}
  \lVert \tilde L_v \rVert_\text{HS}^2 = \tr( \tilde L_v^* \tilde L_v ) =  \sum_{j=0}^\infty \langle \tilde \phi_j, \tilde L_v^*  \tilde L_v \tilde \phi_j \rangle_\mathcal{H} \leq C^2 \lVert v \rVert^2_{L^\infty_\mathfrak{X}}, 
\end{displaymath}
as claimed. 
\end{proof}

An implication of Lemma~\ref{lemmaVBound} is that for any $ v \in L^\infty_\mathfrak{X} $ and every sequence $ f_n \in C^\infty(\mathcal{M}) $, converging to $ f $ in $H^1$ norm, $ v( f ) = \lim_{n\to\infty} v( f_n ) $ even though $ v $ is unbounded (and therefore discontinuous) on $C^\infty(\mathcal{M})$. It also follows from Lemma~\ref{lemmaVBound} that the operator $ A_v $ in Lemma~\ref{lemmaAV} associated with $ v \in H^1_{\mathfrak{X},\divr}  \cap L^\infty_\mathfrak{X} $ also extends uniquely to a bounded operator $ \tilde A_v : H^1 \to H $. 

Next, as discussed in Section~\ref{functionOverview}, we introduce a spectral representation of pointwise Riemannian inner products between gradient vector fields. For that, we first consider the product rule for the positive-definite Laplacian on smooth functions,
\begin{equation}
  \label{eqLapProduct}
  \Delta(fh) = (\Delta f) h + f (\Delta h) - 2 g(\grad f^*, \grad h), \quad f, h \in C^\infty(\mathcal{M}). 
\end{equation}
It follows by definition of the $H^2$ norms that the self-adjoint Laplacian $ \bar \Delta $ is bounded as an operator from $ H^2$ to $H$. As a result, given a sequence $ f_j \in  C^\infty(\mathcal{M}) $ converging to $ f $ in $ H^2 $ norm, we have
\begin{equation}
  \label{eqFCauchy}
  \bar \Delta f = \bar \Delta\left( \lim_{j\to\infty} f_j \right) = \lim_{j\to\infty} \Delta f_j.
\end{equation}
Now, the fact that the Laplace-Beltrami eigenfunctions $ \phi_j $ are smooth implies that given any $f \in C^\infty(\mathcal{M})$, the sequence $ f_0, f_1, \ldots $ with 
\begin{displaymath}
  f_j = \sum_{i=0}^j \hat f_i \phi_i, \quad \hat f_i = \langle \phi_i, f \rangle_H
\end{displaymath}
is Cauchy in $ H^p $ for all $ p \geq 0 $, and hence~\eqref{eqFCauchy} holds. As a result, we can use~\eqref{eqLapProduct} in conjunction with~\eqref{eqCIJK} to obtain:

\begin{lemma}[spectral representation of Riemannian inner products]
  The Riemannian inner product between the gradient vector fields associated with two  functions $f,h\in C^\infty(\mathcal{M})$ can be expressed as
  \begin{displaymath}
      g(\grad f, \grad h) = \frac{ 1 }{ 2 } \sum_{j,k,l=0}^\infty \hat f_j^* \hat h_k (\lambda_j + \lambda_k - \lambda_l) c_{jkl} \phi_l, \quad \hat f_j = \langle \phi_j, f \rangle_H,  
  \end{displaymath}
  where $\hat h_k = \langle \phi_k, h \rangle_H $, and the sum over $l$ in the right-hand side converges in $H$ norm.
  \label{lemmaInnerProd}
\end{lemma}

Note that Lemma~3.3 can be extended to  $ f, h \in H^1 $, which follows from the fact that the map $ ( f, h ) \mapsto g( \grad f, \grad h ) $ is a  bounded linear map  with a dense domain $ C^\infty(\mathcal{M}) \times C^\infty(\mathcal{M}) \subset H^1 \times H^1 $.

\subsection{\label{secSpacesForms}Spaces of differential forms}

We will use the symbols $ \Lambda^k_x \mathcal{M} $, $ \Lambda^k \mathcal{M}$, and $ \Omega^k $ to represent the vector space of complex $ k $-forms at $x \in  \mathcal{M} $, the associated $k $-form bundle, and the space of smooth $ k $-form fields on $ \mathcal{M} $ (totally antisymmetric, $k$-multilinear maps on $\mathfrak{X}^k$, taking values in $C^\infty(\mathcal{M})$, or, equivalently smooth sections of $\Lambda^k \mathcal{M}$). As with vector fields, the spaces $\Omega^k$ can be viewed either as vector spaces over $\mathbb{C}$, or as $C^\infty(\mathcal{M})$-modules. As usual, we identify $ \Omega^0 $ with $C^\infty(\mathcal{M})$.  We also let $ \eta^k_x : \Lambda^k_x \mathcal{M} \times \Lambda^k_x \mathcal{M} \to \mathbb{ R } $ be the canonical metric tensor on $ \Lambda^k_x $, satisfying 
\begin{displaymath}
    \eta^k_x ( \alpha^1 \wedge \cdots \wedge \alpha^k, \beta^1 \wedge \cdots \wedge \beta^k ) = \det [ \eta_x( \alpha^i, \beta^j ) ]_{ij}, \quad \forall \alpha^i, \beta^j \in T^{\mathbb{C}*}_x\mathcal{M}. 
\end{displaymath} 
The metric induces a Hodge star operator $ \star : \Lambda_x^k \mathcal{M} \to \Lambda_x^{d-k} \mathcal{M}$, defined uniquely through the requirement that
\begin{displaymath}
  \alpha \wedge \star \beta = \eta^k_x( \alpha, \beta ) \mu, \quad \forall \alpha, \beta \in \Lambda_x^k \mathcal{M}.
\end{displaymath} 
The Hodge star has the useful property
\begin{equation}
    \star \star \alpha = (-1)^{k(\dim\mathcal{M}-k)} \alpha, \quad \forall \alpha \in \Lambda_x^k \mathcal{M}.
  \label{eqHodgeStar2}
\end{equation}

As in the case of vector fields, we introduce the Banach spaces $ L^p_k $, $ 1 \leq p \leq \infty $, defined as the completion of $ \Omega^k $ with respect to the norms  
\begin{gather*}
  \lVert \alpha \rVert_{L^p_k} = \left( \int_\mathcal{M} \eta^k( \alpha, \alpha )^{p/2} \, d\mu \right)^{1/p}, \quad 1 \leq p \leq \infty, \\
  \lVert \alpha \rVert_{L^\infty_k} = \esssup_{x\in \mathcal{M}} \sqrt{ \eta^k( \alpha, \alpha)_x }.
\end{gather*}
The case $ p = 2 $ is a Hilbert space, $ H_k = L^2_k $, with norm $ \lVert \cdot \rVert_{H_k} = \lVert \cdot \rVert_{L^2_k}$ induced from the inner product 
\begin{displaymath}
    \langle \alpha, \beta \rangle_{H_k} = \int_\mathcal{M} \alpha \wedge \star \beta = \int_\mathcal{M} \eta^k  (\alpha, \beta ) \, d\mu.
\end{displaymath}

A fundamental aspect of the spaces $ \Omega^k $ is that they are linked by the exterior derivative and codifferential operators, $ d_k : \Omega^k \to \Omega^{k+1} $ and $ \delta_k : \Omega^k \to \Omega^{k-1} $, respectively. We recall that $ d_0, \ldots, d_{d-1} $ are the  unique linear maps with the properties:
\begin{enumerate}
  \item $ d_0 $ is the differential of functions. 
  \item $ d_{k+1} d_k = 0 $.
  \item The Leibniz rule,
    \begin{equation}
      \label{eqLeibniz}
      d_{k+l}( \alpha \wedge \beta ) = d_k \alpha \wedge \beta + (-1)^k \alpha \wedge d_l \beta,
    \end{equation}
    holds for all $ \alpha \in \Omega^k $ and $ \beta \in \Omega^l $. 
\end{enumerate}
The codifferential operators are defined uniquely through the requirement that 
\begin{displaymath}
  \langle \alpha, \delta_k \beta \rangle_{H_{k-1}} = \langle d_{k-1} \alpha, \beta \rangle_{H_k}; 
\end{displaymath}  
i.e., $ \delta_k $ is a formal adjoint of $ d_{k-1} $. This definition of $ \delta_k $ is equivalent to 
\begin{equation}
  \label{eqCoDiff}
  \delta_k \alpha = ( -1 )^{\dim\mathcal{M}(k+1)+1} \star d_{d-k} \star \alpha,
\end{equation}
and it also implies $ \delta_k \delta_{k-1} = 0 $. In the case $k=1$, the codifferential operator is related with the divergence on vector fields via $ \divr v = - \delta_1 v^\flat $, $ v \in \mathfrak{X} $. Note that despite its relationship with the exterior derivative in~\eqref{eqCoDiff}, the codifferential does not satisfy a Leibniz rule. 

Another important class of operators on differential forms are the interior product and Lie derivative associated with vector fields. Given a vector field $v \in \mathfrak{X}$, these are defined as the maps $\iota_v : \Omega^k \to \Omega^{k-1}$ and $\mathcal{L}_v : \Omega^k \to \Omega^k$, respectively, such that
\begin{displaymath}
    \iota_v \alpha = \alpha( v, \cdot ), \quad \mathcal{L}_v = d \iota_v + \iota_v d. 
\end{displaymath}
Both $ \iota_v $ and $\mathcal{L}_v$ satisfy Leibniz rules,
\begin{displaymath}
  \iota_v(\alpha \wedge \beta ) = (\iota_v \alpha) \wedge \beta + (-1)^k \alpha \wedge (\iota_v \beta), \quad \mathcal{L}_v( \alpha \wedge \beta ) = ( \mathcal{L}_v \alpha ) \wedge \beta + \alpha \wedge (\mathcal{L}_v),  
\end{displaymath}
for all $\alpha \in \Omega^k $ and $ \beta \in \Omega^l $. Moreover, they have the properties
\begin{equation}
  \mathcal{L}_v d = d \mathcal{L}_v, \quad v^\flat \wedge \star \alpha = (-1)^k \star \iota_v \alpha.
  \label{eqCartan}
\end{equation}
 The following lemma can be viewed as a generalization of Lemma~\ref{lemmaVBound} to spaces of differential forms.
\begin{lemma}
    For every $v \in \mathfrak{X}$, the operators $ \iota_v d $, $ d \iota_v $, and $ \mathcal{L}_v $ extend to unique bounded operators $D_v : H^1_k \to H_k$, $\tilde D_v : H^1_k \to H_k$, and $L_v : H^1_k \to H_k$, respectively.
  \label{lemmaVBoundK}
\end{lemma}
\begin{proof}
  First, establish that $\lVert \iota_v d \omega \rVert^2_{H_k} / \lVert \omega \rVert_{H^1_k}^2$ is bounded using local Cauchy-Schwartz inequalities as in the proof of Lemma~\ref{lemmaVBound}. This, in conjunction with the bounded linear transformation theorem implies the existence and uniqueness of $D_v$, as claimed. The results for $\tilde D_v $ and $L_v$ follow similarly.
\end{proof}

The exterior derivative and codifferential operators lead to the Hodge Laplacian $ \Delta_k : \Omega^k \to \Omega^k $ on $ k $-forms, defined as
\begin{displaymath}
  \Delta_k = \delta_{k+1} d_k + d_{k-1} \delta_k. 
\end{displaymath} 
As with the Laplacian $ \Delta = \Delta_0 $ on functions, the Laplacian on $ k $-forms on closed manifolds has a unique self-adjoint extension $ \bar \Delta_k : D( \bar \Delta_k ) \to H_{k} $, with a pure point spectrum of eigenvalues $ 0 = \lambda_{0,k} \leq \lambda_{1,k} \leq \cdots $ with no accumulation points and an associated smooth orthonormal basis $ \{ \phi_{j,k} \}_{j=0}^{\infty} $ of eigenforms \cite{Strichartz83,Rosenberg97}. 

A central result in exterior calculus on manifolds is the Hodge decomposition theorem, which states that $ \Omega^k $ admits the decomposition
\begin{equation}
  \label{eqHodge}
  \Omega^k = \ran d_{k-1} \oplus \ran \delta_{k+1} \oplus \mathcal{ H }_0^k,  \quad \mathcal{ H }_0^k = \ker \Delta_k,
\end{equation}   
into subspaces of closed ($ \ran d_{k-1} $), coclosed ($\ran \delta_{k+1}$), and harmonic ($\mathcal{H}^k$) forms, all of which are invariant under $ \Delta_k$.  On a a compact manifold, $ \mathcal{ H }_0^k = \ker d_k \cap \ker \delta_k $,  and the dimension of this space is finite. The Hodge decomposition in~\eqref{eqHodge} has an $ L^2 $ extension, 
\begin{displaymath}
  H_{k} = \mathcal{ H }^k_d \oplus \mathcal{ H }^k_\delta \oplus \mathcal{ H }_0^k,  \quad \mathcal{ H }^k_d = \overline{\ran d_{k-1} }, \quad \mathcal{ H }^k_\delta =  \overline{\ran \delta_{k+1} },   
\end{displaymath} 
where the closed spaces $ \mathcal{ H }^k_d $, $ \mathcal{ H }_\delta^k $, and $ \mathcal{ H }^k $ are mutually orthogonal. 

It follows directly from the definition of $ \Delta_k $ that $ d_k \Delta_k = \Delta_{k+1} d_k $ and $ \delta_k \Delta_k = \Delta_{k-1} \delta_k $. This implies that every eigenform of $\Delta_k$ can be chosen to lie in one of the $ \mathcal{ H }^k_d $, $ \mathcal{ H }^k_\delta $, or $ \mathcal{ H }^k_0 $ subspaces. Moreover, for every $ k $-eigenform $ \psi\in \mathcal{ H }^k_d $ there exists a $ ( k - 1 ) $-eigenform $ \varphi \in \mathcal{ H }^k_\delta $ such that $\psi = d_{k-1} \varphi $, and similarly for every $ \psi \in \mathcal{ H }^k_\delta $  there exists as $(k+1)$-eigenform $ \omega \in \mathcal{ H }^{k+1}_d $ such that $\psi = \delta_{k+1} \omega $.     

Besides providing orthonormal bases for the invariant subspaces in the Hodge decomposition of  $ H_k $, the eigenfunctions of $ \Delta_k $ and the corresponding eigenvalues are also useful for constructing Sobolev spaces analogous to the $ H^p $ function spaces in~\eqref{eqHP}. Given $ p \geq 0 $, we define
\begin{displaymath}
  H^p_k = \left \{ \sum_{j=0}^\infty c_j \phi_{j,k} \in H_k : \sum_{j=0}^\infty   \lambda^p_{k,j} \lvert c_j \rvert^2 < \infty \right \}.
\end{displaymath}
These spaces are Hilbert spaces with inner products
\begin{equation}
  \label{eqHPKInnerProduct}
  \langle \alpha, \beta \rangle_{H^p_k} = \sum_{q=0}^p\sum_{j=0}^\infty \lambda_{j,k}^q \hat \alpha^*_j \hat \beta_j, \quad \hat \alpha_j = \langle \phi_{j,k}, \alpha \rangle_{L^2_k},  \quad \hat \beta_j = \langle \phi_{j,k}, \beta \rangle_{L^2_k}
\end{equation}
and norms $ \lVert \alpha \rVert_{H^p_k} = \langle \alpha, \alpha \rangle^{1/2}_{H^p_k} $. As in the case of functions and vector fields, we equip these spaces with positive-semidefinite Dirichlet forms $E_{p,k} : H^p_k \times H^p_k \to \mathbb{C} $, given by
\begin{displaymath}
  E_{p,k} = \sum_{j=0}^\infty \lambda_{j,k}^p \hat \alpha^*_j \beta_j
\end{displaymath}
for $\alpha, \beta $ from~\eqref{eqHPKInnerProduct}. Note that $E_{p,k}$ can be equivalently expressed using the exterior derivative and codifferential operators; in particular, 
\begin{equation}
  E_{1,k}(\alpha, \beta) = \langle \bar d_k \alpha, \bar d_k \beta \rangle_{H_{k+1}} + \langle \bar \delta_k \alpha, \bar \delta_k \beta \rangle_{H_{k-1}},
  \label{eqDirichletK}
\end{equation}
where overbars denote operator closures. Moreover, if $\beta $ is smooth, $E_{p,k}(\alpha,\beta)$ can be expressed in terms of the $k$-Laplacian via  
\begin{displaymath}
  E_{p,k}(\alpha,\beta) = \langle \alpha, \Delta^p_k \beta \rangle_{H_{k}},
\end{displaymath}
with an analogous expression holding if $\alpha \in \Omega^k$. The Dirichlet forms defined above induce the energy functionals $E_{p,k}(f) = E_{p,k}(f,f)$ measuring the roughness of forms in $H^p_k\cap H_k$ through $E_{p,k}(f)/ \lVert f \rVert_{H_k}^2$. 
 
For notational simplicity, henceforth we will drop the overbars from our notation for the closed differential, codifferential, and Laplacian on $k$-forms. We will also drop $k$ superscripts and subscripts from $ \eta^k $, $d_k$, $\delta_k$, and $\Delta_k$.

\section{\label{secSECSmooth}Spectral exterior calculus (SEC) on smooth manifolds}

In this section, we introduce our representation of vector fields and forms using frames constructed from Laplace-Beltrami eigenfunctions and their derivatives.  Besides rigorously satisfying the appropriate frame conditions for a number of $L^2$ and Sobolev spaces of interest in exterior calculus, an advantage of this representation is that it is fully spectral, and thus can also be applied in the discrete case with little modification.

\subsection{\label{secFrameDef}Frame representation of vector fields and forms}

We begin by recalling the definition of a frame of a Hilbert space \cite{Mallat09}.

\begin{definition}[frame of a Hilbert space] 
    Let $ ( V, \langle \cdot, \cdot \rangle_V ) $ be a Hilbert space over the complex numbers and $ u_0, u_1, \ldots $  a sequence of elements $ u_k \in V $. We say that the set $\{ u_k \} $ is a frame if there exist positive constants $ C_1 $ and $ C_2 $ such that the following frame conditions hold for all $ v \in V $: 
  \begin{equation}
    \label{eqFrameCond}
    C_1 \lVert v \rVert^2_V \leq \sum_k \lvert \langle u_k, v \rangle_V \rvert^2 \leq C_2 \lVert v \rVert^2_V.
  \end{equation}
  \label{defFrame}
\end{definition}

The frame $ \{ u_k \} $ induces a linear operator $ T : V \to \ell^2 $, called analysis operator,  such that $ T v = \hat v' = ( \hat v'_k )_k  $ with $ \hat v'_k = \langle u_k, f \rangle_V $. This operator is bounded above and below via the same constants as in~\eqref{eqFrameCond}; that is,
\begin{displaymath}
  C_1 \lVert v \rVert^2_V \leq \lVert T v \rVert_{\ell^2}^2 \leq C_2 \lVert v \rVert_V^2.
\end{displaymath}
The adjoint, $ T^* : \ell^2 \to V $, is called synthesis operator.

The analysis and synthesis operators induce a positive-definite, self-adjoint, bounded operator $ S : V \to V $ with bounded inverse, called frame operator, which is given by $ S = T^* T $. This operator satisfies the bounds
\begin{displaymath}
  C_1 \lVert v \rVert^2_V \leq \langle v, S v \rangle_V \leq C_2 \lVert v \rVert_V^2.
\end{displaymath}
The fact that $S$ has bounded inverse implies that the set $ \{ u'_k \} $ with $ u'_k = S^{-1} u_k $ is also a frame, called dual frame. This frame has the important property
\begin{equation}
  \label{eqVFrameDual}
  v = \sum_k \langle u'_k, v \rangle_V u_k = \sum_k \langle u_k, v \rangle_V u'_k, \quad \forall v \in V.
\end{equation}   
This means that the inner products $ \langle u'_k, v \rangle_V $ between $ v$ and the dual frame elements correspond to expansion coefficients in the original frame that reconstruct $ v $, and conversely, the coefficients $ \langle u_k, v \rangle_V $ reconstruct $ v $ in the dual frame. Denoting the analysis operator associated with the dual frame by $ T' : V \to \ell^2$, we have $T' = T S^{-1}$, $T^{\prime *} = S^{-1} T^* $, and \eqref{eqVFrameDual} can be equivalently expressed as
\begin{equation}
  \label{eqVFrameDual2}
  v = T^* T' v = T^{\prime*} T v, \quad \forall v \in V.
\end{equation}
Moreover, the dual frame operator $S' : V \to V $, $ S'= T^{\prime *} T'$, is equal to $S^{-1} $. 

Another class of operators of interest in frame theory are the Gramm operators $G: \ell^2 \to \ell^2$, $G = T T^*$, and $ G' : \ell^2 \to \ell^2 $, $G' = T' T^{\prime *}$, associated with the frame and dual frame, respectively. While $G $ and $ G'$ are both bounded, unlike $S$ and $S'$, they are non-invertible if the frame has linearly dependent elements. Nevertheless, it follows from~\eqref{eqVFrameDual2} that
\begin{equation}
  \label{eqVFrameDualGramm}
  T v = G T' v, \quad T'v = G' Tv, \quad \forall v \in V,
\end{equation}
which implies that $ G $ (resp.\ $G'$) is invertible on the range of $ T' $ (resp.\ $T$), and its inverse is given by $ G' $ (resp.\ $G$) . Denoting the canonical orthonormal basis of $\ell^2$ by $\{ e_k \}_{k=0}^\infty$, we have 
    \begin{displaymath}
        G_{ij} := \langle e_i, G e_j \rangle_{\ell^2} = \langle T^* e_i, T^* e_j \rangle_{V} = \langle u_i, u_j \rangle_{V},
    \end{displaymath}
and similarly $ G'_{ij} = \langle e_i, G' e_j \rangle_{\ell^2} = \langle u'_i, u'_j \rangle_{V} $. Thus, the matrix elements of the Gramm operators $ G $ and $ G' $ in the $\{ e_k \}_{k=0}^\infty $ basis are equal to the pairwise inner products between the frame and dual frame elements, respectively. By boundedness of these operators, for any $ \hat f = \sum_{j=0}^\infty \hat f_j e_j \in \ell^2 $, we have $G \hat f = \sum_{i,j=0}^\infty e_i G_{ij} \hat f_j$ and  $G' \hat f = \sum_{i,j=0}^\infty e_i G'_{ij} \hat f_j$.        

Since the analysis operator $ T $ is bounded below, it follows by the closed range theorem that $ T $ and $ T^* $ have closed range, and as a result the ranges of $ S $ and $ G $ are also closed. Similarly, all of $ T' $, $ T^{\prime*} $, $ S' $, and $ G' $ have closed range. An important consequence of these properties is that all of these operators have well-defined pseudoinverses. In particular, it can be shown \cite{Christensen95} that the pseudoinverse $(T^{*})^+ $ of $ T^* $ is equal to the dual analysis operator, $ (T^{*})^+ = T' $, and similarly $ (T^{\prime*})^+ = T $. These relationships imply in turn that $G^+ = G'$ and $ G^{\prime+} = G$. 

Clearly, in a separable Hilbert space, every Riesz basis is also a frame. For example, in the case $ V = H = L^2(\mathcal{M},\mu)$, a natural frame is provided by the Laplace-Beltrami eigenfunction basis $ \{ \phi_k \}_{k=0}^\infty $. In that case, the analysis  operator $T $ is unitary, $ T^*T = S = I_H$,  $T T^* = G  = I_{\ell^2} $, by orthonormality of the basis. In the setting of vector fields on manifolds, a natural orthonormal set of smooth fields in $ H_\mathfrak{X} $ is given by the normalized gradient fields $ u_j $ from~\eqref{eqEJ}. However, this only provides a basis for the space of gradient fields, $ H_{\mathfrak{X},\grad} $. To construct a representation of arbitrary vector fields in  $H_\mathfrak{X} $, we can take advantage of the $ C^\infty(\mathcal{M}) $-module structure of smooth vector fields to augment this set by multiplication of gradient fields by smooth functions. Doing so will result in an overcomplete spanning set of $ H_\mathfrak{X}$, which will turn out to meet the frame conditions in Definition~\ref{defFrame}. We will follow a similar approach to construct frames for the $H_k$ spaces of differential forms, where we will also construct frames for higher-order Sobolev spaces through eigenvalue-dependent normalizations of the frame elements as in~\eqref{eqPhiHP}. 

\begin{remark}
    As alluded to in Section~\ref{secOverview}, a key property of the frames for spaces of vector fields and differential forms introduced below is that the matrix elements $G_{ij}$ of the corresponding Gramm operators can be evaluated via closed form expressions that depend only on the eigenvalues $ \lambda_i $ of the Laplacian on functions and the corresponding coefficients $c_{ijk}$ from~\eqref{eqCIJK}. This allows in turn the SEC to be built entirely from the spectral properties of the Laplacian on functions.
\end{remark}

We begin by introducing the vector fields and forms which will be employed in our frame construction and Galerkin schemes below. In the case of vector fields, we define 
\begin{equation}
  b_{ij} = \phi_i \grad \phi_j, \quad 
  \widetilde b_{ij} = e^{-\lambda_j/2} b_{ij}, \quad b_{ij}^{(p)} = b_{ij} / \lVert \phi_i \rVert_{H^p}, \quad \quad \widetilde b_{ij}^{(p)} = \widetilde b_{ij} / \lVert \phi_i \rVert_{H^p}, 
\end{equation}
and
\begin{equation}
  \label{eqFrameVecA}
  \widehat b_{ij} = b_{ij} - b_{ji}, \quad \widecheck b_{ij} = e^{-(\lambda_i + \lambda_j)/2}\widehat b_{ij}, 
\end{equation}
all of which are smooth vector fields in $\mathfrak{X}$. We also define the smooth forms 
\begin{equation}
  \label{eqFrameForm}
  \begin{gathered}
      b^i = \phi_i \in \Omega^0, \\  
      b^{ij_1 \cdots j_k} = b^i \, db^{j_1} \wedge \cdots \wedge db^{j_k} \in \Omega^k, \quad \widetilde b^{ij_1 \cdots j_k} = e^{-( \lambda_{j_1} + \ldots + \lambda_{j_k})/2} b^{ij_1\cdots j_k},  \\ 
      b^{ij_1\cdots j_k}_{p} = b^{ij_1\cdots j_k} / \lVert \phi_i \rVert_{H^p}, \quad  \widetilde  b^{ij_1\cdots j_k}_{p} = \widetilde b^{ij_1\cdots j_k} / \lVert \phi_i \rVert_{H^p}, 
    \end{gathered}
\end{equation}
and
\begin{equation}
  \label{eqFrameFormA}
  \widehat b^{ij_1\cdots j_k} = b^{[ij_1 \cdots j_k]}, \quad \widecheck b^{ij_1 \cdots j_k} = e^{-(\lambda_i + \lambda_{j_1} + \ldots + \lambda_{j_k})}\widehat b^{[i j_1 \cdots j_k]}, 
\end{equation}
where the square brackets $[ij_1 \cdots j_k]$ denote total antisymmetrization with respect to the enclosed indices; e.g., 
\begin{displaymath}
   b^{[i j_1 j_2]} = b^{i j_1 j_2} - b^{i j_2 j_1} + b^{j_1 j_2 i} - b^{j_1 i j_2} + b^{j_2 i j_1} - b^{j_2 j_1 i}.
\end{displaymath}
With these definitions, our main results on frames for Hilbert spaces of vector fields and forms are as follows.

\begin{theorem}[frames for $L^2$ spaces of vector fields and forms]\label{thmFrame}
There exists an integer $J_0 \geq d = \dim \mathcal{M}$, such that for every integer $ J \geq J_0 $, the sets
  \begin{displaymath}
    B_{\mathfrak{X}}^J = \{ b_{ij} : i \in \{ 0, 1, \ldots \}, \; j \in \{ 1, \ldots, J \} \}, \quad \widetilde B_{\mathfrak{X}} = \{ \widetilde b_{ij} : i \in \{ 0,1, \ldots \}, \; j \in \{ 1, 2, \ldots \} \},
  \end{displaymath}
  are frames for $H_\mathfrak{X}$. Moreover, the sets 
  \begin{align*}
    B^J_k &= \{ b^{i j_1 \cdots j_k} : i \in \{ 0, 1, \ldots, \}, \; j_1, \ldots, j_k \in \{ 1, \ldots, J \} \}, \\
    \widetilde B_k &= \{ \widetilde b^{ij_1 \cdots j_k} : i \in \{ 0, 1, \ldots \}, \; j_1, \ldots, j_k \in \{ 1,2, \ldots \} \},
  \end{align*}
  with $k \in \{ 1, \ldots, m \} $, are frames for $H_k$. 
\end{theorem}

\begin{theorem}[frames for order-1 Sobolev spaces of 1-forms]
  \label{thmFrameSob}
  For the same integer $J$ as in Theorem~\ref{thmFrame}, the sets 
  \begin{displaymath}
    B^J_{1,1} = \{ b^{i j}_1 : i \in \{ 0, 1, \ldots \}, \; j \in \{ 1, \ldots, J \} \}, \quad \widetilde B_{1,1} = \{ \widetilde b^{i j}_1 : i \in \{ 0, 1, \ldots \}, \; j \in \{ 1, 2, \ldots \} \}.
  \end{displaymath}
  are frames for $H^1_1$.
\end{theorem}
We will prove Theorems~\ref{thmFrame} and~\ref{thmFrameSob} in Sections~\ref{secFrameProof} and~\ref{secFrameSobProof}, respectively. In addition, while we have not explicitly verified this, it should be possible to show via a similar approach to that in Section~\ref{secFrameSobProof} that frames for $H^p_k$, $ k,p > 1 $, can be constructed using $b^{i j_1 \cdots j_k}_p$ or $\widetilde b^{i j_1 \cdots j_k}_p$. It may also be possible to establish such results inductively with respect to the Sobolev order $p$, using the results in \ref{derivations}. Based on these considerations, we conjecture the following: 

\begin{conj}[frames for Sobolev spaces of forms]
    \label{conjFrameSob}
  For the same integer $J$ as in Theorem~\ref{thmFrame}, the following sets are frames for $H^1_k$, $k \in \{ 1, \ldots, m \}$, $ p \geq 0 $:
  \begin{align*}
    B^J_{k,p} &= \{ b^{i j_1 \cdots j_k}_p : i \in \{ 0, 1, \ldots \}, \; j_1, \ldots, j_k \in \{ 1, \ldots, J \} \}, \\
    \widetilde B_{k,p} &= \{ \widetilde b^{i j_1 \cdots j_k}_1 : i \in \{ 0, 1, \ldots \}, \; j_1, \ldots, j_k \in \{ 1, 2, \ldots \} \}.
  \end{align*}
\end{conj}

\begin{remark}
    It should also be noted that we have not established frame conditions for the antisymmetric elements, $\widehat b_{ij}$ and $ \widehat b^{ij_1\cdots j_k}$, or their rescaled counterparts, $\widecheck b_{ij}$ and $\widecheck b^{ij_1\cdots j_k}$. In fact, to fully span $H_\mathfrak{X}$ and $ H_k$ using $\widehat b_{ij}$ and $\widehat b^{ij_1\cdots j_k}$, respectively, one would have to use infinitely many $i$ and $j$ indices, leading to violations of the upper frame condition. Nevertheless, as demonstrated by the formulas derived in \ref{derivations} and listed in Table~\ref{fig3}, due to cancellation of terms by antisymmetrization, $\widehat b_{ij}$ and $\widehat b^{ij_1 \cdots j_k}$ can sometimes lead to considerable simplification of the representation of operators of interest in exterior calculus (e.g., the 1-Laplacian). Moreover, in the applications presented in Section~\ref{secNumerics} with available analytical results for the eigenvalues and eigenforms of the 1-Laplacian (i.e., the circle and flat torus), we found that SEC formulations based on the antisymmetric elements actually exhibit a moderate performance increase over those based on the non-symmetric elements. These facts motivate further exploration of the construction of frames based on antisymmetric elements. For example, the exponentially scaled $\widecheck b^{ij_1\cdots j_k} $ might provide frames for reproducing kernel Hilbert spaces associated with the heat kernel on $k$-forms.     
\end{remark}

For the remainder of this section, we discuss the basic properties of the vector fields and forms just defined. 
We begin by establishing that, while they may not form a basis, finitely many gradient vector fields of Laplace-Beltrami eigenfunctions are sufficient to generate arbitrary smooth vector fields on closed manifolds. 

\begin{lemma}
    \label{lemmaSpanning}There exists an integer $ J \geq \dim \mathcal{M} $ such that $ \{ \grad \phi_1 \rvert_x, \ldots, \grad \phi_J \rvert_x \} $ is a spanning set of $ T_x \mathcal{M} $, and thus $T_x^{\mathbb{C}}\mathcal{M}$, at every $ x \in \mathcal{M} $.
\end{lemma}

\begin{proof}
    The claim follows from the fact that the there exists an integer $ J $ such that the map $ \vec \phi_J : \mathcal{M} \to \mathbb{ R }^J $ with $ \vec \phi_J( x ) = ( \phi_1( x ), \ldots, \phi_J( x ) ) $ is an embedding of  $\mathcal{M} $, which is proved in \cite[Theorem~4.5]{Portegies14}. Note that $ B_J $ being an embedding implies that $ J \geq \dim \mathcal{M} $. To see why it implies that $ \{ \grad \phi_1 \rvert_x, \ldots, \grad \phi_J \rvert_x \} $ is a spanning set, fix a point $ x \in \mathcal{M} $, and consider the tangent vector pushforward map $ \vec \phi_{J*} \rvert_x : T_x \mathcal{M} \to T_{\vec \phi_J(x)} \mathbb{ R}^J $. Since $ T_{\vec \phi_J(x)} \mathbb{ R}^J  $ is canonically isomorphic to  $ \mathbb{ R }^J $, in a coordinate chart $ u : N_x \to \mathbb{R }^d $ defined on a neighborhood $ N_x $ of $ x $, the pushforward map is represented by a $ d \times J $ matrix $ \bm \Xi(x) $, $ d = \dim \mathcal{M} $, with elements $ \Xi_{ij} = \left. \frac{\partial \phi_j}{\partial u_i}\right \rvert_x $, and because $ \vec\phi_J $ is an embedding, that matrix has full rank, $ \rank \bm \Xi(x) = d $. In this coordinate basis, the components $ \beta_{ij} $ of $ \grad \phi_j \rvert_x  = \sum_{i=1}^d \beta_{ij} \left. \frac{ \partial \ }{ \partial u_i } \right \rvert_x $ are given by $ \beta_{ij} = \sum_{k=1}^d \eta_{ik} \Xi_{kj} $, where $ \eta_{ik} $ are the components of the dual metric $ \eta_x : T^* \mathcal{M}_x \times T^*\mathcal{M}_x \to \mathbb{ R } $. The $ \eta_{ik} $ form a $ d \times d $ invertible matrix, and thus  the $ d \times J $ matrix with elements $ \beta_{ij} $ has rank $ d $. This implies that $ \spn  \{ \grad \phi_1 \rvert_x, \ldots, \grad \phi_J \rvert_x \} = T_x \mathcal{M} $.
\end{proof}

\begin{corollary} \label{corGen}The set $ \{ \grad \phi_1, \ldots, \grad \phi_J \} $ is a generating set for $ \mathfrak{ X } $ viewed as a $ C^\infty(\mathcal{M}) $-module. That is, for every smooth vector field $ v \in \mathfrak{ X} $, there exist (not necessarily unique) smooth functions $ f_1, \ldots, f_J $ such that $ v = \sum_{j=1}^J f_j \grad \phi_j $. 
\end{corollary}  
  
\begin{corollary} \label{corGenK}The collection of $k$-form fields $d\phi_{j_1} \wedge \cdots \wedge d\phi_{j_k}$ with $j_1, \ldots, j_k \in \{1, \ldots, J\}$ spans $\Lambda^k_x$ at every $x\in \mathcal{M}$. As a result, this set is a generating set for $\Omega^k$, which means that for every $\omega \in \Omega^k$ there exist smooth functions $f_{j_1 \cdots j_k} $ such that $\omega = \sum_{j_1,\ldots,j_k=1}^J f_{j_1 \cdots j_k} \, d\phi_{j_1} \wedge \cdots \wedge d\phi_{j_k}$. 

\end{corollary}

It follows from Corollary~\ref{corGen} and the fact that $ \mathfrak{ X } $ is dense in $ H_\mathfrak{X} $ that for every $ v \in H_\mathfrak{X} $ there exist functions $ f_1, \ldots, f_J \in H $ such that $ v = \sum_{j=1}^J f_j \grad \phi_j $. Expanding these functions as $ f_j = \sum_{i=0}^\infty c_{ij} \phi_j $ with $ c_{ij} = \langle \phi_i, f_j \rangle_H $, we conclude that every vector field $ v \in H_\mathfrak{X} $ is expressible in the form 
\begin{equation}
  \label{eqVExpansion}
  v = \sum_{i=0}^\infty \sum_{j=1}^J c_{ij} b_{ij}, 
\end{equation}    
for some (not necessarily unique) constants $ c_{ij} \in \mathbb{C}$. Similarly, Corollary~\ref{corGenK} implies that every $k$-form field in $H_k$ can be expanded as  
\begin{equation}
  \label{eqWExpansion}
  \omega = \sum_{i=0}^\infty \sum_{j_1, \ldots, j_k=1}^J c_{ij_1\cdots j_k} b^{ij_1\cdots j_k}, \quad c_{ij_1\cdots j_k} \in \mathbb{C}. 
\end{equation}    

\begin{exmp}
As a simple example illustrating that $ B^J_{\mathfrak{X}} $ may be a spanning set with linearly dependent elements (as opposed to a basis), suppose that $ ( \mathcal{M}, g ) $ is the circle equipped with the canonical arclength metric, normalized such that  $ \mu( \mathcal{M} ) = 1 $. Then, an orthonormal basis of $H$ consisting of Laplace-Beltrami eigenfunctions is given by 
\begin{displaymath}
  \phi_i( \theta ) = 
  \begin{cases}
    \cos( i \theta / 2 ), & \text{$i$ even}, \\
    \sin( ( i + 1 ) \theta / 2 ), & \text{$i$ odd},
    \end{cases}
\end{displaymath} 
where $ \theta \in ( 0, 2 \pi ) $ is a canonical angle coordinate. Note that the coordinate basis vector field $ \frac{ \partial\ }{\partial_\theta} $ extends to a globally defined harmonic vector field on $ \mathcal{M} $, satisfying $\Delta \frac{ \partial \ }{\partial_\theta}^\flat = 0$. In this coordinate system, the metric and dual metric are given by $ g = g_{11} \, d\theta \otimes d\theta$, $ \eta = g^{\prime 11} \frac{\partial\ }{\partial\theta} \otimes \frac{\partial\ }{\partial\theta} $, where $ g_{11} = 1 / g^{\prime 11} = 1 / ( 2 \pi )^2 $, and we have 
\begin{align*}
  \grad \phi_i &= \eta( d\phi_i, \cdot ) =  
  \begin{cases}
    - 2 i \pi^2  \phi_{i-1} \frac{\partial\ }{\partial\theta}, & \text{$i$ even,}\\
    2 ( i + 1 ) \pi^2  \phi_{i+1}\frac{\partial\ }{\partial\theta}, & \text{$i$ odd}, 
  \end{cases}
  \\ 
  b_{ij} &= 
  \begin{cases}
      - 2 j \pi^2  \phi_i \phi_{j-1} \frac{\partial\ }{\partial\theta}, & \text{$j$ even,}\\
      2 ( j + 1 ) \pi^2 \phi_i \phi_{j+1} \frac{\partial\ }{\partial\theta}, & \text{$j$ odd}.
  \end{cases}
\end{align*}
It therefore follows from standard trigonometric identities that for any odd $ i \geq 1 $,  
\begin{displaymath}
   \frac{ b_{i+1,i} - b_{i,i+1} }{ 2 ( i + 1 ) \pi^2 } = \frac{\partial\ }{\partial\theta},
\end{displaymath}   
with an analogous relationship holding for $ i $ even. This shows that for $ J \geq 2 $, $ B_J $ contains linearly dependent elements. On the other hand, for $ J = 1 $, $ B_1 $ fails to be a spanning set as the harmonic vector field $ \frac{\partial\ }{\partial\theta} $ does not lie in its span. 
\end{exmp}

\begin{remark}
    The circle example above might suggest that our representation of vector fields through linear combinations of elements of $ B_J $ is highly inefficient, since, after all, one could define $ \xi_j = \phi_j \partial_\theta $, and $ \{ \xi_j \}_{j=0}^\infty $ would be an orthonormal basis of $ H_\mathfrak{X} $. However, such a construction implicitly makes use of a special property of the circle, namely that it is a parallelizable manifold. Equivalently, as a $ C^\infty(\mathcal{M}) $-module, the space $ \mathfrak{X} $ of  smooth vector fields is free; that is, it contains a set $ \{ u_1, \ldots, u_d \} $ of $ d = \dim( \mathcal{M} ) $ nowhere-vanishing linearly independent elements. Any such set would be a basis of  $ \mathfrak{ X} $, meaning that for every $ v \in \mathfrak{ X}  $ there  would exist unique smooth function  $ f_1, \ldots, f_d \in C^\infty( \mathcal{M}) $ such that $ v = \sum_{j=1}^d f_j u_j $. In general, for non-parallelizable manifolds (e.g., the 2-sphere), $ \mathfrak X $ does not have a basis, so any spanning set of $ H_\mathfrak{ X} $, such as $ B_J $, that makes use of a generating set of $ \mathfrak{ X } $ will necessarily be overcomplete. 
\end{remark}
  
We continue by stating a number of useful properties of the $ b_{ij} $ fields and their antisymmetric analogs, $ \widehat b_{ij} $. Many of these properties are also listed in Tables~\ref{fig2} and \ref{fig3}. In what follows, all equalities involving infinite sums hold in an $L^2$ sense.
\begin{enumerate}
  \item \emph{Relationship between antisymmetric and nonsymmetric frame elements.} Using the Leibniz rule, we compute
    \begin{displaymath}
      \phi_i \grad \phi_j + \phi_j \grad \phi_i = \grad( \phi_i \phi_j ) = \sum_{k=0}^\infty c_{ijk} \grad \phi_k = \sum_{k=1}^\infty c_{ijk} \widehat b_{0k},
     \end{displaymath}
     where the last equality follows from the fact that $\phi_0$ is a constant equal to 1. It therefore follows that
    \begin{equation}
    \label{eqSymAntisym}
      b_{ij} = \frac{\phi_i \grad \phi_j + \phi_j \grad \phi_i}{2} + \frac{\phi_i \grad \phi_j - \phi_j \grad \phi_i}{2} = \sum_{k=1}^\infty c_{ijk} \widehat b_{0k} + \widehat b_{ij}.
    \end{equation}
\item \emph{Riemannian inner products.} Lemma~\ref{lemmaInnerProd}, \eqref{eqGFourier}, and \eqref{eqCIJK} lead to the following expressions for the Riemannian inner products between the frame elements:
      \begin{equation}
        \label{eqGIJKL_Riemm}
        g( b_{ij}, b_{kl} ) = \phi_i \phi_k g( \grad \phi_j, \grad \phi_l ) = \sum_{n=0}^\infty \phi_i \phi_k \phi_n g_{njl} =  \sum_{m,n,p=0}^\infty c_{ikm}c_{mnp} g_{njl} \phi_p.
        \end{equation}
    Using the above, we can also compute the Riemannian inner products between the antisymmetric vector fields, i.e.,
    \begin{equation}
        \label{eqGIJKL_Antisym}
        g( \widehat b_{ij}, \widehat b_{kl} )  = g(b_{ij}, b_{kl}) + g(b_{ji},b_{lk}) - g(b_{ij}, b_{lk} ) - g( b_{ji}, b_{kl} ).
    \end{equation}
The Riemannian inner products between the frame elements $\widetilde b_{ij} \in \widetilde B_{\mathfrak{X}}$ are given by eigenvalue-dependent rescalings of those in~\eqref{eqGIJKL_Riemm}.

\item \emph{Hodge inner products and matrix elements of the Gramm operators}. Using the spectral representation of the pointwise Riemannian inner products in~\eqref{eqGIJKL_Riemm} and the fact that $c_{mn0} = \delta_{mn}$, we can compute the Hodge inner products 
        \begin{displaymath}
            G_{ijkl} = \langle b_{ij}, b_{kl} \rangle_{H_\mathfrak{X}} = \langle \phi_0, g( b_{ij}, g_{kl} ) \rangle_{H} = \sum_{n=0}^\infty c_{ikn} g_{njl}, 
    \end{displaymath}
    as in~\eqref{eqGIJKL_Hodge}. In addition, we have
    \begin{displaymath}
        \hat G_{ijkl} \equiv \langle \widehat b_{ij}, \widehat b_{kl} \rangle_{H_\mathfrak{X}} = G_{ijkl} + G_{jilk} - G_{ijlk} - G_{jikl}. 
    \end{displaymath}
    Given now any ordering $ b_{i_0 j_0}, b_{i_1 j_1}, \ldots $ of the frame elements in $B^J_\mathfrak{X}$, where $ j_n \leq J $, the above can be used to compute the matrix elements of the corresponding Gramm operator $ G : \ell^2 \to \ell^2 $, viz. 
\begin{displaymath}
    G_{mn} = \langle e_m, G e_n \rangle_{\ell^2} = \langle b_{i_m j_m}, b_{i_n j_n} \rangle_{H_{\mathfrak{X}}} = G_{i_m j_m k_n l_n}.  
\end{displaymath}
The analogous expressions for the Hodge inner products and Gramm matrix elements associated with the $\widetilde B_\mathfrak{X}$ frame are given by appropriate rescalings of the $G_{ijkl}$. 

\end{enumerate}

Additional formulas for the SEC representation of vector fields are listed in Table~\ref{fig2}. The next few results are for the 1-form fields in Theorems~\ref{thmFrame} and~\ref{thmFrameSob}. They will be employed in our proof of Theorem~\ref{thmFrameSob} and in the construction of Galerkin schemes for the Laplacian on 1-forms in Sections~\ref{secFrameSobProof} and~\ref{secGalerkin} ahead, respectively.
\begin{enumerate}
  \item \emph{Exterior derivative and codifferential.} It follows from the Leibniz rule for the exterior derivative in~\eqref{eqLeibniz} and the definition of the codifferential in~\eqref{eqCoDiff} that
    \begin{equation}
      db^{ij} = db^i \wedge db^j, \quad \delta b^{ij} = - \eta(db^i, db^j) + \lambda_j b^i b^j.
      \label{eqDBIJ}
    \end{equation}
    Similarly, we have
    \begin{displaymath}
      d \widehat b^{ij} = 2 db^i \wedge db^j, \quad \delta d\widehat b^{ij} = (\lambda_j - \lambda_i) b^i b^j.
    \end{displaymath}
    Observe, in particular, that if $\lambda_i = \lambda_j$ (i.e., $b^i $ and $b^j$ lie in the same eigenspace of the Laplacian), $\widehat b^{ij}$ is co-closed, $\delta \widehat b^{ij} = 0$. For additional details on these formulas see~\eqref{eqAppCodiff}. 
\item \emph{Riemannian inner products.} Since  the $b^{ij}$ and $b_{ij}$ are Riemannian duals to each other, we have $\eta(b^{ij}, b^{kl}) = g(b_{ij}, b_{kl})$, and the latter can be determined from~\eqref{eqGIJKL_Riemm}. An alternative derivation of this result, directly utilizing the product rule for the Laplacian on functions, can be found in Appendix \ref{appLapl}. The Riemannian inner products between the antisymmetric 1-forms $\widehat b^{ij}$ can be computed analogously to~\eqref{eqGIJKL_Antisym}. The Riemannian inner products between the $k$-form frame elements $b^{ij_1\cdots j_k}$ in $B^J_k$, or the antisymmetric $k$-forms $\widehat b^{ij_1\cdots j_k}_k$, can be evaluated by computing determinants of $k\times k$ matrices of Riemannian inner products between $db^j$'s; see Appendix \ref{appKForms} for further details.

\item \emph{Riemannian inner products between exterior derivatives and codifferentials of the frame elements.} In order to perform operations on the frame elements $b^{ij}_1 \in B^J_{1,1} $, or $\widetilde b^{ij}_1 \in \widetilde B_{1,1}$, with the differential operators of the exterior calculus, we need expressions for Riemannian inner products between exterior derivatives and codifferentials such as $g(db^{ij}_1,db^{kl}_1)$ and $g(\delta b^{ij}_1,\delta b^{kl}_1)$. Closed-form expressions for such inner products based on the Laplacian eigenvalues $\lambda_i$ and the corresponding triple-product coefficients $c_{ijk}$ can be derived using~\eqref{eqDBIJ}; explicit results and derivations can be found in Table~\ref{fig3} and Appnedix \ref{appLapl}. Analogous inner product formulas can also be derived hierarchically for the higher-order frame elements (see Appendix \ref{appKForms}), although currently we do not have closed-form expressions for direct evaluation of pairwise inner products between the $\delta b^{ij_1\cdots j_k}$ at arbitrary $ k > 1 $. The inner product relationships between $b^{ij_1\cdots j_k}$ and their exterior derivatives and codifferentials would be needed to perform operations with the frames for Sobolev spaces of $ k$-forms, $ k > 1 $, in Conjecture~\ref{conjFrameSob}.
\item \emph{Hodge and Sobolev inner products.} These can be computed as described above for  vector fields, using the additional results on Riemannian inner products between exterior derivatives and codifferentials outlined above. Specific formulas can be found in Table~\ref{fig3} and Appendix \ref{appLapl}. Note that the Sobolev inner products between the frame elements for $H^1_1$ and the corresponding Dirichlet forms will be used in our Galerkin approximation scheme for the 1-Laplacian in Section~\ref{secGalerkin} ahead.
\end{enumerate}

\subsection{\label{secFrameProof}Proof of Theorem~\ref{thmFrame}}

We will prove the theorem by establishing the upper and lower frame conditions in Definition~\ref{defFrame} for $ B^J_\mathfrak{X} $ and $B^J_k$, assuming that $ J $ is large-enough so that Lemma~\ref{lemmaSpanning} holds. We begin from $B^J_\mathfrak{X}$.

Given any $ v \in H_\mathfrak{X} $, we have
\begin{align*}
  \sum_{i=0}^\infty \sum_{j=1}^J \lvert \langle b_{ij}, v \rangle_{H_\mathfrak{X}} \rvert^2 &=   \sum_{i=0}^\infty \sum_{j=1}^J \lvert \langle \phi_i \grad \phi_j, v \rangle_{H_\mathfrak{X}} \rvert^2 = \sum_{i=0}^\infty \sum_{j=1}^J \lvert \langle \phi_i, v( \phi_j ) \rangle_H \rvert^2 \\
  &= \sum_{j=1}^J \lVert v( \phi_j ) \rVert^2_H = \sum_{j=1}^J \lVert g( \grad \phi_j, v ) \rVert^2_H \\
  &\leq \sum_{j=1}^J \lVert \lvert g( \grad \phi_j, v ) \rvert \rVert^2_H  \leq \sum_{j=1}^J \left \lVert \sqrt{ g( \grad \phi_j, \grad \phi_j ) g( v, v ) } \right \rVert_H^2  \\
  & \leq  \sum_{j=1}^J \lVert \grad \phi_j \rVert^2_{L^\infty_\mathfrak{X}}  \left \lVert \sqrt{ g( v, v ) } \right \rVert_H^2 
  =  \sum_{j=1}^J \lVert \grad \phi_j \rVert^2_{L^\infty_\mathfrak{X}} \lVert v \rVert_{H_\mathfrak{X}}^2,
\end{align*}  
so that the upper frame condition holds with $ C_2 = \sum_{j=1}^J \lVert \grad \phi_j \rVert^2_{L^\infty_\mathfrak{X}} $. Note that the fact that $ J $ is finite is important in the derivation of this result. Next, to verify the lower frame condition, consider the $ J \times J $ Gramm matrix $\bm \Psi(x) $, for $x\in \mathcal M$, with elements
\begin{displaymath}
  \Psi_{ij}( x ) = g( \grad \phi_i, \grad \phi_j ) \rvert_x, 
\end{displaymath}
and note that because the $ \grad \phi_j \rvert_x $ span $ T_x \mathcal{M} $, that matrix has rank $ d $. Therefore, writing $ v = \sum_{j=1}^J f_j \grad \phi_j $, where the $ f_j $ are functions in $ H $ to be determined, the equation
\begin{displaymath}
  g( \grad \phi_i, v ) \rvert_x = \sum_{j=1}^J f_j( x ) \Psi_{ji}( x )
\end{displaymath}  
has a solution for $ \mu $-a.e.\ $ x \in X $ given by
\begin{displaymath}
  f_j( x ) = \sum_{k=1}^J \Psi_{jk}^+( x ) g( \grad \phi_k, v ) \rvert_x, 
\end{displaymath}
where $ \bm \Psi^+(x) = [\Psi^+_{jk}(x)] $ is the Moore-Penrose pseudoinverse of $ \bm \Psi(x) $. 

Observe now that $ \bm \Psi(x) $ can be expressed using a coordinate chart as $\bm \Psi(x) = \bm \Xi(x)^\top \bm g(x) \bm \Xi(x) $, where the matrices $ \bm \Xi(x) $ and $\bm g( x)$ are as in the proof of Lemma~\ref{lemmaSpanning}. Thus, since $\bm \Xi(x) $ has linearly independent rows and $\bm g(x)$ is invertible, we have
\begin{displaymath}
    \bm \Psi^+(x) = \bm \Xi^+(x) \bm g^{-1}(x) (\bm \Xi^+(x))^\top, \quad \bm \Xi^+(x ) = \bm \Xi( x )^\top (\bm \Xi(x) \bm \Xi(x)^\top)^{-1},
\end{displaymath}
where both $ \bm g^{-1}(x)$ and $ ( \bm \Xi( x ) \bm \Xi( x) )^{-1} $ depend smoothly on $x $ by compactness of $\mathcal{M}$ and smoothness of $ \bm g( x ) $ and $\bm \Xi(x)$, respectively. We therefore conclude that $\bm \Psi^+(x)$ depends smoothly on $x$, and thus that $ v $ admits an expansion of the form~\eqref{eqVExpansion} with
\begin{equation}
  \label{eqVExpansion2}
  c_{ij} = \langle \phi_i, f_j \rangle_H = \sum_{k=1}^J \langle \Psi_{jk}^+ b_{ik}, v \rangle_{H_\mathfrak{X}}. 
\end{equation}
We therefore obtain
\begin{align*}
  \lVert v \rVert^2_{H_\mathfrak{X}} &= \sum_{i=0}^\infty \sum_{j=1}^J \langle c_{ij} b_{ij}, v \rangle_{H_\mathfrak{X}} = \sum_{i=0}^\infty \sum_{j,k=1}^J \langle \Psi_{jk}^+ b_{ik}, v \rangle_{H_\mathfrak{X}}^* \langle b_{ij}, v \rangle_{H_\mathfrak{X}} \\
  &\leq \tilde C \sum_{i=0}^\infty \sum_{j,k=1}^J \lvert \langle b_{ik}, v \rangle_{H_\mathfrak{X}} \rvert \lvert \langle b_{ij}, v \rangle_{H_\mathfrak{X}} \rvert,
\end{align*}
where $ \tilde C = \max_{j,k \in \{ 1, \ldots, J \} } \lVert \Psi^+_{jk} \rVert_{L^\infty} $. Defining now the vectors $ \beta_i = ( \beta_{i1}, \ldots, \beta_{iJ} ) \in \mathbb{ R }^J $ with with $ \beta_{ij} = \lvert \langle b_{ij}, v \rangle_{H_\mathfrak{X}} \rvert $, it follows by equivalence of norms in finite-dimensional vector spaces that there exists a constant $ \hat C> 0$ such that  
\begin{displaymath}
  \sum_{j,k=1}^J \lvert \langle b_{ik}, v \rangle_{H_\mathfrak{X}} \rvert \lvert \langle b_{ij}, v \rangle_{H_\mathfrak{X}} \rvert = \lVert \beta_i \rVert^2_1 \leq \hat C \lVert \beta_i \rVert_2^2,
\end{displaymath} 
where $\lVert \cdot \rVert_p$ is the canonical $p$-norm on $\mathbb{R}^J$. 
This leads to
\begin{displaymath}
  \lVert v \rVert_\mathfrak{X}^2  \leq \tilde C \hat C \sum_{i=0}^\infty \lVert \beta_i \rVert^2_2 = \tilde C \hat C \sum_{i=0}^\infty \sum_{j=1}^J \lvert \langle b_{ij}, v \rangle_{H_\mathfrak{X}} \rvert^2,
\end{displaymath}
which proves the lower frame condition with $ C_1 = 1/ ( \tilde C \hat C ) $. We have thus established that $B_\mathfrak{X}^J$ is a frame of $H_\mathfrak{X}$, as claimed. 

Consider now the frame conditions for $\widetilde B_\mathfrak{X}$. Using Cauchy-Schwartz inequalities as above, we can conclude that
\begin{displaymath}
  \sum_{i=0}^\infty \sum_{j=1}^\infty \lvert \langle \widetilde b_{ij}, v \rangle_{H_\mathfrak{X}} \rvert^2 \leq \sum_{j=1}^\infty \frac{\lVert \grad \phi_j \rVert^2_{L^\infty_\mathfrak{X}}}{e^{\lambda_j}} \lVert v \rVert^2_{H_\mathfrak{X}}.
\end{displaymath}
To bound the infinite sum in the right-hand side, we use the following estimates for the $L^\infty$ norms of Laplace-Beltrami eigenfunctions and their gradients on smooth, closed Riemannian manifolds:
\begin{equation}
  \label{eqHormander}
  \lVert \phi_j \rVert_{L^\infty} \leq C \lambda_j^{(d-1)/4} \lVert \phi_j \rVert_{H}, \quad \lVert \grad \phi_j \rVert_{L^\infty_\mathfrak{X}} \leq \tilde C \lambda_j^{1/2} \lVert \phi_j \rVert_{L^\infty}, \quad C, \tilde C \geq 0, 
\end{equation}
which hold for $j \leq 1 $ and $j \geq 0 $, respectively. The former is a classical result due to H\"ormander \cite{Hormander68}; the latter was proved by Shi and Xu in \cite{ShiXu10}. Combining these results with the Weyl estimate for the asymptotic distribution of Laplace-Beltrami eigenvalues as $j \to \infty$,
\begin{equation}
  \label{eqWeyl}
  j = \hat C \lambda_j^{d/2} + o(\lambda_j^{(d-1)/2}), 
\end{equation}
we obtain
\begin{displaymath}
    \lVert \grad \phi_j \rVert_{L^\infty_{\mathfrak{X}}} \leq \check C \lambda_j^{(d+1)/(2d)} \lVert \phi_j \rVert_H,
\end{displaymath}
where $\hat C$ and $\check C$ in the last two equations are positive constants. Therefore, for any $l \geq 0 $ there exists a finite constant $C_l$ such that $e^{-\lambda_j}\lVert \grad \phi_j \rVert^2_{L^\infty_\mathfrak{X}} \leq C_l j^{-l}$. This implies that  $ C_2 = \sum_{j=0}^\infty e^{-\lambda_j} \lVert \grad \phi_j \rVert^2_{L^\infty_\mathfrak{X}}$ is finite, proving the upper frame bound. 

To verify the lower frame bound, start from any expansion of $v$ in the $B^J_\mathfrak{X}$ frame, 
\begin{displaymath}
  v = \sum_{i=0}^\infty \sum_{j=1}^{J} c_{ij} b_{ij},
\end{displaymath}
and compute 
\begin{displaymath}
  \lVert v \rVert^2_{\mathfrak{X}} \leq C \sum_{i=0}^\infty \sum_{j=1}^J \lvert \langle b_{ij}, v \rangle_{H_\mathfrak{X}} \rvert^2 \leq C e^{\lambda_J} \sum_{i=0}^\infty \sum_{j=1}^{J} \lvert \langle \widetilde b_{ij}, v \rangle_{H_\mathfrak{X}} \rvert^2 \leq C e^{\lambda_J} \sum_{i=0}^\infty \sum_{j=1}^\infty \lvert \langle \widetilde b_{ij}, v \rangle_{H_\mathfrak{X}} \rvert^2,
\end{displaymath}
where $C$ is a lower frame constant for $B^J_\mathfrak{X}$. This shows that the lower frame condition is satisfied for $C_1 = C e^{-\lambda_J}$, and we thus conclude that $\widetilde B_{\mathfrak{X}}$ is a frame. 

We now turn to the frame conditions for $B^J_k$ and $\widetilde B_k$. These conditions follow by similar arguments as those just made to establish the frame conditions for vector fields. 

First, we introduce for convenience an ordering $l \mapsto (j_1(l), \ldots, j_k(l) ) $ of the corresponding indices in $b^{ij_1\ldots j_k}$, where $ l $ is an integer ranging from 1 to $ J^k $, and define $ \alpha_l = d\phi^{j_1(l)} \wedge \cdots \wedge d\phi^{j_k(l)} $. Then, for any $ \omega \in H_k$, we have
\begin{align*}
  \sum_{i=0}^{\infty} \sum_{j_1,\ldots,j_k=1}^J \lvert \langle b^{ij_1\cdots j_k}, \omega \rangle_{H_k} \rvert^2 &= \sum_{i=0}^\infty \sum_{l=1}^{J^k} \lvert \langle \phi_i \alpha_l, \omega \rangle_{H_k} \rvert^2 \\   
  &= \sum_{i=0}^{\infty} \sum_{l=1}^{J^k} \lvert \langle \phi_i, \eta( \alpha_l, \omega ) \rangle_{H} \rvert^2 = \sum_{l=1}^{J^k} \lVert \eta(\alpha_l, \omega) \rVert^2_H \\
  & \leq \sum_{l=1}^{J^k} \left \lVert \sqrt{\eta(\alpha_l,\alpha_l)} \sqrt{\eta(\omega,\omega)} \right \rVert^2_H \leq \sum_{l=1}^{J^k} \lVert \alpha_l \lVert_{L^\infty_k}^2 \lVert \omega \rVert^2_{H_k},
\end{align*}
establishing the upper frame condition with $C_2=\sum_{l=1}^{J^k} \lVert \alpha_l \lVert_{L^\infty_k}^2 $. To verify the lower frame condition, we use \eqref{eqWExpansion} to expand $\omega = \sum_{i=0}^\infty \sum_{l=1}^{J^k}c_{il} \phi_i \alpha_l$, where the expansion coefficients $c_{il}$ can be chosen as (cf.~\eqref{eqVExpansion2})
\begin{displaymath}
c_{il} = \sum_{m=1}^{J^k} \langle \Psi_{lm}^+ \phi_i \alpha_m, \omega \rangle_{H_k}, 
\end{displaymath}
and in the above $\Psi_{lm}^+(x)$ are the elements of the pseudoinverse of the $J^k \times J^k$ Gramm matrix $\Psi_{lm}(x) = \eta(\alpha_l,\alpha_m) \rvert_x$ (these matrix elements depend smoothly on $x$ as in the case of the corresponding Gramm matrix for vector fields). The calculation to establish the lower frame bound for $\lVert \omega \rVert_{H_k}^2 = \sum_{i=0}^\infty \sum_{l=1}^{J^k} \langle c_{il} \phi_i \alpha_l, \omega \rangle_{H_k} $ then proceeds analogously to that in the case of vector fields, leading to the conclusion that $B^J_k$ is an $H_k$-frame. Similarly, that $\widetilde B_k$ is a frame follows analogously to the vector field case. This completes our proof of Theorem~\ref{thmFrame}.

\subsection{\label{secFrameSobProof}Proof of Theorem~\ref{thmFrameSob}}

Let $ \omega $ be an arbitrary 1-form field in $H^1_1$. We begin by stating two auxiliary results on the inner products between the exterior derivative (codifferential) of $\omega$ and the exterior derivative (codifferential) of the frame elements $ b^{ij}_1 $ and $\widetilde b^{ij}_1$. 

\begin{lemma}
    (i) There exist constants $U_J$ and $V_J$, independent of $\omega$, such that
  \begin{displaymath}
    \sum_{i=0}^\infty \sum_{j=1}^J \lvert \langle d b^{ij}_1, d \omega \rangle_{H_2} \rvert^2 \leq U_J \lVert \omega \rVert^2_{H^1_1}, \quad \sum_{i=0}^\infty \sum_{j=1}^J \lvert \langle \delta b^{ij}_1, \delta \omega \rangle_{H} \rvert^2 \leq  V_J \lVert \omega \rVert^2_{H^1_1}.
  \end{displaymath}
  Moreover, there exist a positive real number $\hat C $ and a positive integer $ q \geq 0 $, both independent of $J$, such that $ U_J $ and $ V_J $ are both bounded above by $\hat C J^q$. 

  (ii) There exist finite constants $\tilde U$ and $ \tilde V$, independent of $\omega$, such that
  \begin{displaymath}
    \sum_{i=0}^\infty \sum_{j=1}^\infty \lvert \langle d \widetilde b^{ij}_1, d \omega \rangle_{H_2} \rvert^2 \leq \widetilde U \lVert \omega \rVert^2_{H^1_1}, \quad \sum_{i=0}^\infty \sum_{j=1}^\infty \lvert \langle \widetilde \delta b^{ij}_1, \delta \omega \rangle_{H} \rvert^2 \leq  \tilde V \lVert \omega \rVert^2_{H^1_1}.
  \end{displaymath}
  \label{lemmaDBIJ}
\end{lemma}

A proof of this lemma will be given below. Assuming, for now, that it is valid, it leads to the following corollary:

\begin{corollary}
  \label{corEBIJ}
  The frame elements $b^{ij}_1$ and $\widetilde b^{ij}_1$ satisfy the bounds
  \begin{displaymath}
    \sum_{i=0}^\infty \sum_{j=1}^J \lvert E_{1,1}( b^{ij}_1,  \omega ) \rvert^2 \leq C \lVert \omega \rVert_{H^1_1}, \quad \sum_{i=0}^\infty \sum_{j=1}^\infty \lvert E_{1,1}( \widetilde b^{ij}_1, \omega ) \rvert^2 \leq \tilde C \lVert \omega \rVert_{H^1_1},
  \end{displaymath}
  where $E_{1,1}$ is the Dirichlet form from~\eqref{eqDirichletK}, and $C$ and $\tilde C $ are constants independent of $\omega$. 
\end{corollary}

\begin{proof}
    We first verify the claim for $b^{ij}_1$. By Lemma~\ref{lemmaDBIJ}, the sequences $(i,j) \mapsto \lvert \langle d b^{ij}_1, d \omega \rangle_{H_2} \rvert^2$ and $(i,j) \mapsto \lvert \langle \delta b^{ij}_1, \delta \omega \rangle_H \rvert^2$ are in $\ell^2$. Therefore, using the Cauchy-Schwartz inequality for $\ell^2$ and Lemma~\ref{lemmaDBIJ}(i), we obtain, 
  \begin{align*}
\sum_{i=0}^\infty \sum_{j=1}^J \lvert E_{1,1}( b^{ij}_1,  \omega ) \rvert^2 &= \sum_{i=0}^\infty \sum_{j=1}^J \lvert \langle d b^{ij}_1, d \omega \rangle_{H_2} + \langle \delta b^{ij}_1, \delta \omega \rangle_{H} \rvert^2 \\
    &\leq \sum_{i=0}^\infty \sum_{j=1}^J \lvert \langle db^{ij}_1, d \omega \rangle_{H_2} \rvert^2 + \sum_{i=0}^\infty \sum_{j=1}^J \lvert \langle \delta b^{ij}_1, \delta \omega \rangle_H \rvert^2 \\
    &\quad + 2  \sum_{i=0}^\infty \sum_{j=1}^J \lvert \langle db^{ij}_1, d \omega \rangle_{H_2} \langle \delta b^{ij}_1, \delta \omega \rangle_{H} \rvert \\
    & \leq \sum_{i=0}^\infty \sum_{j=1}^J \lvert \langle db^{ij}_1, d \omega \rangle_{H_2} \rvert^2 + \sum_{i=0}^\infty \sum_{j=1}^J \lvert \langle \delta b^{ij}_1, \delta \omega \rangle_H \rvert^2 \\
    & \quad + 2 \left( \sum_{i=0}^\infty \sum_{j=1}^J \lvert \langle d b^{ij}_1, d \omega \rangle_{H_2} \rvert^2 \right)^{1/2} \\
    & \quad \times \left( \sum_{k=0}^\infty \sum_{l=1}^J \lvert \langle \delta b^{kl}_1, \delta \omega \rangle_H \rangle \rvert^2 \right)^{1/2}\\
    & \leq \left( U_J^{1/2} + V_J^{1/2} \right)^2 \lVert \omega \rVert_{H^1_1}^2,
  \end{align*}
  and the first claim of the corollary follows with $C=\left( U_J^{1/2} + V_J^{1/2}\right)^2$. To verify the second claim, we proceed as above using Lemma~\ref{lemmaDBIJ}(ii) to derive the bound
  \begin{displaymath}
      \sum_{i=0}^\infty \sum_{j=1}^\infty \lvert E_{1,1}( \widetilde b^{ij}_1, \omega ) \rvert^2 \leq \tilde C \lVert \omega \rVert_{H^1_1}^2, \quad \tilde C = \left(\tilde U^{1/2} + \tilde V^{1/2} \right)^2. \qedhere
  \end{displaymath}
\end{proof}

We now return to the proof of Theorem 4.4.  By Corollary~\ref{corEBIJ}, the upper frame condition for $B^J_{1,1}$ follows from
\begin{align*}
  \sum_{i=0}^\infty \sum_{j=1}^J \lvert \langle b^{ij}_1, \omega \rangle_{H^1_1} \rvert^2 &= \sum_{i=0}^\infty \sum_{j=1}^J \lvert \langle b^{ij}_1, \omega \rangle_{H_1} + E_{1,1}( b^{ij}_1, \omega ) \rvert^2 \\
  & \leq \sum_{i=0}^\infty \sum_{j=1}^J \lvert \langle b^{ij}_1, \omega \rangle_{H_1} \rvert^2 + \sum_{i=0}^\infty \sum_{j=1}^J \lvert E_{1,1}( b^{ij}_1, \omega ) \rvert^2 \\
  & \quad + 2 \left( \sum_{i=0}^\infty \sum_{j=1}^J \lvert \langle b^{ij}_1, \omega \rangle_{H_1} \rvert^2 \right)^{1/2} \left( \sum_{k=0}^\infty \sum_{l=1}^J \lvert E_{1,1}( b^{kl}_1, \omega )\rvert^2 \right)^{1/2} \\
  & = \sum_{i=0}^\infty \sum_{j=1}^J \frac{\lvert \langle b^{ij}, \omega \rangle_{H_1} \rvert^2}{\lVert b^i \rVert^2_{H^1}} + \sum_{i=0}^\infty \sum_{j=1}^J \lvert E_{1,1}( b^{ij}_1, \omega ) \rvert^2 \\
  & \quad + 2 \left( \sum_{i=0}^\infty \sum_{j=1}^J \frac{ \lvert \langle b^{ij}, \omega \rangle_{H_1} \rvert^2}{\lVert b^i \rVert^2_{H^1}} \right)^{1/2} \left( \sum_{k=0}^\infty \sum_{l=1}^J \lvert E_{1,1}( b^{kl}_1, \omega )\rvert^2 \right)^{1/2} \\
  & \leq \sum_{i=0}^\infty \sum_{j=1}^J \lvert \langle b^{ij}, \omega \rangle_{H_1} \rvert^2 + \sum_{i=0}^\infty \sum_{j=1}^J \lvert E_{1,1}( b^{ij}_1, \omega ) \rvert^2 \\
  & \quad + 2 \left( \sum_{i=0}^\infty \sum_{j=1}^J \lvert \langle b^{ij}, \omega \rangle_{H_1^1} \rvert^2 \right)^{1/2} \left( \sum_{k=0}^\infty \sum_{l=1}^J \lvert E_{1,1}( b^{kl}_1, \omega )\rvert^2 \right)^{1/2} \\
  & \leq \left( C_2^{1/2} + C^{1/2} \right)^2 \lVert \omega \rVert^2_{H^1_1},
\end{align*}
where $C_2$ is an upper frame constant for $B^J_1$ and $ C$ the constant in Corollary~\ref{corEBIJ}. The upper frame condition for $\widetilde B^J_{1,1}$ follows similarly. 

What remains in the proof of the upper frame conditions in Theorem 4.4 is to verify Lemma~\ref{lemmaDBIJ}. For that, first observe that
\begin{align*}
  \langle db^{ij}, dv \rangle_{H_2} &= \langle db^i \wedge db^j, dv \rangle_{H_2} = \int_\mathcal{M} db^i \wedge db^j \wedge \star dv \\
  &= \int_\mathcal{M} db^i \wedge \star \star (db^j \wedge \star dv) = \langle db^i, \star (db^j \wedge \star dv) \rangle_{H_1},
\end{align*}
where we have used~\eqref{eqDBIJ} and~\eqref{eqHodgeStar2} in the first and third equalities, respectively. By the Hodge decomposition theorem, there exists a unique function $f_j \in H^1$, a unique 2-form $ \alpha_j \in H^1_2$, and a unique harmonic 1-form $\chi_j \in \mathcal{H}^1_0$ such that
\begin{displaymath}
  \star (db^j \wedge \star d\omega ) = df_j + \delta \alpha_j + \chi_j;
\end{displaymath}
as a result,
\begin{align*}
  \sum_{i=0}^\infty \lvert \langle db^i_1, \star (db^j \wedge \star d\omega )\rangle_{H_1} \rvert^2 &= \sum_{i=0}^\infty \lvert \langle db^i_1, df_j + \delta \alpha_j + \chi_j \rangle_{H_1} \rvert^2 = \sum_{i=0}^\infty \lvert \langle \delta d b^i_1, f_j \rangle_H \rvert^2 \\
  &= \sum_{i=0}^\infty \lambda_i \lvert \langle b^i, f_j \rangle_H \rvert^2 = E_1(f_j).  
\end{align*}
We therefore have,
\begin{align*}
  \sum_{j=1}^J \sum_{i=0}^\infty \lvert \langle db^{ij}_1, d\omega \rangle_{H_2}\rvert^2 &=  \sum_{i=0}^\infty \lvert \langle db^i_1, \star (db^j \wedge \star d\omega )\rangle_{H_1} \rvert^2 \\
  &=  \sum_{j=1}^J E_1( f_j ) = \sum_{j=1}^J \langle df_j, df_j \rangle_{H_1} \\
  &\leq \sum_{j=1}^J \lVert \star ( db^j \wedge \star d\omega ) \rVert^2_{H_1} = \sum_{j=1}^J \lVert db^j \wedge \star d\omega \rVert^2_{H_{d-1}} \\
  &= \sum_{j=1}^J \lVert \star \iota_{\grad b^j} \star \star d\omega \lVert^2_{H_{d-1}} = \sum_{j=1}^J \lVert \iota_{\grad b^j} d\omega \rVert^2_{H_1},
\end{align*}
where we have used~\eqref{eqCartan} to obtain the first equality in the second line. It then follows from Lemma~\ref{lemmaVBoundK} that
\begin{displaymath}
    \sum_{j=1}^J \sum_{i=0}^\infty \lvert \langle db^{ij}_1, d\omega \rangle_{H_2} \rvert^2 = \sum_{j=1}^J \lVert D_{\grad b^j} \omega \rVert^2_{H_1} \leq U_J \lVert \omega \rVert_{H^1_1}^2, \quad U_J = \sum_{j=1}^J\lVert D_{\grad b^j } \rVert^2, 
\end{displaymath}
as claimed in part~(i) of the lemma. It can further be shown via local Cauchy-Schwartz inequalities as in the proof of Lemma~\ref{lemmaVBoundK} that the operator norms $ \lVert D_{\grad b^j} \rVert = \lVert D_{\grad \phi_j} \rVert $ can be bounded above by $ \tilde C \lVert \phi_j \rVert_{L^\infty}^{\tilde q}$ for some positive constants $\tilde C $ and $\tilde q $ that do not depend on $j$. This, in conjunction with the H\"ormander bound in~\eqref{eqHormander} implies that that there exists a positive constant $C_U$ and a positive integer $q_U$, both independent of $J$, such that
\begin{equation}
    U_J \leq C_U J^{q_U}.
    \label{eqCUBound}
\end{equation}

Moving on to the second claim of Lemma~\ref{lemmaDBIJ}(i), consider
\begin{displaymath}
  \sum_{j=1}^J\sum_{i=0}^\infty  \lvert \langle \delta b^{ij}_1, \delta \omega \rangle_{H} \rvert^2 = \sum_{i=0}^\infty \sum_{j=1}^J \lvert \langle - \eta(db^i_1, db^j) + \lambda_j b^i_1 b^j, \delta \omega \rangle_H \rvert^2,
\end{displaymath}
where we have used the expression for $\delta b^{ij}$ in~\eqref{eqDBIJ}. Expanding
\begin{displaymath}
  \delta \omega \, db^j = d f_j + \delta \alpha_j + \chi_j,
\end{displaymath}
where $ f_j \in H^1 $, $ \alpha_j \in H^1_2 $, and $ \chi_j \in \mathcal{ H }^1_0$ are unique, we compute,
\begin{align}
  \nonumber\sum_{j=1}^J \sum_{i=0}^\infty \lvert \langle \eta( db^i_1, db^j ), \delta \omega \rangle_H \rvert^2 &= \sum_{j=1}^J \sum_{i=0}^\infty \left\lvert \int_\mathcal{M} db^i_1 \wedge \star d b^j \, \delta \omega \right \rvert^2 \\
  \nonumber &= \sum_{j=1}^J \sum_{i=0}^\infty \lvert \langle db^i_1, \delta \omega \, db^j \rangle_{H_1} \rvert^2 = \sum_{j=1}^J E_1( f_j ) \\
  \nonumber& \leq \sum_{j=1}^J \lVert \delta \omega \, db^j \rVert^2_{H_1} = \sum_{j=1}^J \left \lVert \sqrt{\eta(db^j, db^j)} \delta \omega \right \rVert^2_{H} \\
  \label{eqDelV1} &\leq \sum_{j=1}^J \lVert \grad \phi_j \rVert_{L^\infty_\mathfrak{X}}^2 \lVert \delta \omega \rVert^2_H 
  \leq \sum_{j=1}^J \lVert \grad \phi_j \rVert_{L^\infty_\mathfrak{X}}^2 \lVert \omega \rVert^2_{H^1_1}.
\end{align}
Moreover, we have
\begin{align}
    \nonumber \sum_{j=1}^J \sum_{i=0}^\infty \lvert \langle \lambda_j b^i_1 b^j, \delta \omega \rangle_H \rvert^2 &= \sum_{j=1}^J \sum_{i=0}^\infty\frac{\lambda_j^2}{ \lVert b^i \rVert_{H^1}^2}  \lvert \langle  b^i b^j, \delta \omega \rangle_H \rvert^2 \\
    \nonumber & \leq \sum_{j=1}^J \sum_{i=0}^\infty \lambda_j^2 \lvert \langle b^i, b^j \, \delta \omega \rangle_H \rvert^2 = \sum_{j=1}^J \lambda_j^2 \lVert b^j \, \delta \omega \rVert^2_H \\
    & \leq \sum_{j=1}^J \lambda_j^2 \lVert b^j \lVert_{L^\infty}^2 \lVert \omega \rVert_{H^1_1}^2.
  \label{eqDelV2}
\end{align}
Equations~\eqref{eqDelV1} and~\eqref{eqDelV2} imply that the sequences $(i,j) \mapsto \lvert \langle \eta(db^i_1, db^j ), \delta \omega \rangle_H \rvert $ and $(i,j) \mapsto \lvert \langle \lambda_j b^i_1 b^j, \delta \omega \rangle_H \rvert$ are both in $\ell^2$. Therefore, by the Cauchy-Schwarz inequality for that space we can conclude that
\begin{align*}
  \sum_{j=1}^J \sum_{i=0}^\infty \lvert \langle \delta b^{ij}_1, \delta \omega \rangle_H \rvert^2 & = \sum_{i=0}^\infty \sum_{j=1}^J \lvert \langle - \eta(db^i_1, db^j) + \lambda_j b^i_1 b^j, \delta \omega \rangle_H \rvert^2 \\
  & \leq \sum_{j=1}^J \sum_{i=0}^\infty \lvert \langle \eta( db^i_1, db^j ), \delta \omega \rangle_H \rvert^2 + \sum_{j=1}^J \sum_{i=0}^\infty \lvert \langle \lambda_j b^i_1 b^j, \delta \omega \rangle_H \rvert^2 \\ 
  & \quad + 2 \sum_{j=1}^J \sum_{i=0}^{\infty} \lvert \langle \eta(db^i_1,db^j), \delta\omega \rangle_H \langle \lambda_j b_1^i b^j, \delta \omega \rangle_H \lvert \\
  & \leq \sum_{j=1}^J \sum_{i=0}^\infty \lvert \langle \eta( db^i_1, db^j ), \delta \omega \rangle_H \rvert^2 + \sum_{j=1}^J \sum_{i=0}^\infty \lvert \langle \lambda_j b^i_1 b^j, \delta \omega \rangle_H \rvert^2 \\ 
  & \quad + 2 \left( \sum_{j=1}^J \sum_{i=0}^\infty \lvert \langle \eta( db^i_1, db^j ), \delta \omega \rangle_H \rvert^2  \right)^{1/2} \\
  & \quad \times \left( \sum_{l=1}^J \sum_{k=0}^\infty \lvert \langle \lambda_l b^k_1 b^l, \delta \omega \rangle_H \rvert^2  \right)^{1/2}\\ 
  & \leq V_J \lVert \omega \rVert_{H^1_1}, 
\end{align*}
where 
\begin{displaymath}
    V_J^{1/2} =  \left( \sum_{j=1}^J \lVert \grad \phi_j \rVert^2_{L^\infty_\mathfrak{X}} \right)^{1/2} + \left(  \sum_{j=0}^J \lambda_j^2 \lVert \phi_j \rVert^2_{L^\infty} \right)^{1/2}.
\end{displaymath}
This establishes the existence of the $\omega$-independent constants $V_J$ claimed in the lemma. Invoking $L^\infty$ and Weyl bounds as in the case of $U_J$, we can also deduce that there exist a real number $C_V$ and a positive integer $q_V$ such that 
\begin{equation}
    \label{eqCVBound}
    V_J \leq C_V J^{q_V}.
\end{equation}
Combining~\eqref{eqCUBound} and~\eqref{eqCVBound} leads to $ U_J \leq \hat C J^q$ and $V_J \leq \hat C J^q$ with $ \hat C= \max\{ C_U, C_V\}$ and $ q = \max \{ q_U, q_V \}$. This completes our proof of Lemma~\ref{lemmaDBIJ}(i) and thus the upper frame condition for $B^J_{1,1}$. 

To prove Lemma~\ref{lemmaDBIJ}(ii), we proceed as above to establish that
\begin{displaymath}
    \sum_{j=1}^k \sum_{i=0}^\infty \lvert \langle d \widetilde b^{ij}_1, d \omega \rangle_{H_2} \rvert^2 \leq \tilde U_k \lVert \omega \rVert_{H^1_1}^2, \quad \tilde U_k = \sum_{j=1}^k e^{-\lambda_j} \lVert D_{\grad b^j} \rVert^2, 
\end{displaymath}
and
\begin{gather*}
    \sum_{j=1}^k \sum_{i=0}^\infty \lvert \langle \delta \widetilde b^{ij}_1, \delta \omega \rangle_H \rvert^2 \leq \tilde V_k \lVert \omega \rVert_{H^1_1}^2, \\ 
    \tilde V_k^{1/2} =  \left( \sum_{j=1}^k e^{-\lambda_j}\lVert \grad \phi_j \rVert^2_{L^\infty_\mathfrak{X}} \right)^{1/2} + \left(  \sum_{j=0}^k e^{-\lambda_j}\lambda_j^2 \lVert \phi_j \rVert^2_{L^\infty} \right)^{1/2}.
\end{gather*}
The $L^\infty$ and Weyl estimates in~\eqref{eqHormander} and~\eqref{eqWeyl}, respectively, then again imply that $ \tilde U = \lim_{k\to\infty} \tilde U_k $ and $\tilde V = \lim_{k\to\infty} \tilde V_k$ is finite, proving Lemma~\ref{lemmaDBIJ}(ii), and completing our proof of the upper frame condition for $\widetilde B_{1,1}$. 

Next, to verify the lower frame conditions, we express $ \omega $ as a linear combination $ \omega = \sum_{j=1}^J f_j \, d\phi_j $,
where $f_1, \ldots, f_J $ are $H^1$ functions satisfying
\begin{displaymath}
  f_j(x) = \sum_{k=1}^J \Psi_{jk}^+(x) \eta(d\phi_k, \omega)
\end{displaymath}
for $\mu$-a.e.\ $x\in \mathcal{M}$, and $\Psi_{jk}^+(x)$ are the elements of the Moore-Penrose pseudoinverse of the $J\times J$ Gramm matrix 
\begin{displaymath}
    \Psi_{jk}(x) = \eta(d\phi_i, d\phi_j) = g(\grad \phi_i, \grad \phi_j)
\end{displaymath}
from Section~\ref{secFrameProof}. Note that the existence of such an expansion for $\omega$ follows from the fact that $\{ d\phi_1, \ldots, d\phi_J \} $ is a generating set of the space of smooth 1-forms $\Omega^1$, and the latter is dense in $H^1_1$. The $f_j$  functions can be expanded in the $\{ \phi_i^{(1)} \}_{i=0}^\infty $ basis of $H^1 $ from~\eqref{eqPhiHP}, viz.
\begin{displaymath}
    f_j = \sum_{i=0}^\infty c_{ij} \phi_i^{(1)}, \quad c_{ij} = \langle \phi^{(1)}_i, f_j \rangle_{H^1} = \sum_{k=1}^J \langle \phi_i^{(1)}, \Psi^+_{jk} \eta(d\phi_k, \omega) \rangle_{H^1}.
\end{displaymath}
Therefore, setting $b^i = \phi_i$ and $b^i_1 = \phi_i^{(1)}$ per our notational convention for frame elements, we obtain
\begin{align}
  \nonumber \lVert \omega \rVert^2_{H^1_1} &= \langle \omega, \omega \rangle_{H^1_1} = \sum_{i=0}^\infty \sum_{j,k=1}^J \langle b^i_1, \Psi_{jk}^+ \eta(db^k, \omega ) \rangle_{H^1} \langle b^{ij}_1, \omega \rangle_{H^1_1} \\
  \nonumber &\leq \sum_{i=0}^\infty \sum_{j,k=1}^J \lvert \langle b^i_1, \Psi_{jk}^+, \eta(db^k, \omega) \rangle_{H^1} \langle b_1^{ij}, \omega \rangle_{H^1_1} \rvert \\
  \label{eqLBound1}& \leq \left( \sum_{i=0}^\infty \sum_{j,k=1}^J \lvert \langle b^i_1, \Psi_{jk}^+ \eta(db^k, \omega) \rangle_{H^1} \rvert^2 \right)^{1/2} \left( \sum_{n=0}^\infty \sum_{p=1}^J \lvert \langle b^{np}_1, \omega \rangle_{H^1_1} \lvert^2 \right)^{1/2}.
\end{align}
Note that to arrive at the inequality in the last line we used the $\ell^2$ Cauchy-Schwartz inequality on the sequences $(i,j) \mapsto \langle b_1^i, \Psi_{jk}^+ \eta(db^k,\omega) \rangle_{H^1}$ and $(i,j) \mapsto \langle b_1^{ij}, \omega \rangle_{H_1^1}$, both of which can be verified to indeed lie in that space. We now proceed to bound the first term in the last line.

First, since $\{ b_1^i \}_{i=0}^\infty$ is an orthonormal basis of $H^1$, we have
\begin{displaymath}
  \sum_{i=0}^\infty \sum_{j,k=1}^J \lvert \langle b_1^i, \Psi_{jk}^+ \eta(db^k,\omega) \rangle_{H^1} \rvert^2 = \sum_{j,k=1}^J \lVert \Psi_{jk}^+ \eta(db^k,\omega) \rVert^2_{H^1}.
\end{displaymath}
Moreover, observe that for any $f\in H^1 $ and $h \in C^\infty(\mathcal{M})$, $\lVert h f \rVert_{H^1}$ can be bounded above by $\tilde C \lVert f \rVert_{H^1}$, where $\tilde C^2 $ is a polynomial function of $\lVert h \rVert_{L^\infty}$ and $\lVert \grad h \rVert_{L^\infty_\mathfrak{X}}$. This implies that there exists a constant $C$ such that
\begin{align}
  \nonumber\sum_{j,k=1}^J \lVert \Psi_{jk}^+ \eta(db^k, \omega) \rVert^2_{H^1} &\leq C \sum_{k=1}^J \lVert \eta(db_k, \omega) \rVert^2_{H^1} \\
  \label{eqLBound2} &= C \sum_{k=1}^J \left( \lVert \eta(db^k, \omega) \rVert_H^2 + \lVert d\eta(db^k,\omega) \rVert^2_{H_1} \right).
\end{align}
In the above, the term $\lVert \eta(db^k,\omega) \rVert_H^2$ can be bounded above by $\lVert \grad b^k \rVert^2_{L^\infty_\mathfrak{X}} \lVert \omega \rVert_{H^1_1}$ using local Cauchy-Schwartz inequalities and the fact that $\lVert \cdot \rVert_{H_1} \leq \lVert \cdot \rVert_{H_1^1}$. To bound $\lVert d\eta(db^k,\omega) \rVert_{H_1}$, we use~\eqref{eqCartan} to write down
\begin{displaymath}
  d\eta(db^k,\omega) = d \iota_{\grad b^k} \omega.
\end{displaymath}
It follows from Lemma~\ref{lemmaVBoundK} that there exists a constant $\hat C$ such that
\begin{equation}
  \label{eqLBound3}
  \lVert d\eta(db^k,\omega) \rVert \leq \hat C \lVert \omega \rVert_{H^1_1}.
\end{equation}

Combining~\eqref{eqLBound1}--\eqref{eqLBound3}, we conclude that there exists a constant $\bar C$ such that
\begin{displaymath}
  \lVert \omega \rVert^2_{H^1_1} \leq \bar C \lVert \omega \rVert_{H^1_1} \left( \sum_{i=0}^\infty \sum_{j=1}^J \lvert \langle b^{ij}_1, \omega \rangle_{H^1_1} \rvert^2 \right)^{1/2},
\end{displaymath}
and thus the lower frame condition for $B^J_{1,1}$ holds with $C_1 = 1/ \bar C^2$. To verify the lower frame condition for $\widetilde B_{1,1}$, we use the results just established to compute
\begin{align*}
    \lVert \omega \rVert^2_{H^1_1} &\leq C_1 \sum_{i=0}^\infty \sum_{j=1}^J \lvert \langle b^{ij}_1, \omega \rangle_{H^1_1} \rvert^2 \leq C_1 e^{\lambda_J} \sum_{i=0}^\infty \sum_{j=1}^J \lvert \langle \widetilde b^{ij}_1, \omega \rangle_{H^1_1} \rvert^2\\
    & \leq C_1 e^{\lambda_J} \sum_{i=0}^\infty \sum_{j=1}^\infty \lvert \langle \widetilde b^{ij}_1, \omega \rangle_{H^1_1} \rvert^2.
\end{align*}
This proves the lower frame condition for $\widetilde B_{1,1}$, and completes our proof of Theorem~\ref{thmFrameSob}.

\section{\label{secSECRep}Frame and operator representations of vector fields}

As described in Section~\ref{vectorOverview}, the SEC is based on alternative representations of vector fields with respect to frames, or as operators on functions. In this section, we make these notions precise, and further examine the convergence properties of finite-rank analogs of these representations. 

Unless otherwise stated, throughout this section, $T: H_{\mathfrak{X}} \to \ell^2$, $T^*: \ell^2 \to H_{\mathfrak{X}}$, $S = T^* T : H_{\mathfrak{X}}\to H_{\mathfrak{X}} $, and $G = T T^*: \ell^2 \to \ell^2$ will be the analysis, synthesis, frame, and Gramm operators, respectively, associated with one of the frames for the $H_{\mathfrak{X}}= L^2_\mathfrak{X} $ space of vector fields from Theorem~\ref{thmFrame}. We also let  $T': H_{\mathfrak{X}} \to \ell^2$, $T^{\prime*}: \ell^2 \to H_{\mathfrak{X}}$, $S' = T^{\prime *} T' : H_{\mathfrak{X}}\to H_{\mathfrak{X}} $, and $G'= T' T^{\prime *}: \ell^2 \to \ell^2$ be the corresponding operators for the dual frame. For notational simplicity, we use the symbols $\alpha_k$ and $\alpha'_k$ with $k\in \{  1, 2,  \ldots \} $ to represent the frame and dual frame elements, respectively. For example, in the case of the $B^J_\mathfrak{X}$ frames from Theorem~\ref{thmFrame}, we set  $\alpha_k = b_{p_kq_k}$, where $k \mapsto ( p_k, q_k) $ is any ordering of the $(i,j)$ indices in $b_{ij}$ with $ k \in \{ 0, 1, \ldots \} $. A convenient choice of such ordering is a lexicographical ordering, i.e., $ ( p_0, q_0 ) = ( 0, 1 ) $, \ldots, $(p_{J-1}, q_{J-1}) = ( 0, J )$, $(p_{J}, q_{J}) = (1,0)$, \ldots. We use the notation $ ( i, j ) \mapsto r_{ij} $ to represent the inverse of this ordering map.  Note that for any choice of frame from Theorem~\ref{thmFrame} we have $ \alpha_k = s_k \phi_{p_k} \grad \phi_{q_k} $, where $s_k = 1 $ in the case of the $ B^J_\mathfrak{X}$ frames and $s_k = e^{-\lambda_{q_k}/2}$ in the case of the $\widetilde B_\mathfrak{X} $ frame. We also let $\{ e_k \}_{k=0}^\infty$ be the canonical orthonormal basis of $\ell^2$, and $ \pi_l : \ell^2 \to \ell^2 $ the orthogonal projection operators with range $\tilde W_l = \spn \{ e_0, \ldots, e_{l-1} \}$. Moreover, the set $ \{ e_{ij} \}_{i,j=0}^\infty $, $ e_{ij} \in \mathcal{ B }_2(\ell^2) $, will be the canonical orthonormal basis of $\mathcal{ B }_2( \ell^2 ) $ with $ e_{ij} = e_i \langle e_j, \cdot \rangle_{\ell^2} $.   

In addition to the various frame operators acting on vector fields, we will consider the unitary Fourier operators $U : H \to \ell^2$, $ U^{(1)} : H^1 \to \ell^2$, and $ \tilde U : \mathcal{H} \to \ell^2 $ associated with the $\{ \phi_k \}_{k=0}^\infty$, $ \{ \phi_k^{(1)} \}_{k=0}^\infty$, and $ \{ \tilde \phi_k \}_{k=0}^\infty $ orthonormal bases of $H$, $H^1$, and $ \mathcal{H}$, respectively, where $ U f = ( \langle \phi_k, f \rangle_{H})_k$, $ U^{(1)} f = (\langle \phi_k^{(1)}, f \rangle_{H^1})_k$, and $ \tilde U f = (\langle \tilde \phi_k, f \rangle_{\mathcal{H}})_k$. As noted in Section~\ref{secFrameDef}, $U$, $U^{(1)}$, and $ \tilde U $ are special cases of analysis operators. Together, $U$ and $U^{(1)}$ induce the linear isometry $V: \mathcal{B}(H^1,H) \to \mathcal{B}(\ell^2)$ with $V A = U A U^{(1)*}$, while $ U $ and $ \tilde U $ induce the unitary map $ \tilde V : \mathcal{B}_2(\mathcal{H},H) \to \mathcal{B}_2(\ell^2) $ with $\tilde V A = U A \tilde U^*$ . 

\subsection{\label{secSECRepInfDim}SEC representations of vector fields and their correspondence}

Let $v$ be an arbitrary bounded vector field in $ H_{\mathfrak{X}} \cap L^\infty_{\mathfrak{X}}$. The SEC is based on the following three representations of $v$:
\begin{enumerate}
    \item \emph{Frame representation}, given by the sequence $ \hat v = T' v \in \ell^2$, such that $v = \sum_{k=0}^\infty \hat v_k \alpha_k$.
    \item \emph{Dual frame representation}, given by the sequence $\hat v' = T v \in \ell^2$, such that $v = \sum_{k=0}^\infty \hat v'_k \alpha'_k$.
    \item \emph{Operator representation}, given by the bounded operator $ L = W v \in \mathcal{B}(\ell^2) $,  $ W = \iota \circ V $, or the Hilbert-Schmidt operator $ \tilde L = \tilde W v \in \mathcal{B}_2(\ell^2)$, $ \tilde W = \iota_2 \circ \tilde V $, where $ \iota : L^\infty_\mathfrak{X} \to \mathcal{B}(H^1,H) $ and $ \iota_2: L^\infty_\mathfrak{X} \to \mathcal{B}_2(\mathcal{H},H) $ are the embeddings from Lemma~\ref{lemmaVBound}. When we wish to distinguish between $ L $ and $ \tilde L $ we will refer to the former as the \emph{bounded operator representation} and the latter as the \emph{Hilbert-Schmidt operator representation} of $v$. 
\end{enumerate}
Among these, the frame and dual frame representations only make use of the inner-product-space  structure of $ H_{\mathfrak{X}} \cap L^\infty_{\mathfrak{X}}$.  The operator representations make use of the relationship between $L^\infty_{\mathfrak{X}}$ and  bounded or Hilbert-Schmidt operators on functions, which is special to vector fields.  

As one might expect, the need to pass between these representations arises in a number of cases. On the one hand, many of the numerical procedures involving vector fields that one can envision being formulated via SEC produce output in the frame representation (that is, as linear combinations of frame elements), and in order to act with these vector fields on functions an operator representation is needed. For instance, the Galerkin approximation scheme for the eigenforms of the 1-Laplacian in Section~\ref{secGalerkin} yields approximate eigenforms as linear combinations of 1-form frame elements, and in order to visualize these forms we compute the pushforwards of the their vector field duals into data space. The pushforward operation requires evaluation of the action of these vector fields on the embedding map of the manifold, which we carry out using the operator representation. Conversely, one may be given $v$ in the operator representation (e.g., from data sampled along integral curves of the flow generated by $ v $ \cite{Giannakis19}), and then seek a frame representation for denoising and/or further use in a numerical procedure. 

The correspondence between the frame and operator representations of vector fields in SEC is illustrated with a commutative diagram in Fig.~\ref{figRep}. For the remainder of this section, we discuss aspects of this correspondence in the infinite-dimensional setting. Then, in Section~\ref{secSECFiniteDimRep} we examine finite-rank representations and their convergence properties. 

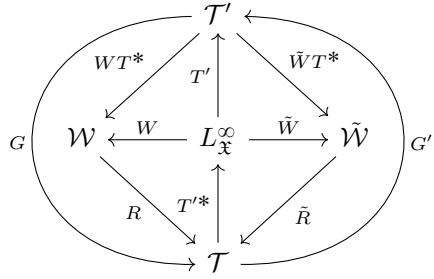
\begin{figure}
\begin{displaymath}
    \begin{tikzcd}[column sep=3em,row sep=3em]
        & \mathcal{T}' \arrow[swap]{dl}{WT^*} \arrow{dr}{\tilde WT^*} \arrow[bend right=90, looseness=2.2,swap]{dd}{G}\\
        \mathcal{W} \arrow[swap]{dr}{R} & L^\infty_\mathfrak{X} \arrow[swap]{l}{W} \arrow{r}{\tilde W} \arrow{u}{T'} & \tilde{\mathcal{W}} \arrow{dl}{\tilde R}\\
        & \mathcal{T}  \arrow{u}{T^{\prime *}} \arrow[bend right=90, looseness=2.2,swap]{uu}{G'}
    \end{tikzcd}
\end{displaymath}
    \caption{\label{figRep}Commutative diagram illustrating the frame, dual frame, bounded operator, and Hilbert-Schmidt operator representations of bounded vector fields in $L^\infty_\mathfrak{X}$. $\mathcal{T} \subset \ell^2 $ and $\mathcal{T}' \subset \ell^2 $ are the ranges of the analysis and dual analysis operators, $T$ and $T'$, respectively, restricted to $L^\infty_\mathfrak{X}$. $\mathcal{W} \subset \mathcal{B}(\ell^2) $ and $\tilde{\mathcal{W}} \subset \mathcal{B}_2(\ell^2) $ are the ranges of the operators $W$ and $\tilde W $ yielding the bounded operator and Hilbert-Schmidt operator representations vector fields in $L^\infty_\mathfrak{X}$, respectively. The operators $R$ and $\tilde R$ carry out the transformation from the bounded and Hilbert-Schmidt operator representations, respectively, to the dual frame representation. The Gramm and dual Gramm operators, $G $ and $ G'$, respectively, map between the frame and dual frame representations.}
\end{figure}

We first consider how to pass from the frame representation $ \hat v = T' v \in \ell^2 $ of a bounded vector field $ v $ to the corresponding operator representation $ L = W v \in B(\ell^2)$. Let $ \mathcal{T}' = T' L^\infty_\mathfrak{X} \subseteq \ell^2 $ be the range of the dual analysis operator restricted to $ L^\infty_\mathfrak{X} $. Then, the inverse of $ T' \rvert_{L^\infty_\mathfrak{X}} $ is given by $ T^* \rvert_{\mathcal{T}'} $, and we have 
\begin{equation}
  \label{eqFrame2Op}
  L = W T^* \hat v.
\end{equation}
Since every bounded operator $ A \in \mathcal{B}(\ell^2) $ is uniquely characterized by the coefficients $ A_{ij} = \langle e_i, A e_j \rangle_{\ell^2} $ (its ``matrix elements''), it suffices to compute
\begin{align}
    \nonumber L_{ij} &= \langle e_i, L e_j \rangle_{\ell^2} = \langle e_i, U\iota(v) U^{(1)*} e_j \rangle_{\ell^2} = \langle U^* e_i, \iota( v ) U^{(1)*} e_j \rangle_{H}  \\
    \label{eqLIJ}&= \langle \phi_i, \iota(v ) \phi^{(1)}_j \rangle_H = \frac{1}{\sqrt{1+\lambda_j}}\langle \phi_i, v( \phi_j ) \rangle_H, 
\end{align} 
 given the sequence $ \hat v = T' v \in \ell^2$. Substituting $ v = T^* T' v = \sum_{m=0}^\infty \hat v_m \alpha_m $, we obtain
\begin{displaymath}
  L_{ij} = \frac{1}{\sqrt{1+\lambda_j}} \sum_{m=1}^\infty \hat v_m \langle \phi_i, \alpha_m \phi_j \rangle_H,
\end{displaymath}
and since, for the frame elements in Theorem~\ref{thmFrame}, $ \alpha_m = s_m \phi_{p_m} \grad \phi_{q_m} $, it follows that
\begin{equation}
    \label{eqLIJVHat}
    L_{ij} = \sum_{m=0}^\infty \frac{ s_m }{ \sqrt{1+\lambda_j} }  \langle \phi_{p_m}, \grad \phi_{q_m}( \phi_j ) \rangle_{H_\mathfrak{X}} = \sum_{m=0}^\infty \frac{ s_m }{ \sqrt{1+\lambda_j} } G_{ijp_mq_m} \hat v_m,
\end{equation}
where $ G_{ijp_mq_m} $ are the Hodge inner products from~\eqref{eqGIJKL_Hodge}. The above expression fully characterizes the operator $WT^* $ in~\eqref{eqFrame2Op}. 

Next, we consider how to pass from the operator representation to the \emph{dual} frame representation. Defining $ \mathcal{W} = W L^\infty_\mathfrak{X}$, this procedure amounts to computing the $ \ell^2 $ sequence $ \hat v' = T v $ given $ L \in \mathcal{W} $, where $ v $ is the unique $ L^\infty_\mathfrak{X} $ vector field with the operator representation $ L = W v $. In this case, the matrix elements $ L_{ij} = \langle e_i, L e_j \rangle_{\ell^2} $ are known, and we have
\begin{equation}
  \label{eqVHatLIJ}
  \hat v'_k = \langle \alpha_k, v \rangle_{H_\mathfrak{X}} = s_k \langle \phi_{p_k} \grad \phi_{q_k}, v \rangle_{H_\mathfrak{X}} = s_k \langle \phi_{p_k}, v( \phi_{q_k} ) \rangle_H = L_{p_kq_k}. 
\end{equation} 
The above expression defines a linear map $ R: \mathcal{W} \to \mathcal{T } = T L^\infty_\mathfrak{X} $, that carries out the transformation from the operator representation to the dual frame representation. In particular, it follows from~\eqref{eqVHatLIJ} that the components $ \hat v'_k $ of the dual frame representation are equal to a subset of the matrix elements $L_{ij} $ of the operator representation, as appropriate for the frame from Theorem~\ref{thmFrame} used. In the special case of the $ \widetilde B_\mathfrak{X} $ frame, where all combinations of $ \phi_i \grad \phi_j $ are used as frame elements, the mapping $ k \mapsto (p_k,q_k) \in \mathbb{N}^2 $ is bijective, and the components of the dual frame representation are in one-to-one correspondence with the operator frame elements. 

What remains to complete the portion of the commutative diagram in Fig.~\ref{figRep} involving the bounded operator representation is a map from the dual frame representation $ \hat v' = T v \in \mathcal{T}' $ to the frame representation $  \hat v = T' v  \in \mathcal{T}' $ of $ v \in L^\infty_\mathfrak{X} $. This is accomplished by means of the dual Gramm operator $ G' $ (see~\eqref{eqVFrameDualGramm}), i.e., 
\begin{equation}
    \label{eqVHatVHatPrime}
    \hat v = T' v = T' T^{\prime*}\hat v' = G' \hat v' = G^+ \hat v'.
\end{equation} 
Note that, unlike $WT^*$ and $ R $, we do not have a closed-form expression for the action of $ G' $ on sequences. Nevertheless, in Section~\ref{secSECFiniteDimRep} we will see that this operator can be approximated by a strongly convergent sequence of finite-rank operators associated with finite collections of frame elements, whose action can be explicitly evaluated. 

We now turn attention to the task of passing between the frame and Hilbert-Schmidt operator representations. In this case, given $ v \in L^\infty_\mathfrak{X} $, we have 
\begin{displaymath}
    \tilde L = \tilde W v = \tilde W T^* \hat v \in \mathcal{B}_2(\ell^2),
\end{displaymath}
and we can expand $ \tilde L $ in the Hilbert-Schmidt operator basis $ \{ e_{ij} \} $, viz.
\begin{displaymath}
    \tilde L = \sum_{i,j=0}^\infty \tilde L_{ij} e_{ij}.
\end{displaymath}
By construction of the $ e_{ij} $, the expansion coefficients $ \tilde L_{ij} $ are equal to the matrix elements of $ \tilde L $ in the $ \{ e_i \} $ basis of $ \ell^2 $; that is (cf.~\eqref{eqLIJ}), 
\begin{align*}
    \tilde L_{ij} &= \langle e_{ij}, \tilde L \rangle_\text{HS} = \langle e_i, \tilde L e_j \rangle_{\ell^2} = \langle e_i, U\iota_2(v) \tilde U^* e_j \rangle_{\ell^2} \\
     & = \langle U^* e_i, \iota_2(v) \tilde U^* e_j \rangle_{H} = \langle \phi_i, \iota_2(v) \tilde \phi_j \rangle_H = e^{-\lambda_j/2} \langle \phi_i, v(\phi_j) \rangle_H.  
\end{align*}
Using this result and proceeding as in~\eqref{eqLIJVHat}, we find
\begin{displaymath}
    L_{ij} = \sum_{n=0}^\infty e^{-\lambda_j/2} s_n G_{ijp_nq_n} \hat v_n,
\end{displaymath}
which fully characterizes the operator $ \tilde W T^* $. Observe now that in the special case of the $ \widetilde B_\mathfrak{X} $ frame, we have
\begin{displaymath}
    e^{-\lambda_j/2} s_n G_{ijp_nq_n} = \frac{\langle b_{ij}, b_{p_nq_n} \rangle_{H_\mathfrak{X}}}{e^{(\lambda_j+\lambda_{q_n})/2}} = \langle \widetilde b_{ij}, \widetilde b_{p_nq_n} \rangle_{H_\mathfrak{X}} = \langle \alpha_{r_{ij}}, \alpha_n \rangle_{H_\mathfrak{X}} = G_{r_{ij}n}. 
\end{displaymath}
That is, for this frame, the matrix elements of the operator $ \tilde WT^* $ are in one-to-one correspondence with the matrix elements of the Gramm operator. 

The remaining transformations to pass from the Hilbert-Schmidt operator representation to the dual frame representation, and from the dual frame representation to the frame representation are completely analogous to those in~\eqref{eqVHatLIJ} and \eqref{eqVHatVHatPrime}, respectively, so we do not discuss them further here.

\subsection{\label{secSECFiniteDimRep}Finite-dimensional representations and their convergence properties}

We now consider how to construct finite-rank analogs of the representations of vector fields introduced in Section~\ref{secSECRepInfDim}. We begin by introducing the finite-rank (hence, compact) analysis and synthesis operators, $ T_l = \pi_l T $ and $T^*_l = T^* \pi_l$, respectively, where
\begin{displaymath}
    T_lv  = ( \langle \alpha_0, v \rangle_{H_\mathfrak{X}}, \ldots, \langle \alpha_{l-1}, v \rangle_{H_\mathfrak{X}}, 0, 0, \ldots ), \quad T^*_l( c ) = \sum_{k=0}^{l-1} c_k \alpha_k, 
\end{displaymath}
for $ v \in H_\mathfrak{X}$ and $c = ( c_0, c_1, \ldots ) \in \ell^2 $. We also define finite-rank analysis and synthesis operators for the dual frame, $ T'_l = \pi_l T' $ and $ T_l^{\prime*} = T^{\prime*} \pi_l $,  and the finite-rank Gramm operators $G_l = T_l T_l^* $, $ G'_l = T_l' T_l^{\prime *} $. As in the case of the full frame and its dual (see Section~\ref{secFrameDef}), we can relate the finite-rank operators associated with the frame and dual frames by pseudoinverses.  In particular, since $ \pi_l $ is an orthogonal projection, we have $ \pi_l^+ = \pi_l $, and therefore
\begin{displaymath}
  (T^{*}_l)^+ = ( T^* \pi_l )^+ = \pi_l^+ ( T^{*} )^+ = \pi_l T' = T'_l.
\end{displaymath}
Similarly, we have $ ( T^{\prime*}_l )^+ = T_l $ and $ G_l^+ = G_l' $. It is important to note that by virtue of its relationship with $ G_l^+ $, it is possible to compute the nonzero matrix elements $ \langle e_i, G_l' e_j \rangle_{\ell^2}$ of $ G'_l $ numerically by computing the pseudoinverse of the $ l \times l $ Gramm matrix $\bm G_l = [ \langle e_i, G_l e_j \rangle ]_{i,j=1}^{l}$, whose elements are known analytically. Thus, we can consider $ G'_l $ to be available to us in applications, albeit not in closed form.  

In addition to finite-rank operators for the frame and dual frame, we will need finite-rank Fourier operators $ U_l = \pi_l U $, $ U^{(1)}_l = \pi_l U^{(1)} $, and $ \tilde U_l = \pi_l \tilde U $ for the $ \{ \phi_j \}  $, $ \{ \phi_j^{(1)} \} $, and $ \{ \tilde \phi_j \} $ bases of $ H $, $ H^1 $, and $ \mathcal{H} $, respectively. These operators lead to finite-rank analogs  $ V_l : \mathcal{B}(H^1,H) \to \mathcal{B}(\ell^2) $ and $ \tilde V_l : \mathcal{B}_2( \mathcal{H}, H ) \to \mathcal{B}_2( \ell^2 )$ of $ V $ and $ \tilde V $, respectively, such that $ V_l A = U_l A U^{(1)}_l $ and $ \tilde V_l = U_l A \tilde U_l $. In addition, $W_l : L^\infty_\mathfrak{X} \to \mathcal{B}(\ell^2)$ and $ \tilde W_l : L^\infty_\mathfrak{X} \to \mathcal{B}_2(\ell^2)$ with $ W_l = \iota \circ V_l $ and $\tilde W_l = \iota_2 \circ \tilde V_l $ are finite-rank analogs of $W$ and $\tilde W$, respectively. As with $ G_l $ and $ G_l' $, the operators $ W_l $ and $\tilde W_l$ can be represented by $l \times l $ matrices.

With the above definitions, we can construct finite-rank analogs of the frame, dual frame, and operator representations for vector fields introduced in Section~\ref{secSECRepInfDim}. In particular, $ \hat v_l = T'_l v $, $ \hat v_l' = T_l v  $, $ L_l = W_l v $, and $ \tilde L_l = \tilde W_l v $ are respectively finite-rank frame, dual frame, bounded operator, and Hilbert-Schmidt operator representations of a vector field $ v \in L^\infty_\mathfrak{X} $. To examine the convergence properties of these representations, note that since $ \{ e_k \} $ is a basis, the projection operators $ \pi_l $ converge strongly to the identity as $l \to \infty$ (i.e., $\pi_l c \to c $ for any $ c \in \ell^2 $). This implies that $T_l \to T $, $ T'_l \to T $, $ U_l \to U $, $ U^{(1)}_l \to U^{(1)} $, $ \tilde U_l \to \tilde U $, and as a result  $ G_l \to G $, $G'_l \to G'$, $ W_l \to W $, $\tilde W_l \to \tilde W$, where all limits are taken in the respective strong operator topologies. Therefore, $\hat v_l $ and $\hat v_l'$ converge to $ \hat v = T v $ and $ \hat v' = T' v $, respectively, in $\ell^2$ norm, and $ L_l$ and $\tilde L_l$ converge to $L = W v $ and $ \tilde L = \tilde W v $, respectively, strongly. In fact, since $\tilde L_l = \sum_{i,j=0}^{l-1} \tilde L_{ij} e_{ij}$, $\tilde L_{ij} = \langle e_i, \tilde L e_j \rangle_{\ell^2}$, it follows that  $\tilde L_l $ converges to $ \tilde L = \sum_{i,j=0}^\infty \tilde L_{ij} e_{ij} $ in Hilbert-Schmidt operator norm. This implies convergence in $\mathcal{B}(\mathcal{H},H)$ operator norm, which implies in turn the strong convergence just stated. In effect, by restricting vector fields to act on the RKHS $\mathcal{H}$ containing functions of higher regularity than $H^1$, the Hilbert-Schmidt operator representation allows for a stronger mode of convergence than the bounded operator representation. 

Formulas for passing between the finite-rank frame, dual frame, and operator representations can be constructed analogously to those described in Section~\ref{secSECRepInfDim}. As an application of the operator representation of vector fields, which was already mentioned in Sections~\ref{laplacianOverview} and \ref{secSECRepInfDim}, and will be employed in Section~\ref{secNumerics}, we note that the pushforward map $F_* : \mathfrak{X} \to \Gamma \mathbb{R}^n$ on vector fields associated with an embedding $F: \mathcal{M} \to \mathbb{R}^n$ (here, $\Gamma \mathbb{R}^n \simeq \mathbb{R}^n$ is the space of smooth vector fields on $\mathbb{R}^n$) is given by $ F_* v = v( F ) $, where $ v \in \mathfrak{X} $ acts  on $F$ componentwise. As a result, we can consistently approximate $F_*v $ by  $L_l F$, where $ L_l = W_l v $ is the operator representation of $v$. The result of this operation is an ``arrow plot'', consisting of approximate tangent vectors to the image  $F(\mathcal{M}) \subset \mathbb{R}^n$ of the manifold under the embedding.

\section{\label{secGalerkin}Galerkin method for the 1-Laplacian}

We now apply the framework developed in Section~\ref{secSECSmooth} to construct a Galerkin approximation scheme for the eigenvalues and eigenforms of the 1-Laplacian, $ \Delta_1 $. Eigenvalue problems for other differential operators of interest in exterior calculus can be formulated analogously. For notational simplicity, in this section we will use the symbols $ (\nu_k, \varphi_k) $ to denote a general eigenvalue-eigenfunction pair of $ \Delta_1 $, as opposed to the multi-index notation $ (\lambda_{j,1}, \phi_{j,1} ) $ from Section~\ref{secSpacesForms}. 

\subsection{Variational eigenvalue problem for the 1-Laplacian and its Galerkin approximation}

We begin by stating the eigenvalue problem for the 1-Laplacian in strong form. This amounts to finding $ \varphi_k\in \Omega^1 $ and $ \nu_k \in \mathbb{ C } $ such that 
\begin{equation}
  \label{eqEigStrong}
  \Delta_1 \varphi_k = \nu_k \varphi_k.
\end{equation}
As is well known \cite{Strichartz83,Rosenberg97}, for the class of smooth closed Riemannian manifolds studied here, $ \Delta_1 $ has a unique self-adjoint extension $ \bar\Delta_1 : D( \bar \Delta_1 ) \to H_1 $, with a dense domain $ D( \bar \Delta_1 ) \simeq H^2_1 \subset H_1 $ and a compact resolvent. As a result, we can obtain weak solutions to~\eqref{eqEigStrong} by passing to a variational formulation, with an associated well posed Galerkin approximation scheme \cite{BabuskaOsborn91}. To construct this variational eigenvalue problem, we introduce the sesquilinear forms $ L_\theta : H^1_1 \times H^1_1 \to \mathbb{C} $, $\theta > 0 $,  and $ B : H^1_1 \times H^1_1  \to \mathbb{C}$, defined by
\begin{displaymath}
    L_\theta( \psi, \omega ) = E_{1,1}( \psi, \omega ) + \theta \langle \psi, \omega  \rangle_{H_1}, \quad B( \psi, \omega ) = \langle \psi, \omega \rangle_{H_1}, 
\end{displaymath}
where $ E_{1,1} $ is the Dirichlet form on $H_1$ from~\eqref{eqDirichletK}. Note that $E_{1,1}(\psi,\omega)$ can be formally obtained by performing integration by parts on the expression $\langle \psi,\Delta_1\omega \rangle_{H_1} $, taking $ \psi $ and $ \omega $ to be smooth 1-form fields. The term in $L_\theta(\psi,\omega)$ proportional to $ \theta $ is a regularization term, ensuring that $ L_\theta $ has a coercivity property important to the well-posedness of our Galerkin scheme. Specifically, we seek to solve the following variational eigenvalue problem:
\begin{definition}[eigenvalue problem for the 1-Laplacian, weak form] 
    \label{defnEigWeak}We say that $\nu_k\in \mathbb{C}$ and $ \varphi_k  \in H^1_1 $, solve the weak eigenvalue problem for $L_\theta$ if the equality
    \begin{displaymath}
        L_\theta( \psi, \varphi_k ) = \nu_k B( \psi, \varphi_k )
    \end{displaymath}
    holds for all $ \psi \in H^1_1 $.
\end{definition}

We refer to the solutions $(\nu_k, \varphi_k)$ of the problem in Definition~\ref{defnEigWeak} as weak eigenvalues and eigenfunctions of the 1-Laplacian, respectively. Clearly, every classical eigenvalue (eigenfunction) from~\eqref{eqEigStrong} is also a weak eigenvalue (eigenfunction). 

We now discuss the well-posedness of the weak eigenvalue just formulated, and establish its Galerkin approximation using our frames for $ H^1_1$. A data-driven analog of this Galerkin method, utilizing frame elements of $H^1_1$ approximated from data, will be presented in Section~\ref{secDataDriven}. 

\begin{lemma}
    The sesquilinear forms $L_\theta$ and $B$ obey the bounds
    \begin{gather*}
        \lvert L_\theta( \psi, \omega ) \rvert \leq ( 1 + \theta ) \lVert \psi \rVert_{H^1_1} \lVert \omega \rVert_{H^1_1}, \quad L_\theta( \omega, \omega ) \geq \min\{ \theta, 1 \} \lVert \omega \rVert^2_{H^1_1}, \\
        \lvert B( \psi, \omega ) \rvert \leq \lVert \psi \rVert_{H^1_1} \lVert \omega \rVert_{H^1_1}, 
    \end{gather*}
    for all $ \psi, \omega \in H^1_1$. 
    \label{lemmaSesquiBounds}
\end{lemma}

\begin{proof}
    The upper bounds on $\lvert L_\theta(\psi,\omega) \rvert $ and $ \lvert B(\psi,\omega) $ follow directly from application of Cauchy-Schwartz inequalities. To verify the lower (coercivity) bound on $ L_\theta(\omega,\omega)$, note that if $0< \theta \leq 1 $, we have
    \begin{displaymath}
        L_\theta(\omega,\omega) = \theta ( \theta^{-1} E_{1,1}(\omega,\omega) + \lVert \omega \rVert^2_{H_1}) \geq \theta ( E_{1,1} (\omega,\omega) + \lVert \omega \rVert^2_{H_1}) = \theta \lVert \omega \rVert^2_{H^1_1}. 
    \end{displaymath}
    The claim for $\theta > 1 $ can be verified similarly. 
\end{proof}

According to classical results in spectral approximation theory \cite{BabuskaOsborn91}, Lemma~\ref{lemmaSesquiBounds} implies that there exists a compact operator $ A_\theta : H^1_1 \to H^1_1 $ such that 
\begin{equation}
    \label{eqAOp}
    L_\theta(\psi,A_\theta \omega) = B( \psi, \omega ), \quad \forall \psi,\omega \in H^1_1.
\end{equation}
This implies in turn that $ ( \nu_k, \varphi_k) $, $ \nu_k \neq 0 $, is a weak eigenvalue-eigenvector pair if and only if
\begin{displaymath}
    A_\theta \varphi_k = \nu_k^{-1} \varphi_k.
\end{displaymath}
Due to the above, $A_\theta$ can be thought of as a ``solution operator'' for the variational eigenvalue problem in Definition~\ref{defnEigWeak}. In particular, the properties of spectral approximations to the solutions of that problem can be analyzed in terms of approximations of the eigenvalue problem for $A_\theta$. This approach leads to a Galerkin approximation scheme, as follows.

Let $ \Pi_1, \Pi_2, \ldots $ be a family of finite-rank projection operators on $H^1_1$, converging pointwise to the identity; that is, $ \Pi_l^2 = \Pi_l $ and $ \lim_{l\to\infty} \Pi_l \omega = \omega $ for every $ \omega \in H^1_1 $. Let also $ W_l $ be the closed subspaces of $ H^1_1 $ defined as $ W_l = \ran \Pi_l $. These spaces, which will be constructed explicitly below, will be our Galerkin approximation spaces. In particular, we will solve: 

\begin{definition}[eigenvalue problem for the 1-Laplacian, Galerkin approximation] 
    \label{defnEigGalerkin}We say that $\nu_{k,l} \in \mathbb{C}$ and $ \varphi_{k,l} \in W_l $, solve the Galerkin truncated eigenvalue problem for $L_\theta$ if the equality
    \begin{displaymath}
        L_\theta( \psi, \varphi_{k,l}) = \nu_{k,l} B( \psi, \varphi_{k,l} )
    \end{displaymath}
    holds for all $ \psi \in W_l $.
\end{definition}

Given a basis $ \{ v_{0,l}, \ldots, v_{w_l-1,l}\} $ of $ W_l $, where $ w_l = \dim W_l $, the eigenvalue problem in Definition~\ref{defnEigGalerkin} is equivalent to a generalized matrix eigenvalue problem. That is, $ ( \nu_{k,l}, \varphi_{k,l} ) $ with $ \varphi_{k,l} = \sum_{j=0}^{w_l-1} c_{j} v_{j,l} $ is a solution if and only if 
\begin{equation}
    \label{eqGEV}
    \bm{L}_l \vec a_{k,l} = \nu_{k,l} \bm{B}_l \vec a_{k,l},
\end{equation}
where $ \bm{L}_l $ and $ \bm B_l $ are the $ w_l \times w_l $ matrices with elements $ [ \bm{L}_l ]_{ij} = L_\theta( v_{i,l}, v_{j,l} ) $ and $[ \bm B_l ]_{ij} = B( v_{i,l}, v_{j,l})$, respectively, and $ \vec a_{k,l} = ( a_0, \ldots, a_{w_l-1})^\top \in \mathbb{R}^l$ is the $l$-dimensional column vector containing the expansion coefficients of $ \varphi_{k,l} $ in the $ \{ v_{j,l} \} $ basis of $W_l $. In addition, one can verify that $ ( \nu_{k,l}, \varphi_{k,l} ) $ solves the eigenvalue problem in Definition~\ref{defnEigGalerkin} if and only if 
\begin{displaymath}
    A_{\theta,l} \varphi_{k,l} = \nu^{-1}_{k,l} \varphi_{k,l},
\end{displaymath}
where $ A_{\theta,l} : H^1_1 \to H^1_1 $ is the finite-rank operator given by $ A_{\theta, l} = \Pi_l A_\theta $. Note that this operator satisfies (cf.~\eqref{eqAOp})
\begin{displaymath}
    L_\theta( \psi, A_{\theta,l} \omega ) = B( \psi, \omega ), \quad \forall \psi \in W_l, \quad \forall \omega \in H^1_1.
\end{displaymath}

Now, because $ A_\theta $ is compact, the fact that $ \Pi_l $ converges pointwise to the identity implies that $ A_{\theta,l} $ converges to $ A_\theta $ in norm. This implies in turn that for every eigenvalue $ \nu_k $ of $ A_\theta $ (which is nonzero by coercivity of $L_\theta$, and thus isolated and with finite geometric multiplicity by compactness of $A_\theta$) there exists a sequence $ \nu_{k,l} $ of eigenvalues of $ A_{l,\theta} $ converging as $ l\to \infty $ to $\nu_k $. Moreover, for every eigenfunction $ \varphi_k $ in the eigenspace of $ A_\theta $ at eigenvalue $ \nu_k $, there exists a sequence $ \varphi_{k,l} $ of eigenfunctions of $ A_{\theta,l} $ at eigenvalue $ \nu_{k,l} $ converging in $ H^1_1 $ norm to $ \varphi_k $. This establishes convergence of the solutions of the  the Galerkin scheme in Definition~\ref{defnEigGalerkin} to those of the eigenvalue problem in Definition~\ref{defnEigWeak}.

\subsection{Construction of the Galerkin approximation spaces}

What remains is to construct the projection operators $ \Pi_l $ and the associated subspaces $ W_l $. Here, we will construct these operators making use of the result established in Theorem~\ref{thmFrameSob} that $ \{ b^{ij}_1 \} $ with $ i \in \{ 0,1, \ldots \} $ and $ j \in \{ 1, \ldots, J \} $ is a frame of $ H^1_1 $. As in Section~\ref{secSECRep}, for notational simplicity, we set $\alpha_k = b^{p_kq_k}_1$, where $k \mapsto ( p_k, q_k) $ is any ordering of the $(i,j)$ indices in $b^{ij}_1$ with $ k \in \{ 0, 1, \ldots \} $. We also let  $T: H^1_1 \to \ell^2$, $T^*: \ell^2 \to H^1_1$, and $S = T^* T : \ell^2 \to \ell^2$ be the corresponding analysis, synthesis, and frame operators, respectively. We also consider the finite-rank operators $ T_l = \pi_l T $, $T^*_l = T^* \pi_l$, $ S_l = T_l^* T_l $, and $G_l = T_l T_l^* $ , associated with the canonical orthogonal projection operators $ \pi_l: \ell^2 \to \ell^2 $. All of these operators converge strongly to their infinite-rank counterparts as $ l \to \infty$; see Section~\ref{secSECFiniteDimRep}.

By construction, at each $ l $, the operator $ G_l $ is a positive-semidefinite, self-adjoint, compact  operator on $\ell^2$. As a result, there exists an orthonormal basis $ \{ u_{k,l} \}_{k=1}^\infty$ of $\ell^2$ consisting of eigenvectors of $ G_l$. We denote the corresponding eigenvalues by $\eta_{k,l}$, and order the eigenpairs $(\eta_{k,l}, u_{k,l})$ in order of decreasing $ \eta_{k,l} $. We let $ w_l $ be the number of nonzero eigenvalues, where $ w_l = \rank T_l \leq l $. The following lemma, whose proof is left to the reader, summarizes certain properties of the $(\eta_{k,l}, u_{k,l})$ eigenpairs. 

\begin{lemma}
    The eigenpairs $(\eta_{k,l},u_{k,l})$ have the following properties:

    (i) $ \eta_{k,l} $ is bounded above by the operator norm of $G$.
    
    (ii) $\eta_{k,l} $ is a nonzero eigenvalue of $G_l$ if and only if it is a nonzero eigenvalue of $S_l$. Moreover, the vectors
    \begin{displaymath}
        v_{k,l} = \frac{1}{\eta_{k,l}^{1/2}} T^* u_{k,l}, \quad 0 \leq k \leq w_l - 1,
    \end{displaymath}
    are orthonormal eigenvectors of $S_l$ corresponding to the same eigenvalues as $u_{k,l}$.
    \label{lemmaEigSL}
\end{lemma}

By Lemma~\eqref{lemmaEigSL}(ii), we can approach the problem of constructing orthonormal sets in $H^1_1$, consisting of eigenvectors of $S_l$ corresponding to nonzero eigenvalues, through the eigenvalue problem of $G_l$. This is advantageous, since the eigenvectors of $G_l $ corresponding to nonzero eigenvalues can be computed via the eigenvalue problem of the $l \times l $ Grammian matrix for $H^1_1$ introduced in Section~\ref{laplacianOverview}, which we denote here by $ \bm G^1_l$ to make its dependence on $l = M  J $ explicit. In particular, observe that given any $ c = ( c_0, c_1, \ldots ) \in \ell^2 $, we have $ d = G_l c = ( d_0, \ldots, d_{l-1}, 0, 0, \ldots ) $, where the first $l$ elements of $ d $ are given by $d_j = \sum_{j=0}^{l-1} G_{ij} c_j$, with
\begin{displaymath}
    G_{ij} = \langle e_i, G e_j \rangle_{\ell^2} = \langle T^* e_i, T^* e_j \rangle_{H^1_1} = \langle \alpha_i, \alpha_j \rangle_{H^1_1},
\end{displaymath}
and the inner products $\langle \alpha_i, \alpha_j \rangle_{H^1_1}$ can be computed in closed form via the formulas in Table~\ref{fig3}. The above implies that the nonzero eigenvalues $ \eta_{k,l} $ of $ G_l $ are equal to the nonzero eigenvalues of  $ \bm G_l^1 $,  and the corresponding eigenvectors $ \vec u_{k,l} = ( u_{0,k,l}, \ldots, u_{l-1,k,l} )^\top $ of that matrix yield $ u_{k,l} = ( u_{0,k,l}, \ldots, u_{l-1,k,l}, 0, 0, \ldots ) $. We thus obtain
\begin{displaymath}
    v_{k,l} = \frac{ 1 }{ \eta_{k,l}^{1/2} } \sum_{j=0}^{l-1}  u_{j,k,l} \alpha_j.
\end{displaymath}

Based on these considerations, we define our Galerkin approximation spaces as 
\begin{equation}
    \label{eqWL}
    W_l = \spn\{ v_{1,l}, \ldots, v_{w_l,l} \},
\end{equation}
and the projection operators $ \Pi_l : H^1_1 \to H^1_1 $ as orthogonal projectors onto those subspaces. 

\begin{lemma}
    The sequence $ \Pi_l $ of projection operators converges pointwise to the identity; that is, for any $ \omega \in H^1_1 $, $ \lim_{l \to \infty } \Pi_l \omega = \omega $. 
    \label{lemmaProj}
\end{lemma}

\begin{proof}
    Since the frame operator $ S $ has a bounded inverse, it suffices to show that $ S \Pi_l \omega $ converges to $ S \omega$ for any $\omega \in H^1_1$. To verify this, observe first that
    \begin{align*}
        T \Pi_l \omega &= \sum_{k=0}^{w_l-1} \langle v_{k,l}, \omega \rangle_{H^1_1} T v_{k,l} = \sum_{k=0}^{w_l-1} \langle u_{k,l}, T\omega \rangle_{\ell^2} u_{k,l} = \pi_l \sum_{k=0}^{w_l-1} \langle u_{k,l}, T \omega \rangle_{\ell^2} u_{k,l} = \\
        & = \pi_l \sum_{k=0}^\infty \langle u_{k,l}, T \omega \rangle_{\ell^2} u_{k,l} = \pi_l T \omega. 
    \end{align*}
    We therefore have
    \begin{displaymath}
        S \Pi_l \omega = T^* T \Pi_l \omega = T^* \pi_l T \omega = S_l \omega, 
    \end{displaymath}
    which converges to $ S \omega $ by the pointwise convergence of $ S_l $ to $S$.
\end{proof}

Lemma~\ref{lemmaProj} implies that with the choice of approximation spaces in~\eqref{eqWL}, the Galerkin scheme in Definition~\ref{defnEigGalerkin} converges. Moreover, all of the matrix elements of the associated sesquilinear forms can be evaluated using Lemma~\ref{lemmaEigSL} in conjunction with the formulas listed in Section~\ref{secFrameDef} and \ref{derivations}. Explicitly, the $w_l \times w_l $ matrices appearing in the generalized eigenvalue problem in~\eqref{eqGEV} are given by
\begin{equation}
    \bm L_l = \bm H_l^{-1/2} \bm U_l^\dag ( \bm E_l - \theta \bm G^1_l ) \bm U_l \bm H_l^{-1/2}, \quad \bm B_l = \bm H_l^{-1/2} \bm U_l^\dag \bm G_l \bm U_l \bm H_l^{-1/2},
    \label{eqGEVMat}
\end{equation}
where $\bm H_l$ is a $w_l \times w_l$ diagonal matrix with $[\bm H_l]_{ii} = \eta_{i,l}$, $\bm U_l$ an $ l \times w_l$ matrix with $[\bm U_l]_{ij} = u_{i,j,l}$, $\bm E_l$ an $l \times l$ matrix with $[\bm E_l]_{ij} = E_{1,1}(\alpha_i, \alpha_j)$ determined from Table~\ref{fig3}, and $\bm G_l$ the $l \times l$ Grammian matrix for $H_1 $ with $[\bm G_l]_{ij} = \langle \alpha_i, \alpha_j \rangle_{H_1}$ determined from Table~\ref{fig3}. Note that the matrices $\bm L_l$ and $\bm B_l$ in~\eqref{eqGEVMat} differ from the corresponding matrices appearing in the generalized eigenvalue problem in~\eqref{eqGEVIntro} in that they include  $ \bm H_l$- and $ \theta $-dependent terms, which do not appear in~\eqref{eqGEVIntro}. Due to the absence of these terms, \eqref{eqGEVIntro} represents the problem in Definition~\ref{defnEigGalerkin} in a basis of $W_l$ that exhibits unbounded growth of $H^1_1$ norm with $l $ (controlled by the $ \bm H_l $ terms in~\eqref{eqGEVMat}), and is also not compatible with the coercivity condition in Lemma~\ref{lemmaSesquiBounds} (enforced by the $\theta$-dependent terms). While both of these issues could potentially affect the numerical conditioning of~\eqref{eqGEVIntro}, especially at large spectral orders $l $, in the examples studied in Section~\ref{secNumerics} we found that \eqref{eqGEVIntro} and~\eqref{eqGEVMat} perform comparably.     

\section{\label{secDataDriven}Data-driven approximation}

All of the schemes in Sections~\ref{secSECSmooth} and~\ref{secGalerkin} can be implemented given knowledge of the eigenvalues and eigenfunctions of the Laplacian on functions. The problem of approximating these objects from finite sets of points in a purely data-driven manner (that is, without requiring explicit knowledge of the manifold $\mathcal{M}$ and/or its embedding in data space) has been studied extensively in recent years \cite{BelkinNiyogi03,CoifmanLafon06,Singer06,VonLuxburgEtAl08,BerryHarlim16,BerrySauer16b,BerrySauer16,belkin2007convergence,shi2015convergence,trillos2018error,trillos2018variational,HeinEtAl05}, leading to the development of approximation techniques with well-established pointwise and spectral convergence guarantees. In this section, we summarize the main properties of one such techniques, namely the diffusion maps algorithm \cite{CoifmanLafon06}, and describe the analogs of the methods of Sections~\ref{secSECSmooth} and~\ref{secGalerkin} in a data-driven, discrete setting. It should be noted that, while generally expected on the basis of results for related techniques \cite{belkin2007convergence,shi2015convergence,trillos2018error,trillos2018variational}, to our knowledge, the spectral convergence result for diffusion maps in Theorem~\ref{thmHeatConv} below has not been stated elsewhere in the literature, so we have included here a self-contained proof for completeness.

\subsection{Assumptions for data-driven approximation}

We consider that the Riemannian manifold $(\mathcal{M},g)$ is embedded in $n$-dimensional Euclidean space by means of a smooth, isometric embedding $F: \mathcal{M} \to \mathbb{R}^n $. Using the notation $ y \cdot z = \sum_{i=1}^n y^i z^i $ to represent the canonical Euclidean inner product between two vectors $ y = ( y^1, \ldots, y^n ) $ and $ z = ( z^1, \ldots, z^n ) $ in $\mathbb{R}^n$, we thus have $ g_x( u, v ) = F_{*,x} u \cdot F_{*,x} v $ for any point $ x \in \mathcal{M} $ and tangent vectors $ u, v \in T_x \mathcal{M} $, where $ F_{*,x} : T_x \mathcal{M} \to T_{F(x)} \mathbb{R}^n $ is the pushforward map on tangent vectors associated with $F$, and we have used the canonical isomorphism $ T_{F(x)} \mathbb{R}^n \simeq \mathbb{R}^n $. If the embedding $F $ is not isometric, then the method described below can be modified via the techniques developed in \cite{BerrySauer16} to yield approximations of Laplacian eigenvalues and eigenfunctions with respect to any (known) metric $ g $. 

We also assume that we have access to a dataset consisting of $ N $ samples $ y_1, \ldots, y_{N} $ in $ \mathbb{ R}^n $ with $ y_j = F( x_j ) $ taken on a sequence of distinct points $ x_1, x_2, \ldots $ in $ \mathcal{M} $, which is equidistributed with respect to a smooth sampling measure $ \sigma $ supported on $\mathcal{M}$. By that, we mean that given any continuous function $ f : \mathcal{M} \to \mathbb{C} $, the result
\begin{equation}
	\label{eqErgodicLimit}
	\lim_{N\to\infty} \frac{1}{N} \sum_{j=1}^{N} f( x_j ) = \int_\mathcal{M} f \, d\sigma
\end{equation}  
holds, and moreover $ \sigma $   has a smooth density $ \rho = d\sigma/d\mu $ with respect to the Riemannian measure $ \mu $, bounded away from zero. Such an equidistributed sequence can be provided, e.g., by i.i.d.\ points on $\mathcal{M}$ (as is commonly assumed in machine learning applications), or by an orbit of an ergodic dynamical system (in which case, the $x_j$ are not independent). The requirement in~\eqref{eqErgodicLimit} is equivalent to assuming that the sequence of sampling measures $ \sigma_N = N^{-1} \sum_{n=0}^{N-1} \delta_{x_j} $ weak-converges to $ \sigma $; that is, 
\begin{displaymath}
	\lim_{N\to\infty} \int_\mathcal{M} f \, d\sigma_N = \lim_{N\to\infty} \frac{1}{N} \sum_{j=1}^{N} f( x_j ) = \int_\mathcal{M} f \, d\sigma, \quad \forall f \in C(\mathcal{M}).
\end{displaymath}   

In this data-driven setting, we will be working with the $N$-dimensional Hilbert space $L^2(\mathcal{M},\sigma_N)$ associated with the discrete sampling measure $ \sigma_N$, equipped with the inner product 
\begin{displaymath}
    \langle f, h \rangle_{L^2(\mathcal{M},\sigma_N)} = \int_\mathcal{M} f^* h \, d\sigma_N = \frac{1}{N} \sum_{j=1}^{N} f^*(x_j) h(x_j). 
\end{displaymath}   
Note that $ L^2(\mathcal{M},\sigma_N) $ consists of equivalence classes of functions on $\mathcal{M} $ which are equal up to sets of zero $ \sigma_N $ measure; that is, two functions $ f: \mathcal{M} \to \mathbb{C}$ and $ h: \mathcal{M} \to \mathbb{C} $ satisfying $ f( x_j ) = h( x_j ) $ for all $ j \in \{ 1, \ldots, N \} $, but taking arbitrarily different values at other points, lie in the same $L^2(\mathcal{M},\sigma_N) $ equivalence class. Because the points $x_1, \ldots, x_N$ are all distinct, $ L^2(\mathcal{M},\sigma_N) $ is isomorphic as a Hilbert space to $ \mathbb{C}^N$ equipped with the normalized dot product $ f \cdot  g / N $, but here we prefer to work with $ L^2(\mathcal{M},\sigma_N) $ to emphasize the fact that our data-driven approximation spaces contain equivalence classes of functions on the same underlying manifold as the equivalence classes comprising $L^2(\mathcal{M},\mu)$.   

\subsection{\label{secKernelApprox}Kernel method for the eigenvalues and eigenfunctions of the Laplacian on functions}

Following the approach introduced in the diffusion maps algorithm \cite{CoifmanLafon06}, and further generalized in \cite{BerrySauer16b}, we compute data-driven approximations of the eigenvalues and eigenfunctions of the Laplacian on functions through the eigenvalues and eigenfunctions of a kernel integral operator approximating, in a suitable sense, the heat operator $  e^{- \tau \bar \Delta}$, $ \tau \geq 0 $, on $ L^2(\mathcal{M},\mu) $, where $ \bar \Delta $ is the self-adjoint Laplacian (see Section~\ref{secSpacesFunctions}).  This kernel integral operator is constructed from a smooth, exponentially decaying kernel $ k_\epsilon : \mathcal{M} \times \mathcal{M} \to \mathbb{R}_+ $, bounded away from zero. Here, as a concrete example, we work with a Gaussian kernel,
\begin{displaymath}
	k_\epsilon(x,x') = \exp\left( - \frac{ \lVert F( x ) - F( x' ) \rVert^2_{\mathbb{R}^n} }{ \epsilon } \right),
\end{displaymath} 
where $ \epsilon $ is a positive bandwidth parameter. Approximation techniques based on other classes of kernels, including kernels with variable bandwidth functions \cite{BerryHarlim16,BerrySauer16}, have equivalent asymptotic properties while generally achieving higher performance in terms of approximation accuracy and noise robustness, particularly in applications with large variations in the sampling density $\rho$.  

Having specified an appropriate kernel, we introduce the associated kernel integral operators $ \hat K_\epsilon : L^2( \mathcal{M}, \mu ) \to C( \mathcal{M} ) $ and $ \hat K_{\epsilon,N} : L^2( \mathcal{M}, \sigma_N ) \to C( \mathcal{M} ) $, where
\begin{displaymath}
	\hat K_\epsilon f = \int_\mathcal{M} k_\epsilon( \cdot, x ) f( x ) \rho(x) \, d\mu(x), \quad \hat K_{\epsilon,N} h = \int_\mathcal{M} k_\epsilon( \cdot, x ) h( x ) \, d\sigma_N( x ). 
\end{displaymath}    
Composing $ \hat K_\epsilon $ with the canonical inclusion operator $ \iota : C( \mathcal{M}) \to L^2(\mathcal{M},\mu ) $, we also define $ \tilde K_\epsilon : L^2(\mathcal{M},\mu) \to L^2(\mathcal{M},\mu )$ and $ K_\epsilon : C( \mathcal{M} ) \to C( \mathcal{M}) $, where $ \tilde K_\epsilon = \iota \hat K_\epsilon $ and $  K_\epsilon = \hat K_\epsilon \iota $. Similarly, we define $\tilde K_{\epsilon,N} : L^2(\mathcal{M},\sigma_N) \to L^2(\mathcal{M},\sigma_N) $ and $  K_{\epsilon,N} : C(\mathcal{M}) \to C(\mathcal{M}) $, where $ \tilde K_\epsilon = \iota_N \hat K_{\epsilon,N} $, $  K_{\epsilon,N} = \hat K_{\epsilon,N} \iota_N $, and $ \iota_N : C(\mathcal{M}) \to L^2(\mathcal{M},\sigma_N) $ is the canonical restriction operator from $C(\mathcal{M}) $ to $ L^2(\mathcal{M},\sigma_N) $.  

\begin{proposition}
    \label{propK}
The operators $ \tilde K_{\epsilon} $, $ K_{\epsilon} $, $ \tilde K_{\epsilon,N} $, and $ K_{\epsilon,N} $ have the following properties.

(i) They are all compact. 

(ii) As $ N \to \infty $, $ K_{\epsilon, N} $ converges pointwise to $ K_\epsilon $; that is, for any $ f \in C(\mathcal{M}) $, we have $ \lim_{N\to\infty}  K_{\epsilon, N} f = K_\epsilon f $ in uniform norm. 
\end{proposition}

\begin{proof}
    (i) That $\tilde K_{\epsilon,N} $ and $ K_{\epsilon,N}$ are compact follows immediately from the fact that they have finite rank. The compactness of $\tilde K_\epsilon$ follows from the facts that $ k_\epsilon $ is a Hilbert-Schmidt kernel on $L^2(\mathcal{M}\times \mathcal{M}, \sigma \times \sigma)$ (i.e., $ \int_\mathcal{M} \int_\mathcal{M} k_\epsilon( x, x' ) \, d\sigma(x)\, d\sigma(x') < \infty$), and $ L^2(\mathcal{M},\mu) $ and $L^2(\mathcal{M},\sigma) $ are isomorphic Hilbert spaces (by smoothness of $ \rho $ and compactness of $\mathcal{M}$). The compactness of $ K_\epsilon $ can be verified using the Arzel\`a-Ascoli theorem in conjunction with the continuity of $ k_\epsilon$; see, e.g., \cite{DasGiannakis19}. 

    (ii) The claim is a direct consequence of~\eqref{eqErgodicLimit} and the fact that $\int_\mathcal{M} f \rho \, d\mu = \int_\mathcal{M} f \, d\sigma$ for any $ f\in L^2(\mathcal{M},\mu) $. 

\end{proof}  

Proposition~\ref{propK}(ii) shows that, on $C(\mathcal{M})$, we can approximate kernel integral operators with respect to the Riemannian measure by kernel integral operators with respect to the sampling measure. However, the pointwise convergence established there does not, in general, imply spectral convergence for these operators. Moreover, in applications we work in the finite-dimensional Hilbert space $L^2(\mathcal{M},\sigma_N)$, as opposed to the infinite-dimensional Banach space $C(\mathcal{M})$, which necessitates establishing connections between the spectral properties of $\tilde K_{\epsilon,N}$ and $K_{\epsilon,N}$. Another issue that must be addressed is that of approximating the heat operator $ e^{-\tau \bar \Delta}$ by a suitable modification of $\tilde K_\epsilon $. 

Following \cite{CoifmanLafon06,BerrySauer16b}, we proceed by normalizing $ k_\epsilon $ to construct a smooth ergodic Markov kernel $ p_\epsilon : \mathcal{M} \times \mathcal{M} \to \mathbb{R}_+$, satisfying $ \int_\mathcal{M} p_\epsilon(x,\cdot ) \, d\sigma = 1 $ for all $x\in\mathcal{M}$. For that, we introduce the normalization functions $ r_\epsilon = \hat K_\epsilon 1 $ and $ l_\epsilon = \hat K_\epsilon( 1 / r_\epsilon ) $, which are both smooth, positive, and bounded away from zero, and define
\begin{displaymath}
    p_\epsilon(x,x') = \frac{k_\epsilon(x,x')}{l_\epsilon(x) r_\epsilon(x')}.
\end{displaymath}
The Markov property of $ p_\epsilon $ then follows by construction, and its ergodicity follows from the fact that it is bounded below. In the data-driven case, we define
\begin{equation}
    \label{eqDMNorm}
    r_{\epsilon,N} = \hat K_{\epsilon,N} 1, \quad l_{\epsilon,N} = \hat K_{\epsilon,N}( 1 / r_{\epsilon,N} ), \quad p_{\epsilon,N}(x,x') = \frac{k_\epsilon(x,x')}{l_{\epsilon,N}(x) r_{\epsilon,N}(x')}, 
\end{equation}
and $ p_{\epsilon,N}$ is a smooth Markov kernel satisfying $\int_\mathcal{M} p_{\epsilon,N}(\cdot, x ) \, d \sigma_N(x) = 1 $. As in the case of the kernel $ k_\epsilon $, we define the kernel integral operators $ \hat P_\epsilon : L^2(\mathcal{M},\mu ) \to C(\mathcal{M}) $ and $\hat P_{\epsilon,N} : L^2(\mathcal{M}, \sigma_N) \to C(\mathcal{M})$ via
\begin{displaymath}
    \hat P_\epsilon f = \int_\mathcal{M} p_\epsilon(\cdot, x) f(x)  \rho(x) \, d\mu(x), \quad \hat P_{\epsilon, N} h = \int_\mathcal{M} p_{\epsilon,N}(\cdot, x) h(x) \, d\sigma_N(x),
\end{displaymath}
and also introduce the operators $ \tilde P_{\epsilon} : L^2(\mathcal{M},\mu) \to L^2(\mathcal{M},\mu) $, $ P_{\epsilon} : C(\mathcal{M}) \to C(\mathcal{M}) $, $ \tilde P_{\epsilon,N} : L^2(\mathcal{M},\sigma_N) \to L^2(\mathcal{M},\sigma_N)$, and $ P_{\epsilon,N}: C(\mathcal{M}) \to C(\mathcal{M})$, where $\tilde P_\epsilon = \iota \hat P_\epsilon$, $ P_\epsilon = \hat P_\epsilon \iota $, $\tilde P_{\epsilon,N} = \iota_N \hat P_{\epsilon,N} $, and $ P_{\epsilon,N} = \hat P_{\epsilon,N} \iota_N$. These operators have the analogous properties to those stated in Proposition~\ref{propK}. Among them, $ \tilde P_{\epsilon,N} $ is represented by the Markov matrix $ \bm P $ from~\eqref{eqPDM} (note that, due to cancellation of terms, the normalization of $k_\epsilon(x,x')$ by $r_{\epsilon,N}$ and $l_{\epsilon,N}$ to construct $p_{\epsilon,N}$ is equivalent to the normalization procedure used to construct $\bm P$; see \cite{BerrySauer16b} for details). The following theorem summarizes how the nonzero eigenvalues and corresponding eigenvectors of $\tilde P_\epsilon$ can be approximated by the corresponding eigenvalues and eigenvectors of $\tilde P_{\epsilon,N}$, which are accessible from data as described in Section~\ref{dm}. 

\begin{theorem}
    \label{thmSpecConv}
    The following hold:

(i) $ \Lambda_{j,\epsilon} $ is a nonzero eigenvalue of $\tilde P_\epsilon$ if and only if it is a nonzero eigenvalue of $ P_\epsilon $. Similarly, $ \Lambda_{j,\epsilon,N}$ is a nonzero eigenvalue of $\tilde P_{_\epsilon,N}$ if and only if it is a nonzero eigenvalue of $P_{\epsilon,N}$. Moreover, the $\Lambda_{j,\epsilon,N}$ and $\Lambda_{j,\epsilon,N}$ are all real, and thus lie in the interval $[-1,1]$ by the Markov property of $p_\epsilon$ and $p_{\epsilon,N}$.
    
    (ii) If $\tilde \phi_{j,\epsilon} \in L^2(\mathcal{M},\mu)$ and $\vec \phi_{j,\epsilon,N} \in L^2(\mathcal{M},\sigma_N)$ are eigenvectors of $P_\epsilon$ and $P_{\epsilon,N}$ at nonzero eigenvalues $\Lambda_{j,\epsilon}$ and $\Lambda_{j,\epsilon,N}$, respectively, then the smooth functions 
    \begin{displaymath}
        \phi_{j,\epsilon} = \frac{1}{\Lambda_{j,\epsilon}} \hat P_\epsilon \tilde \phi_{j,\epsilon}, \quad \phi_{j,\epsilon,N} = \frac{1}{\Lambda_{j,\epsilon,N}} \hat P_{\epsilon,N} \vec \phi_{j,\epsilon,N}
    \end{displaymath}
    are eigenvectors of $ P_\epsilon$ and $ P_{\epsilon,N}$, respectively, at the same eigenvalues. 

    (iii) For every nonzero eigenvalue $\Lambda_{j,\epsilon}$ of $ P_\epsilon $, the sequence of eigenvalues $\Lambda_{j,\epsilon,N} $ of $ P_{\epsilon,N} $ converges as $ N \to \infty $ to $ \Lambda_{j,\epsilon } $. Moreover, if $ \phi_{j,\epsilon} $ is an eigenvector of $ P_{\epsilon}$ corresponding to eigenvalue $\Lambda_{j,\epsilon}$, then there exists a sequence of eigenvectors $\phi_{j,\epsilon,N}$ of $P_{\epsilon,N}$ at eigenvalues $ \Lambda_{j,\epsilon,N} $, converging as $ N \to \infty $ to $  \phi_{j,\epsilon} $ in uniform norm.
\end{theorem}

\begin{proof}
    (i,ii) The claims on the relationships between the eigenvalues and eigenvectors of $\tilde P_\epsilon$ and $ P_\epsilon$ (and those of $\tilde P_{\epsilon,N}$ and $ P_{\epsilon,N}$) can be verified from the definition of these operators. In addition, it can be verified that $\tilde P_\epsilon$ and $\tilde P_{\epsilon,N}$ are related to self-adjoint operators by similarity transformations, which implies that their eigenvalues are real. 

    (iii) The convergence of the eigenvalues follows by showing that the operators $ P_{\epsilon,N}$ converge compactly to $ P_\epsilon $ (a stronger notion of convergence than pointwise convergence, but weaker than convergence in operator norm); see \cite{VonLuxburgEtAl08} for additional details. A proof of the convergence of the eigenvectors can be found in \cite{DasGiannakis19}. We also note that \cite{VonLuxburgEtAl08} establishes pointwise convergence of projection operators onto the corresponding eigenspaces. 
\end{proof}

That we can approximate eigenvalues and eigenvectors of $\tilde P_\epsilon$ through eigenvalues and eigenvectors of $\tilde P_{\epsilon,N} $ is important since the eigenvalue problem for the latter operator is equivalent to the numerically solvable $N \times N $ matrix eigenvalue problem for $ \bm P $. Note that by ergodicity and the Markov property, the eigenvalues $\Lambda_{j,\epsilon}$ and $\Lambda_{j,\epsilon,N}$ can be ordered as $1 = \Lambda_{0,\epsilon} > \Lambda_{1,\epsilon} \geq \Lambda_{2,\epsilon} \geq \cdots$ and $1 = \Lambda_{0,\epsilon,N} > \Lambda_{1,\epsilon,N} \geq \Lambda_{2,\epsilon,N} \geq \cdots \Lambda_{N-1,\epsilon,N}$, respectively. We will adopt these orderings for the remainder of the paper. What remains is to establish how the eigenvalues and eigenvectors of $\tilde P_\epsilon$ approximate in turn eigenvalues and eigenvectors of the Laplacian.  

\begin{theorem}
    \label{thmHeatConv}
    For every $j \in \mathbb{N}_0$ and as $\epsilon \to 0^+$, the quantities $\lambda_{j,\epsilon} = (1 - \Lambda_{j,\epsilon})/\epsilon$ converge to the Laplace-Beltrami eigenvalue $\lambda_j$. Moreover, for each Laplace-Beltrami eigenfunction $\phi_j$ corresponding to $\lambda_j$, there exist eigenfunctions $\tilde \phi_{j,\epsilon}$ of $ \tilde P_\epsilon$ corresponding to eigenvalue $\lambda_{j,\epsilon}$ converging, as $\epsilon \to 0^+$, to $ \phi_j$ in $L^2(\mathcal{M},\mu)$ norm.
\end{theorem}

\begin{proof}
    By \cite[Proposition~10]{CoifmanLafon06}, as $\epsilon \to 0^+ $, the family of operators $ \tilde \Delta_\epsilon := (I - \tilde P_\epsilon)/\epsilon$ converges to $ \bar \Delta $, pointwise on $C^\infty(\mathcal{M}) \cap L^2(\mathcal{M},\mu)$. Moreover, as can be directly verified from the normalization procedure in~\eqref{eqDMNorm}, $\tilde P_\epsilon$ is related to a self-adjoint operator $ \bar P_\epsilon : L^2(\mathcal{M},\mu) \to L^2(\mathcal{M},\mu)$ by a similarity transformation by a bounded multiplication operator with a bounded inverse; specifically, $ \tilde P_\epsilon = D_\epsilon^{1/2} \bar P_\epsilon D_\epsilon^{-1/2} $, where $D_\epsilon : L^2(\mathcal{M},\mu) \to L^2(\mathcal{M},\mu) $ is the multiplication operator by the smooth, strictly positive function $ d_\epsilon = r_\epsilon / ( l_\epsilon \rho )$. It can also be shown through small-$\epsilon$ expansions that, as $\epsilon \to 0^+$, $ d_\epsilon$ converges to a constant in $H^2$ norm. By virtue of the above, $ \bar \Delta_\epsilon = ( I - \bar P_\epsilon ) / \epsilon $ is a family of self-adjoint, compact operators, with the same eigenvalues $ \lambda_{j,\epsilon} $ as $ \tilde \Delta_\epsilon $, converging pointwise to $ \bar \Delta $ on $C^\infty(\mathcal{M}) \cap L^2(\mathcal{M},\mu)$. Now, observe that $C^\infty(\mathcal{M})$ is a \emph{core} for $\bar \Delta $ (i.e., $ \bar \Delta $ is equal to the closure of its restriction on $C^\infty(\mathcal{M})$, which follows immediately from the fact that, on manifolds without boundary, the Laplacian $ \Delta$ on $C^\infty(\mathcal{M})$ functions is essentially self-adjoint). By results from spectral approximation theory of self-adjoint operators \cite[Proposition~10.1.18]{Oliveira2008}, this actually implies that the eigenvalues $\lambda_{j,\epsilon}$ converge to $ \lambda_j$, as claimed. Moreover, by related results \cite[Chapter X.7, Corollary~3]{DunSchII1988}, the orthogonal projections to the eigenspaces of $ \bar \Delta_\epsilon $ corresponding to $ \lambda_{j,\epsilon}$ converge pointwise to the projectors onto the eigenspaces of $\bar \Delta$ at eigenvalue $\lambda_j$. The latter, in conjunction with the fact that $D_\epsilon$ converges to a multiplication operator by a constant function, leads to the claim on the convergence of the eigenfunctions of $\tilde P_\epsilon$ to Laplace-Beltrami eigenfunctions as $\epsilon \to 0^+$.      
\end{proof}

In summary, we can conclude from Theorems~\ref{thmSpecConv} and~\ref{thmHeatConv} that the quantities
\begin{equation}
    \lambda_{j,\epsilon} = \frac{1 - \Lambda_{j,\epsilon}}{\epsilon}, \quad \lambda_{j,\epsilon,N} = \frac{1-\Lambda_{j,\epsilon,N}}{\epsilon}, 
    \label{eqEigApprox}
\end{equation}
converge to the eigenvalues of the Laplacian, i.e.,
\begin{displaymath}
    \lim_{\epsilon \to 0^+} \lim_{N\to\infty} \lambda_{j,\epsilon,N} = \lim_{\epsilon \to 0^+} \lambda_{j,\epsilon} =  \lambda_j. 
\end{displaymath}
Moreover, for any eigenfunction $ \phi_j \in C^\infty(\mathcal{M}) $ of $ \Delta $ at eigenvalue $\lambda_j$,  there exist eigenfunctions $\tilde \phi_{j,\epsilon}$ and  $ \tilde \phi_{j,\epsilon,N} $ of $ P_{\epsilon} $ and $P_{\epsilon,N}$, respectively, such that the smooth functions
\begin{equation}
    \phi_{j,\epsilon} = \frac{1}{\Lambda_{j,\epsilon}} \hat P_{\epsilon}\tilde  \phi_{j,\epsilon}, \quad \phi_{j,\epsilon,N} = \frac{1}{\Lambda_{j,\epsilon,N}} \hat P_{\epsilon,N}\vec \phi_{j,\epsilon,N} 
    \label{eqEigApproxPhi}
\end{equation}
satisfy
\begin{displaymath}
    \lim_{\epsilon\to 0^+} \lim_{N\to\infty} \phi_{j,\epsilon,N} = \lim_{\epsilon\to 0^+}  \phi_{j,\epsilon} = \phi_j,
\end{displaymath}
where the limits are taken with respect to uniform norm. 

\begin{remark}
    In addition to the convergence results stated above as iterated ($\epsilon \to 0^+$ after $N \to \infty$) limits, in applications it is clearly important to have convergence results for limits where $N\to \infty$ and $\epsilon \to 0^+$ simultaneously.  In particular, note that at fixed $N \in \mathbb{N} $, the eigenvalues $\lambda_{j,\epsilon,N}$ degenerate as $\epsilon \to 0^+$, and fail to provide a good approximation of $\lambda_j$. This necessitates taking $ N \to \infty $ limits along a sequence $ \epsilon(N)$ decreasing towards zero at a sufficiently slow rate. In the literature, this problem has mainly been studied in the context of i.i.d.\ samples \cite{belkin2007convergence,shi2015convergence,trillos2018error,trillos2018variational,BerrySauer16}. For example, \cite{trillos2018variational} shows that in dimension  $ d = \dim \mathcal{M} \geq 3$ it suffices to take $\epsilon(N)$ such that $\lim_{N\to\infty} [ ( \log N )^{1/d} / ( N^{1/d} \epsilon(N) ) ] = 0$. Here, we have opted to state spectral convergence results in terms of iterated limits, as they are valid for arbitrary sampling scenarios satisfying the weak convergence property in~\eqref{eqErgodicLimit}. As previously noted, a common scenario with non-i.i.d.\ sampling is that of time-ordered data taken along orbits of ergodic dynamical systems.
\end{remark}

\subsection{Data-driven frame elements and approximation of sesquilinear forms}

Using the approximate eigenvalues and eigenfunctions from~\eqref{eqEigApprox} and~\eqref{eqEigApproxPhi}, we can construct data-driven analogs of the various basis and frame elements for functions, vector fields, and forms introduced in Sections~\ref{secPrelim} and~\ref{secSECSmooth}. For example, 
\begin{displaymath}
    \phi_{j,\epsilon,N}^{(p)} = \frac{ \phi_{j,\epsilon,N}}{\lambda_{j,\epsilon,N}^{p/2}}, \quad b_{ij,\epsilon,N}^{(p)} = \frac{ \phi_{i,\epsilon,N} \grad \phi_{j,\epsilon,N}}{\lambda_{i,\epsilon,N}^{p/2}, },  \quad  b^{i j_1 \cdots j_k}_{p,\epsilon,N} = \frac{d \phi_{j_1,\epsilon,N} \wedge \cdots \wedge d\phi_{j_k,\epsilon,N}}{\lambda_{i,\epsilon,N}^{p/2}}  
\end{displaymath}
are data-driven analogs of the basis functions $\phi^{(p)}_j$ in~\eqref{eqPhiHP}, the frame elements $b_{ij}^{(p)}$ in~\eqref{eqFrameVecA} for vector fields, and the frame elements $b^{ij_1\cdots j_k}_p$ in~\eqref{eqFrameForm} for forms, respectively. All of these objects are ``concrete'', i.e., the take values pointwise on $\mathcal{M}$ as opposed to being defined up to null sets. Moreover, their pointwise evaluation in practice relies on the ability to compute derivatives of the kernel $k_\epsilon$. 

In SEC, however, it oftentimes suffices to consider quantities that can be computed using only the ``weak'' counterparts $\tilde \phi_{j,\epsilon,N}$ of $ \phi_{j,\epsilon,N}$ lying in $L^2(\mathcal{M},\sigma_N)$. As a concrete example, consider the Galerkin method for the 1-Laplacian in Definition~\ref{defnEigGalerkin}. To construct a data-driven analog of this scheme we compute the following quantities, using the shorthand notation $ \alpha_{k,\epsilon,N} = b_{1,\epsilon,N}^{i_k j_k} $ as in Section~\ref{secGalerkin}:
\begin{enumerate}
    \item Triple products $c_{ijk,\epsilon,N} = \langle \vec \phi_{i,\epsilon,N}, \vec \phi_{j,\epsilon,N} \vec \phi_{k,\epsilon,N} \rangle_{L^2(\mathcal{M},\sigma_N)} $.
    \item Approximations $G_{ij,\epsilon,N}$ and $E_{ij,\epsilon,N}$ of the $H_1$ inner products $\langle \alpha_{i,\epsilon,N}, \alpha_{j,\epsilon,N} \rangle_{H_1}$ and Dirichlet energies $E_{1,1}(\alpha_{i,\epsilon,N}, \alpha_{j,\epsilon,N})$, computed via the formulas in Table~\ref{fig3}, with the eigenvalues $ \lambda_i $ and triple products $c_{ijk}$ replaced by $\lambda_{i,\epsilon,N}$ and $c_{ijk,\epsilon,N}$, respectively. 
    \item Approximations $u_{k,l,\epsilon,N} \in \ell^2$ of the eigenvectors of the operator $G_l$, computed by solving the eigenvalue problem of the $l \times l $ matrix $\bm G^1_{l,\epsilon,N} $ with elements $[\bm G_{l,\epsilon,N}]_{ij} = G_{ij,\epsilon,N} + E_{ij,\epsilon,N}$. 
    \item Approximation of the $w_l \times w_l $ matrices $\bm L_l$ and $\bm B_l$ by $ w_{l,\epsilon,N} \times w_{l,\epsilon,N}$ matrices $\bm L_{l,\epsilon,N}$ and $\bm B_{l,\epsilon,N}$, respectively, where $ w_{l,\epsilon,N} $ is the number of nonzero eigenvalues of $\bm G_{l,\epsilon,N}$, and $\bm L_{l,\epsilon,N}$ and $ \bm B_{l,\epsilon,N}$ are computed via~\eqref{eqGEVMat}, using the results of Steps 2 and 3 above as appropriate.
\end{enumerate}

By Theorems~\ref{thmSpecConv} and~\ref{thmHeatConv}, 
\begin{displaymath}
    \lim_{\epsilon\to 0^+} \lim_{N\to\infty} \bm L_{l,\epsilon,N} = \bm L_l, \quad \lim_{\epsilon \to 0^+} \lim_{N\to\infty} \bm B_{l,\epsilon,N} = \bm B_l, 
\end{displaymath}
in any matrix norm. Thus, for any solution $(\nu_{k,l}, \vec a_{k,l})$ of the generalized eigenvalue problem in~\eqref{eqGEV} there exist solutions $(\nu_{k,l,\epsilon,N}, \vec a_{k,l,\epsilon,N})$ of the generalized eigenvalue problems
\begin{displaymath}
    \bm L_{l,\epsilon,N} \vec a_{k,l,\epsilon,N} = \nu_{k,l,\epsilon,N} \bm B_{l,\epsilon,N} \vec a_{k,l,\epsilon,N}, 
\end{displaymath}
such that $\lim_{\epsilon\to 0^+} \lim_{N\to\infty} \nu_{k,l,\epsilon,N} = \nu_{k,l}$ and $\lim_{\epsilon \to 0^+} \lim_{N\to\infty} \vec a_{k,l,\epsilon,N} = \vec a_{k,l}$. The latter, in conjunction with the convergence of the Laplacian eigenfunctions in Theorem~\ref{thmSpecConv}(iii), implies that the reconstructed 1-forms
\begin{displaymath}
    \varphi_{k,l,\epsilon,N} = \sum_{j=0}^{w_{l,\epsilon,N}-1} [\vec a_{k,l,\epsilon,N}]_j \alpha_{j,\epsilon,N}
\end{displaymath}
converge, in $H^1_1$ norm, to the reconstructed form $\varphi_{k,l}$ associated with $\vec a_{k,l}$. By convergence of the Galerkin scheme in Definition~\ref{defnEigGalerkin}, we therefore conclude that
\begin{displaymath}
    \lim_{l\to\infty}\lim_{\epsilon\to 0^+} \lim_{N\to\infty} \nu_{k,l,\epsilon,N} = \nu_k, \quad \lim_{l\to\infty} \lim_{\epsilon\to 0^+} \lim_{N\to\infty} \varphi_{k,l,\epsilon,N} = \varphi_{k,l},
\end{displaymath}
where $(\nu_k,\varphi_k)$ is a weak eigenvalue-eigenvector pair of the 1-Laplacian, solving the variational eigenvalue problem in Definition~\ref{defnEigWeak}.

\section{Numerical examples}\label{secNumerics}

In this section, we apply the SEC to several smooth manifolds and a fractal set to verify and demonstrate the utility of our approach.  In each example, we constructed the $1$-Laplacian and its eigenvalue and eigenforms using the same procedure, which we describe here.  First, we applied the diffusion maps algorithm described in Section \ref{dm} to the data in order to estimate the eigenvalues and eigenfunctions of the $0$-Laplacian.  Setting $M=20$ eigenfunctions, we used 100 eigenfunctions to compute the $M\times M\times M$ tensor $c$ which is the Fourier representation of function multiplication (see Table~\ref{fig2} in Section \ref{functionOverview}).  Using the formulas in Table~\ref{fig3} in Section \ref{laplacianOverview}, we constructed the $M^2 \times M^2$ energy matrix $\hat{\bm E}$ and Hodge Grammian $\hat{ \bm G } $ for the anti-symmetric formulation of the SEC.  Following the procedure in Section \ref{laplacianOverview}, we then projected the eigenvalue problem onto the appropriate Sobolev $H^1_1$ basis, and computed the eigenvalues and eigenforms of the $1$-Laplacian.  Finally, we visualized the vector fields corresponding to the eigenforms by computing their operator representation, and pushing forward these vector fields into the original data space as described in Sections \ref{laplacianOverview} and~\ref{secSECFiniteDimRep}.  

In the online supplementary material we have included Matlab code which implements the SEC $1$-Laplacian construction along with a Diffusion Maps implementation. We also include code that generates all the data sets shown below and a simple ``DEMO.m'' file to replicate our results. 

\subsection{Validation of the SEC 1-Laplacian spectra}\label{validation}

In this section we apply the SEC based construction of the 1-Laplacian to two examples where the spectrum of the 1-Laplacian can easily be worked out analytically.  

\begin{figure}
\includegraphics[width=0.32\textwidth]{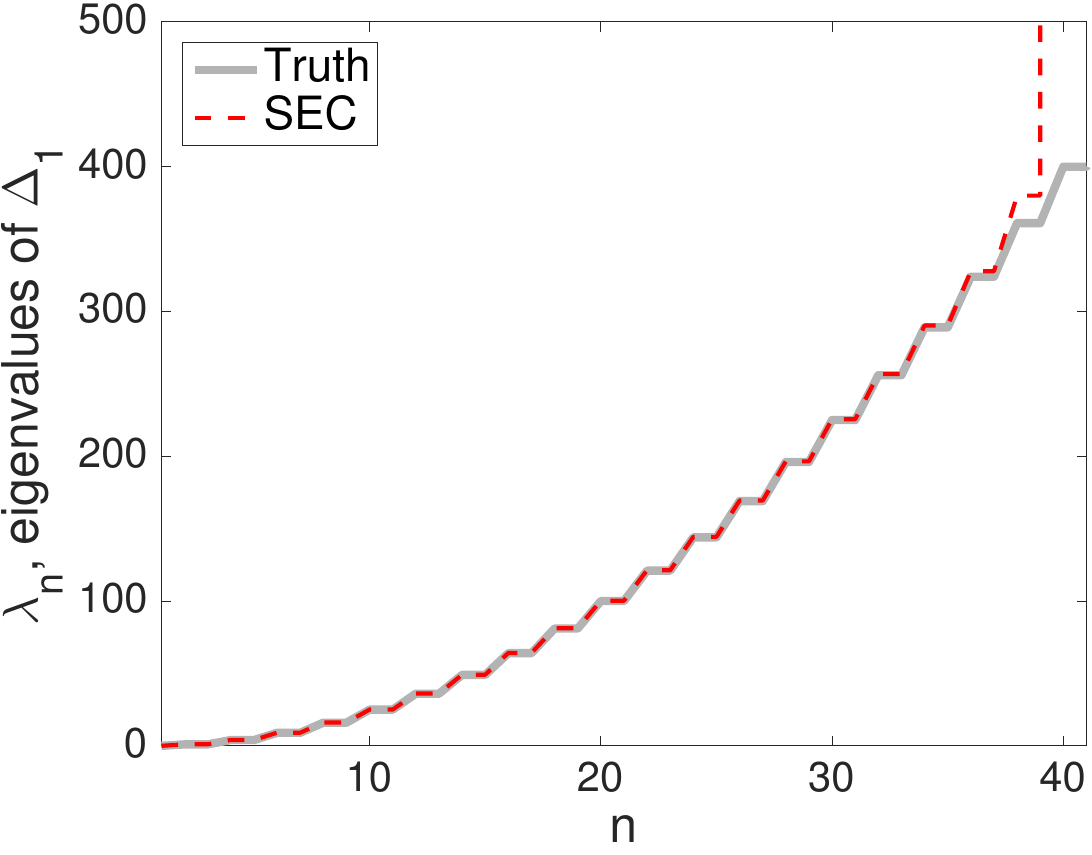}
\includegraphics[width=0.32\textwidth]{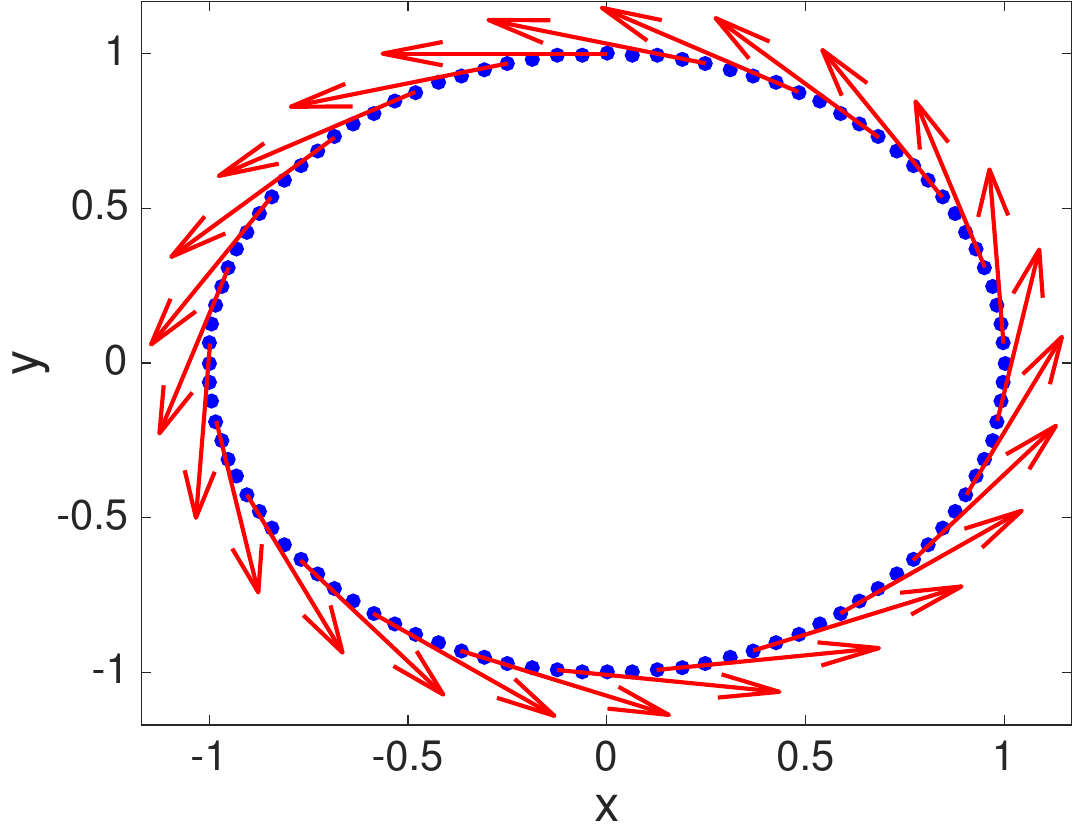}
\includegraphics[width=0.32\textwidth]{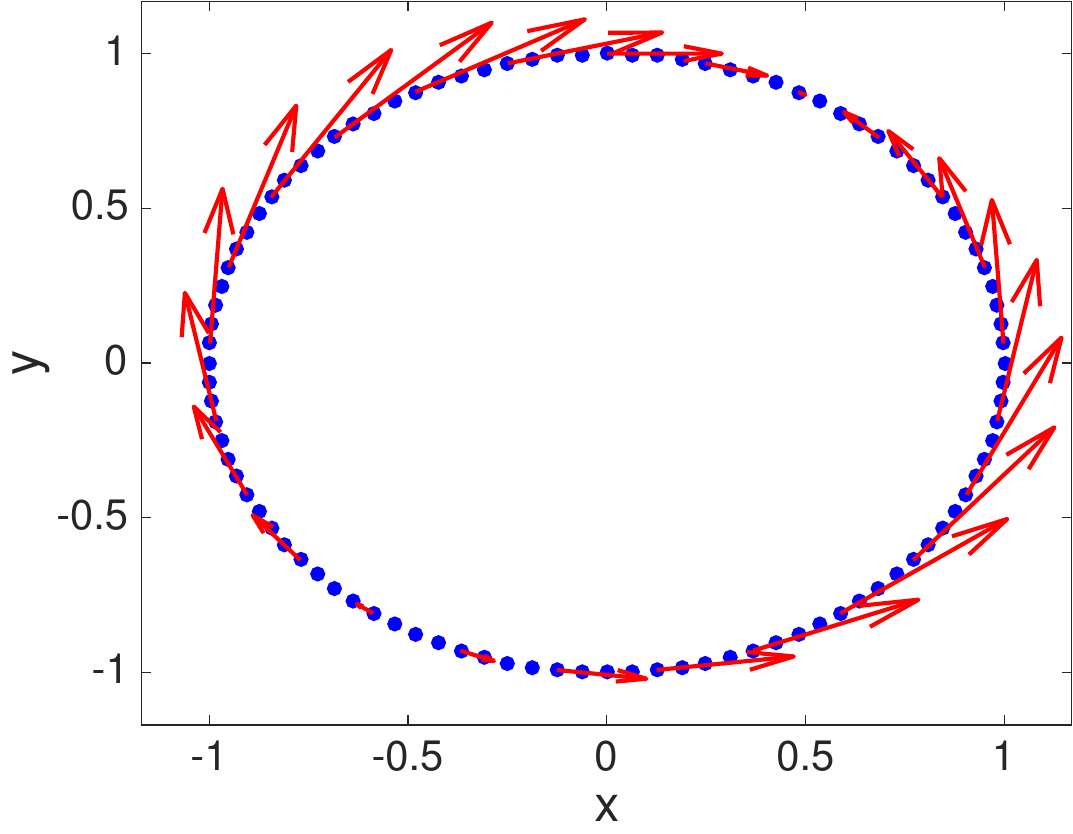}\\
\includegraphics[width=0.32\textwidth]{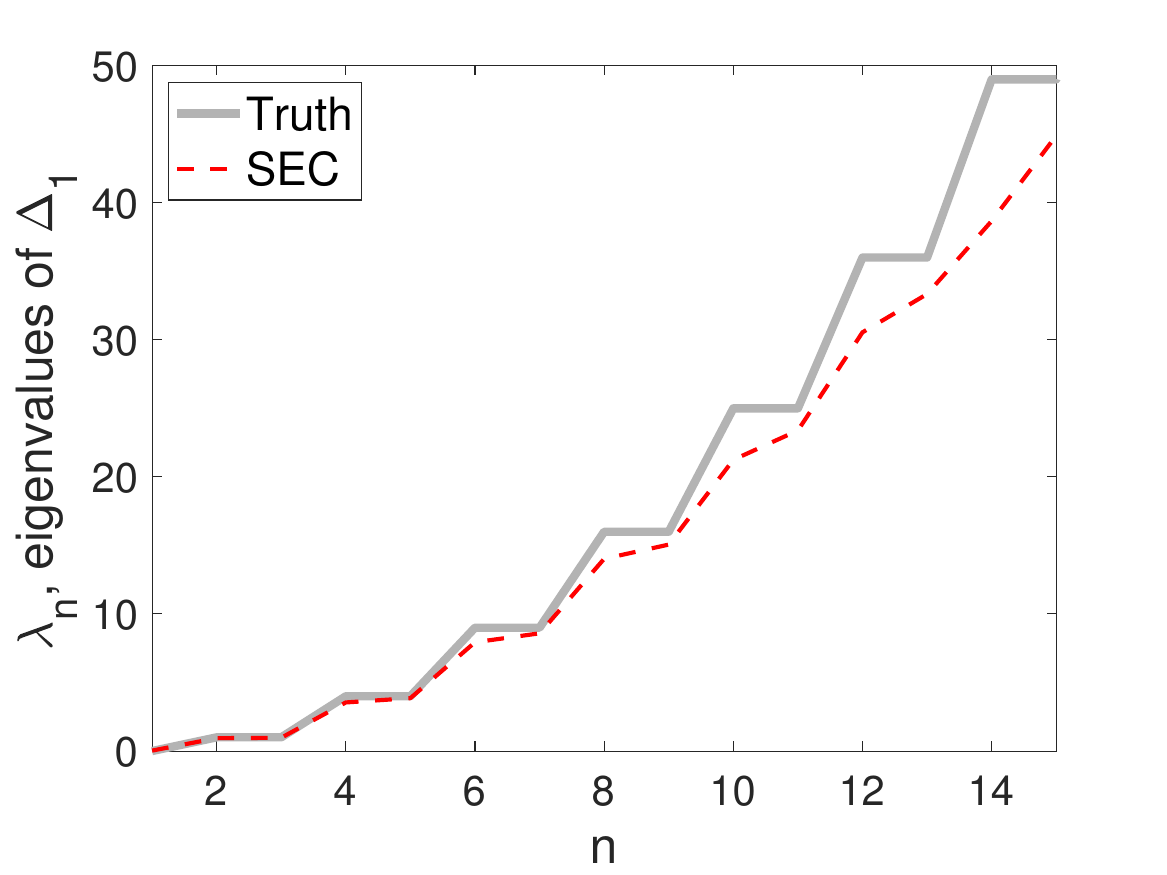}
\includegraphics[width=0.32\textwidth]{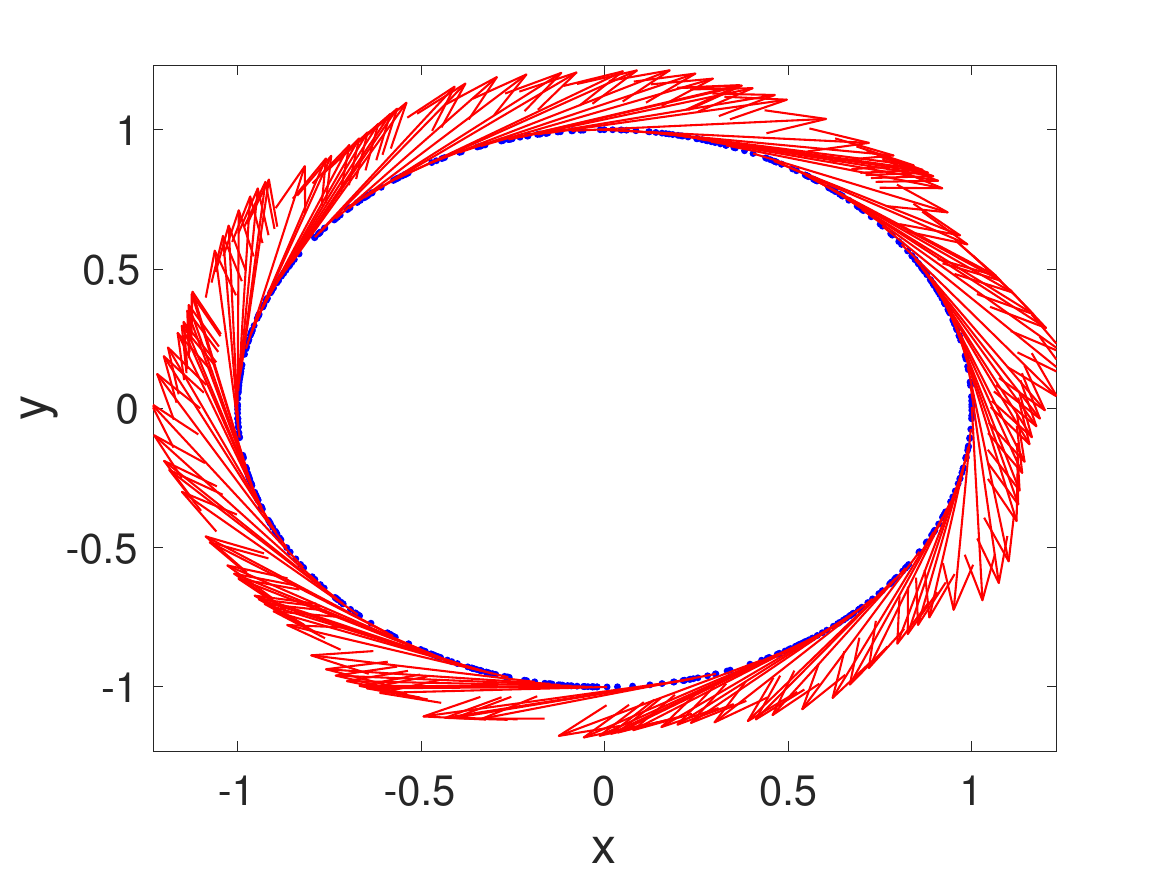}
\includegraphics[width=0.32\textwidth]{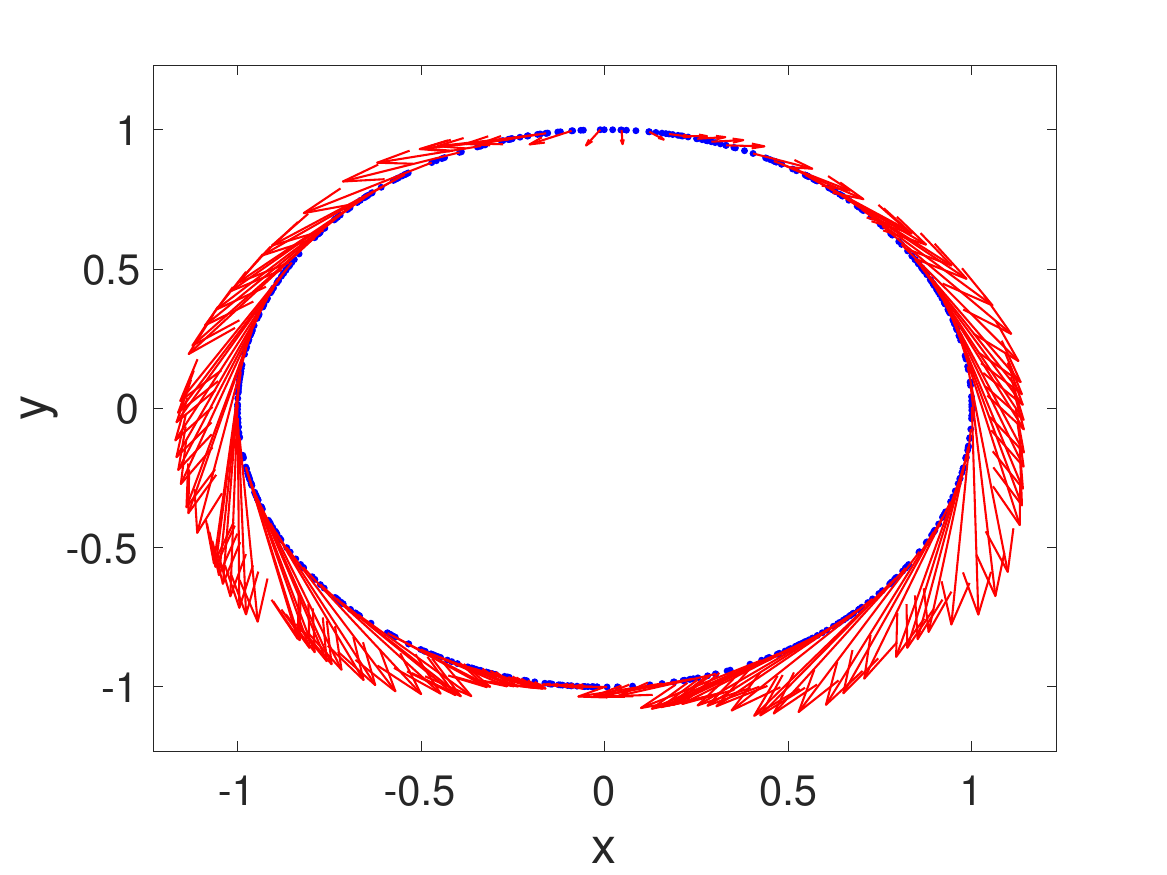}
\caption{\label{S1plots} Top Row, Left: Eigenvalues of the Laplacian on 1-forms estimated using the SEC (red, dashed) compared to the analytical spectrum (gray, solid) for a unit circle.  Middle: The vector field representation of the 1-form corresponding to the smallest eigenvalue ($\nu_1\approx2.46 \times 10^{-6}$) clearly corresponds to the harmonic form $d\theta$. Right: The eigenform of the second smallest eigenvalue ($\nu_2\approx1.06$) is a linear combination of $\sin(\theta) \, d\theta$ and $\cos(\theta)\,d\theta$, which both have eigenvalue $1$. Bottom Row: Computation is repeated using 500 points randomly sampled from the uniform distribution on the circle.  Notice the agreement between the initial part of the spectrum and shape of the smoothest eigenforms. However, the spectrum diverges from the true spectrum faster than for the uniform grid of data due to variance of the estimators.}
\end{figure}

We first consider the circle $S^1$, where there is a a nowhere-vanishing, harmonic 1-form $d\theta$ associated with a canonical angle coordinate $\theta$. Because $S^1 $ is one-dimensional and $d \theta$ is nowhere-vanishing, every smooth 1-form can be represented as $f\, d\theta$, with $f$ a smooth function. Moreover, $d(f\,d\theta) = df \wedge d\theta =0$, so that $ \Delta_1(fd\theta) = \delta d (f\, d\theta)+d\delta(f\, d\theta) = d\delta(f\, d\theta)$, and $\delta d\theta = \Delta \theta = 0$ on $S^1$. Therefore,
\begin{align*}
    \Delta_1(f\,d\theta) &= d\delta(f\, d\theta) = d(-g(df,d\theta) + f\, \delta d \theta) = -d(g(df,d\theta)) \\
    &= -d(d\theta(\nabla f)) = -d\left(\frac{df}{d\theta}\right) = -\frac{d^2f}{d\theta^2}d\theta, 
\end{align*}
meaning that $\Delta_1(fd\theta) = \Delta(f)d\theta$.  Thus, the eigenforms of the 1-Laplacian on $S^1$ are simply $\sin(k\theta)\,d\theta$ and $\cos(k\theta)\,d\theta$, $k\in\mathbb{N}_0$, and the eigenvalues are the same as those of $\Delta$ which are simply $\{0,1,1,4,4,...,k^2,k^2,...\}$.  

We apply the SEC by generating 101 uniformly spaced data points on the unit circle in $\mathbb{R}^2$, and using the method outlined at the beginning of the section.  In the leftmost plot in Fig.~\ref{S1plots}, we compare the analytic eigenvalues of $\Delta_1$ (gray, solid) to the eigenvalues estimated by the SEC (red, dashed).  We also show the first two eigenforms of the $1$-Laplacian in Fig.~\ref{S1plots}, and note that the first eigenform clearly approximates the harmonic form $d\theta$.  Notice that a closed integral curve of a harmonic form should correspond to a unique representative of a nontrivial $1$-homology class, as will be further demonstrated below. In order to demonstrate the effect of random sampling on the manifold (as opposed to the uniform grid) we repeated this experiment using $500$ independent uniformly distributed random points on $S^1$.  More points were required in this example due to the increased variance in the eigenvalues and eigenfunctions produced by the diffusion maps algorithm.  The increased error from diffusion maps causes increased error in the SEC results, as shown in the eigenvalues of the 1-Laplacian.  However, the SEC still obtained good approximations of the eigenforms at the level of having the coarse-grained structure (note the similarity of the first two eigenforms to those produced using the $101$ evenly spaced data points).  This demonstrates both the sensitivity of the SEC regarding eigenvalue precision, as well as the robustness of the SEC with respect to coarse-grained structure, which we believe is very desirable for applications.

\begin{figure}
\raggedleft
\includegraphics[width=0.32\textwidth]{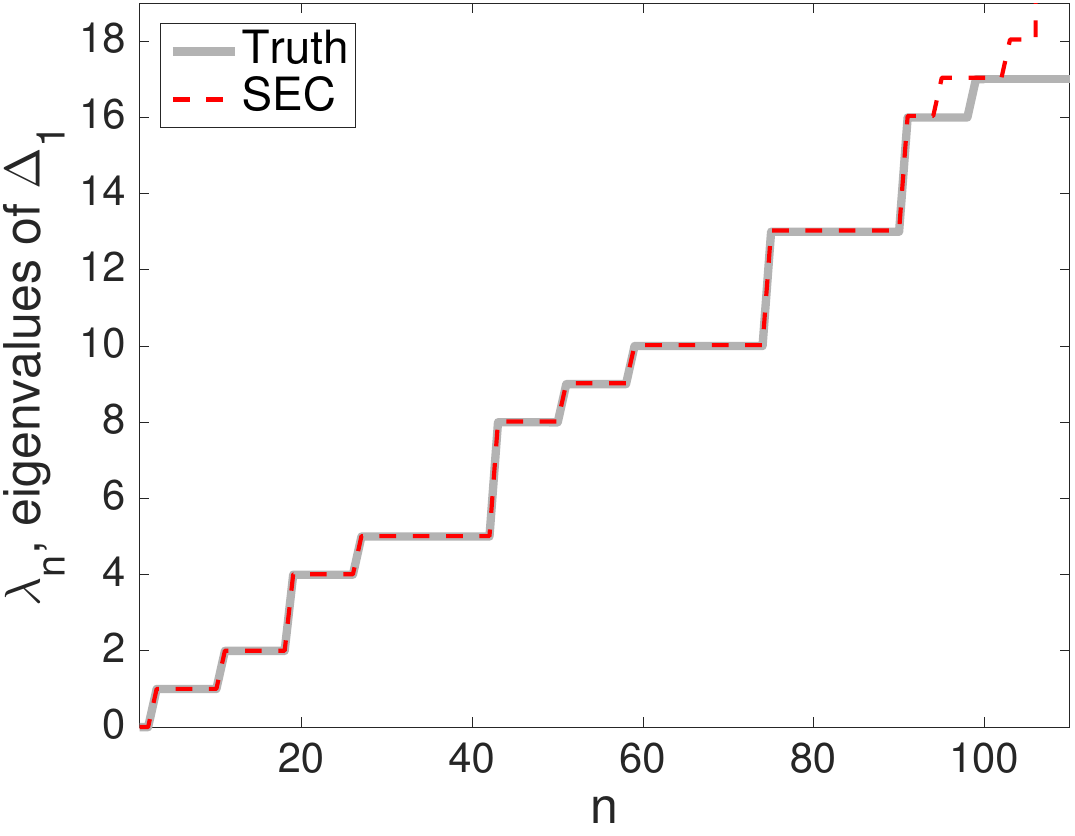}
\includegraphics[width=0.32\textwidth]{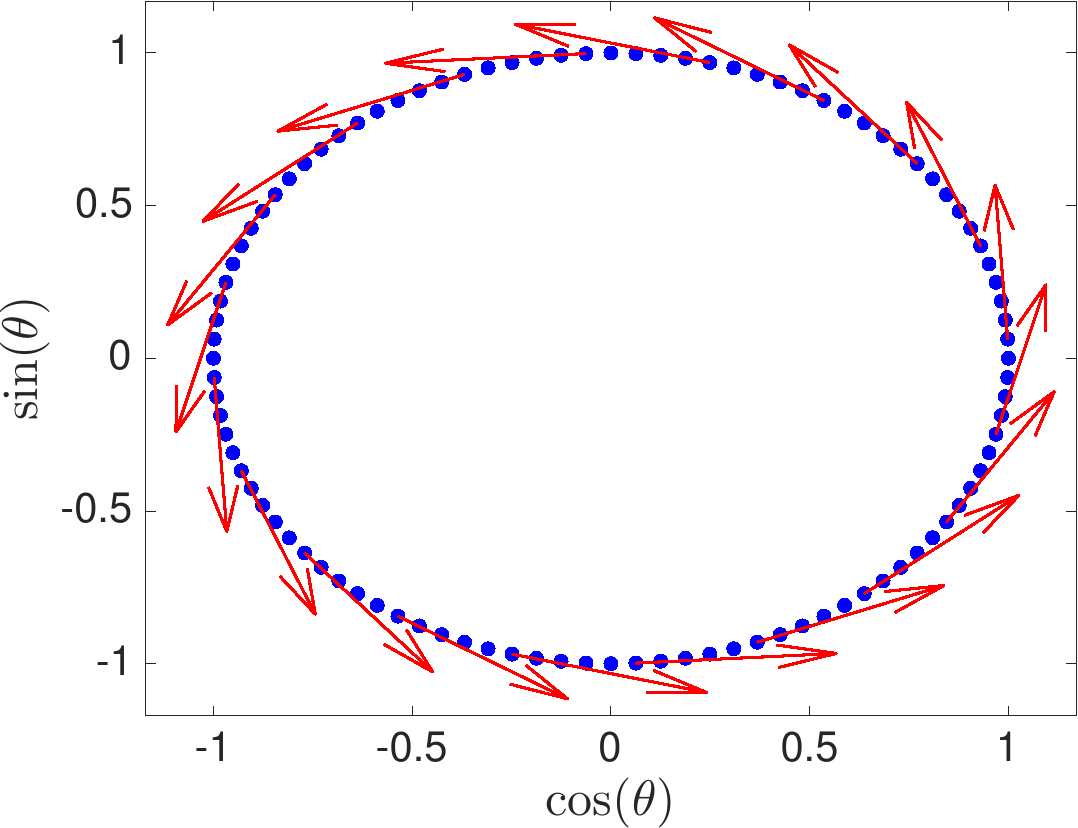}
\includegraphics[width=0.32\textwidth]{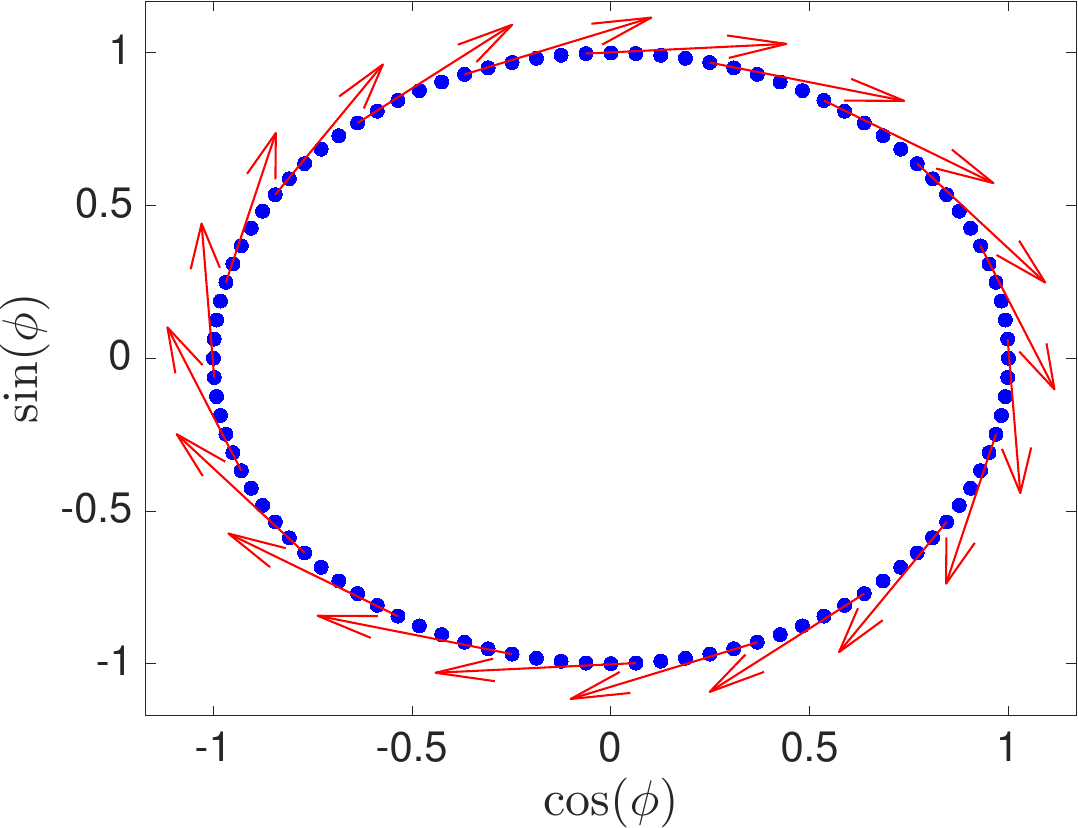} \\
\includegraphics[width=0.32\textwidth]{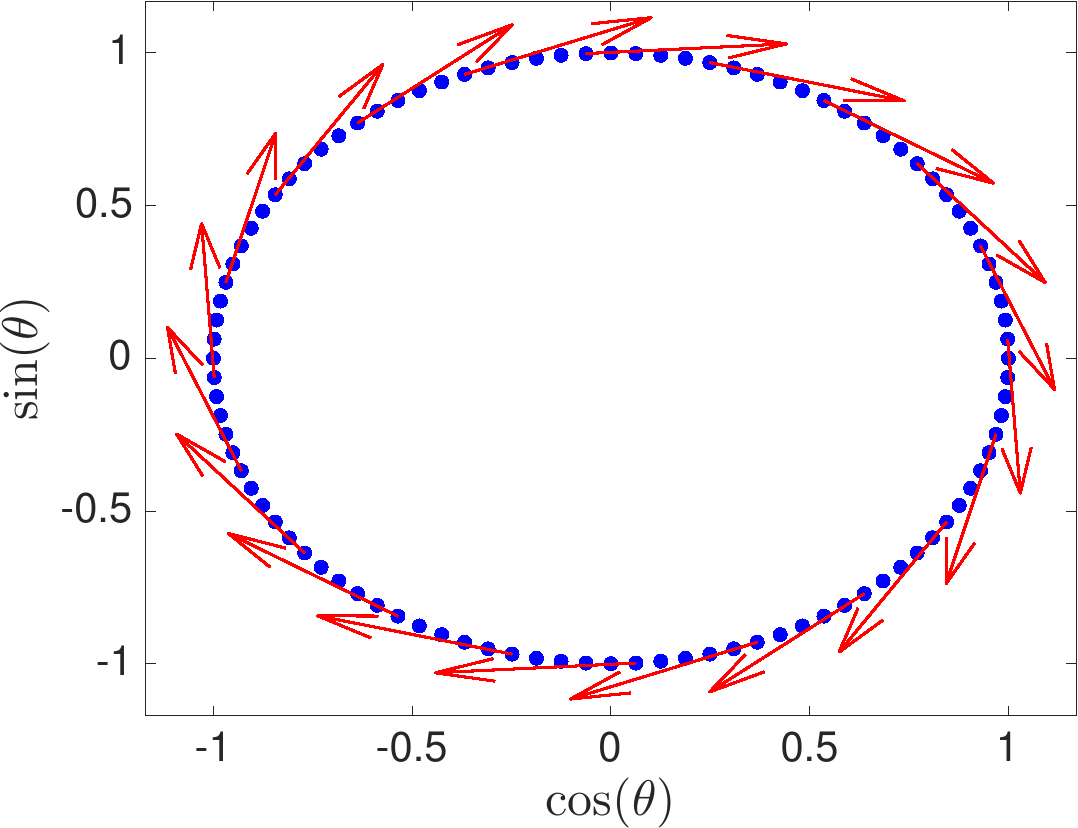}
\includegraphics[width=0.32\textwidth]{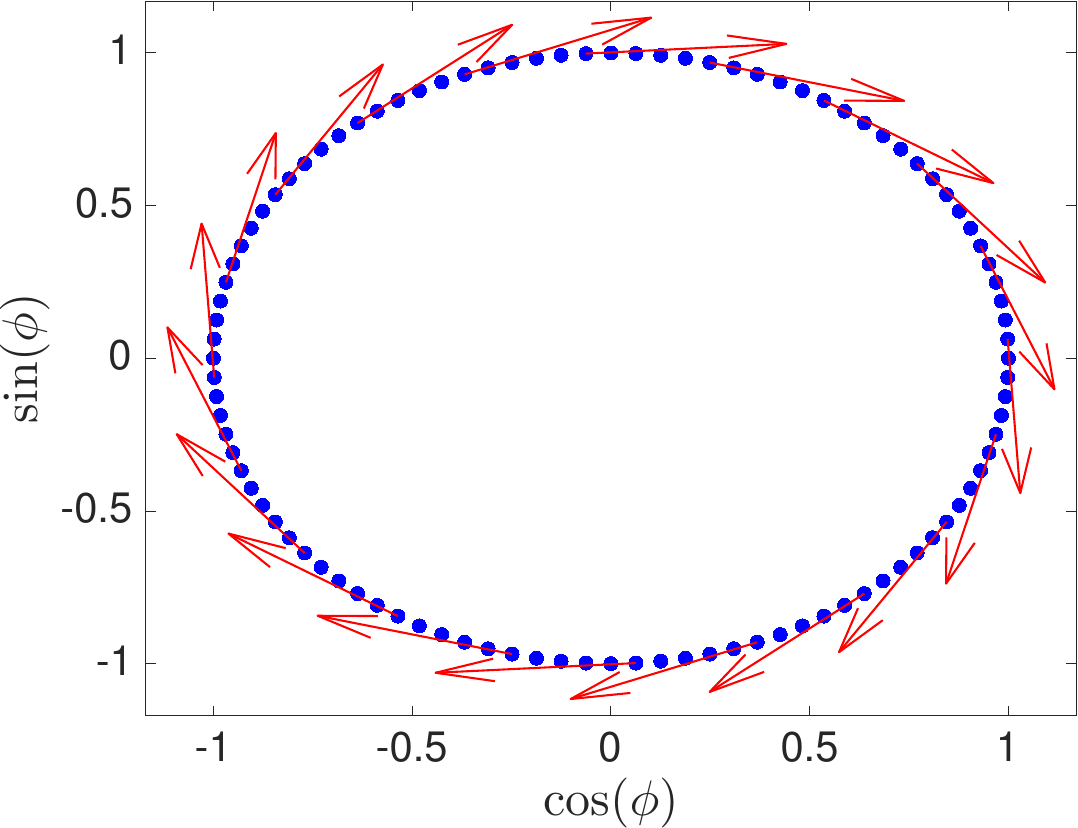}
\caption{\label{T2plots} Top, Left: Eigenvalues of the Laplacian on 1-forms estimated using the SEC (red, dashed) compared to the analytical spectrum (gray, solid) for a flat torus.  Top, Middle/Right: The vector field representation of the 1-form corresponding to the smallest eigenvalue ($\nu_1\approx1.20 \times 10^{-4}$) shown in the embedding coordinates $(\cos(\theta),\sin(\theta))$ (middle) and $(\cos(\phi),\sin(\phi))$ (right); this eigenform approximates $d\theta-d\phi$ (for clarity we only draw arrows every fifth data point). Bottom, Middle/Right: The eigenform of the second smallest eigenvalue ($\nu_2\approx 2.43\times 10^{-4}$) shown in the same coordinates as the first eigenform; this eigenform approximates $-d\theta-d\phi$.  Notice that the two eigenforms together form a basis for the harmonic 1-forms which are the span of $d\theta$ and $d\phi$.}
\end{figure}

We next consider the 1-Laplacian on the flat torus which has two nontrivial $1$-homology classes.  Every smooth 1-form on the flat torus can be written as $f\,d\theta+h\,d\phi$, where $ \theta, \phi $ are canonical angle coordinates and $ f, h $ smooth functions. Considering $f\,d\theta$, we compute
\begin{align*} 
    \delta d(f\, d\theta) &= \delta(df \wedge d\theta) = \delta \left( \frac{\partial f}{\partial \phi}d\phi \wedge d\theta \right) = \delta \left( - \frac{\partial f}{\partial \phi}d\theta \wedge d\phi \right) \\
    &= -\star d\star  \left( -\frac{\partial f}{\partial \phi}d\theta \wedge d\phi \right) \nonumber = \star d  \left( \frac{\partial f}{\partial \phi} \right) = \star \left( \frac{\partial ^2f}{\partial \phi^2}d\phi + \frac{\partial ^2f}{\partial \theta\, \partial \phi} d\theta \right) \nonumber \\ 
    & = - \frac{\partial ^2f}{\partial \phi^2}d\theta + \frac{\partial ^2f}{\partial \theta\, \partial \phi} d\phi, 
\end{align*}
where we choose the ordering $d\theta \wedge d\phi = -d\phi \wedge d\theta$ so that $\star\, d\theta = d\phi$ and $\star d\phi = -d\theta$.  Next,
\begin{align*} d\delta(f\,d\theta) &= -d(\star d\star (f\,d\theta)) = -d\star d(f\,d\phi) = -d\star (df\wedge d\phi) \\
    &= -d\star \left(\frac{\partial f}{\partial \theta}d\theta\wedge d\phi \right)  = -d\left(\frac{\partial f}{\partial \theta}\right) = -\frac{\partial ^2f}{\partial \theta^2}d\theta - \frac{\partial ^2f}{\partial \theta \, \partial \phi} d\phi, 
 \end{align*}
 so that 
 \[ \Delta_1(f\,d\theta) = \delta d(f\,d\theta) + d\delta(f\,d\theta) = \left(-\frac{\partial ^2f}{\partial \phi^2}-\frac{\partial ^2f}{\partial \theta^2}\right)d\theta = \Delta_0(f)\, d\theta,  \]
 and similarly $\Delta_1(h\,d\phi) = \Delta_0(h) d\phi$.  Thus, for an eigenform with eigenvalue $\nu$ we have
 \[ \nu(f\,d\theta+h\, d\phi) = \Delta_1(f\,d\theta+h\,d\phi) = \Delta_0(f)\,d\theta+\Delta_0(h)\,d\phi \]
 which implies that both $f$ and $h$ must be eigenfunctions of the $0$-Laplacian with eigenvalue $\nu$.  Thus, the non-zero eigenvalues of the $1$-Laplacian are the same as those of the $0$-Laplacian up to multiplicity.  Due to the two harmonic forms $d\theta$ and $d\phi$, each eigenvalue of the $1$-Laplacian has double the multiplicity of the same eigenvalue for the $0$-Laplacian.  
 
 To verify the SEC on this example, we generate $\text{10,000}$ uniformly spaced data points on a flat torus in $\mathbb{R}^4$ with the embedding $(\cos\theta,\sin\theta,\cos\phi,\sin\phi)^\top$.  In the top left of Fig.~\ref{T2plots}, we plot the analytic spectrum of the $1$-Laplacian (gray, solid curve) along with the SEC approximation of the spectrum (red, dashed curve).  We also show the vector fields corresponding to the first two SEC eigenforms. The latter approximate $d\theta-d\phi$ and $-d\theta-d\phi$, which span the space of harmonic forms which is the span of $d\theta$ and $d\phi$.  
 
 Notice that this example requires 100 times more data than the circle (data required grows exponentially in the intrinsic dimension) for the diffusion maps algorithm to yield the same accuracy for the $0$-Laplacian eigenfunctions and eigenvalues.  This also means that the diffusion maps algorithm takes significantly longer to run on this larger data set.  Crucially, the SEC still only uses $M=20$ (and $100$ eigenfunctions to compute the $c$ tensor) so the matrices in the SEC were the same size for the torus example as for the circle example.  Thus, following the initial diffusion maps step, the SEC algorithm runtime is the same in the torus example as for the circle example. This demonstrates how the SEC formulation allows us to decouple the representation of differential forms and their associated Laplacian operators from the amount of data. In particular, larger data sets are only needed in the initial diffusion maps step to obtain the best possible estimate of the $0$-Laplacian eigenfunctions and eigenvalues.  Since many of these eigenfunctions and eigenvalues are known to be poor estimates, we can  set $M$ much less than the number of data points to obtain high quality $1$-Laplacian representations.
 
 \subsection{Topological features and eigenforms via the SEC}
 \begin{table}
     \caption{\label{spectra} First eight eigenvalues of the Laplacian on 1-forms estimated using the SEC for a M\"obius band, torus (standard embedding in $\mathbb{R}^3$), a genus two surface (two-holed torus), a two-dimensional sphere in $\mathbb{R}^3$, and the Lorenz~63 attractor (L63).  Notice that the eigenvalues close to zero represent harmonic forms, and the number of the eigenvalues close to zero matches the first Betti numbers for the manifolds which are 1, 2, 4, and 0 respectively.  While the Lorenz~63 attractor is not a manifold, a coarse approximation as a manifold would suggest a Betti number of 2 due to the two holes.}
\begin{center}
    \small
{\renewcommand{\arraystretch}{1.1}
\begin{tabular*}{\linewidth}{@{\extracolsep{\fill}}ccccc}
\hline
M\"obius &	Torus 	& Genus 2 & Sphere & L63 \\
	\hline\hline
0.0242 &	0.0040 & 0.0021 & 1.9349 & 0.0011 \\
1.0415 &	0.0093 & 0.0026 & 1.9521 & 0.0017 \\
 1.0449 &   0.2574 & 0.0026 & 1.9781 & 0.0030 \\
  3.8684 &  0.2575 & 0.0041 & 1.9817 & 0.0072 \\
  3.8948 &  0.2575 & 0.0893 & 2.0042 & 0.0105 \\
  8.0352 &  0.2587 & 0.0901 & 2.0172 & 0.0109 \\
  8.1018 &  0.8061 & 0.2151 & 5.8001 & 0.0205 \\
  8.9369 &  0.8067 & 0.2175 & 5.8142 & 0.0262 \\
\hline
\end{tabular*}
}
\end{center}
\end{table}

\begin{figure}
\centering
\includegraphics[width=0.45\textwidth]{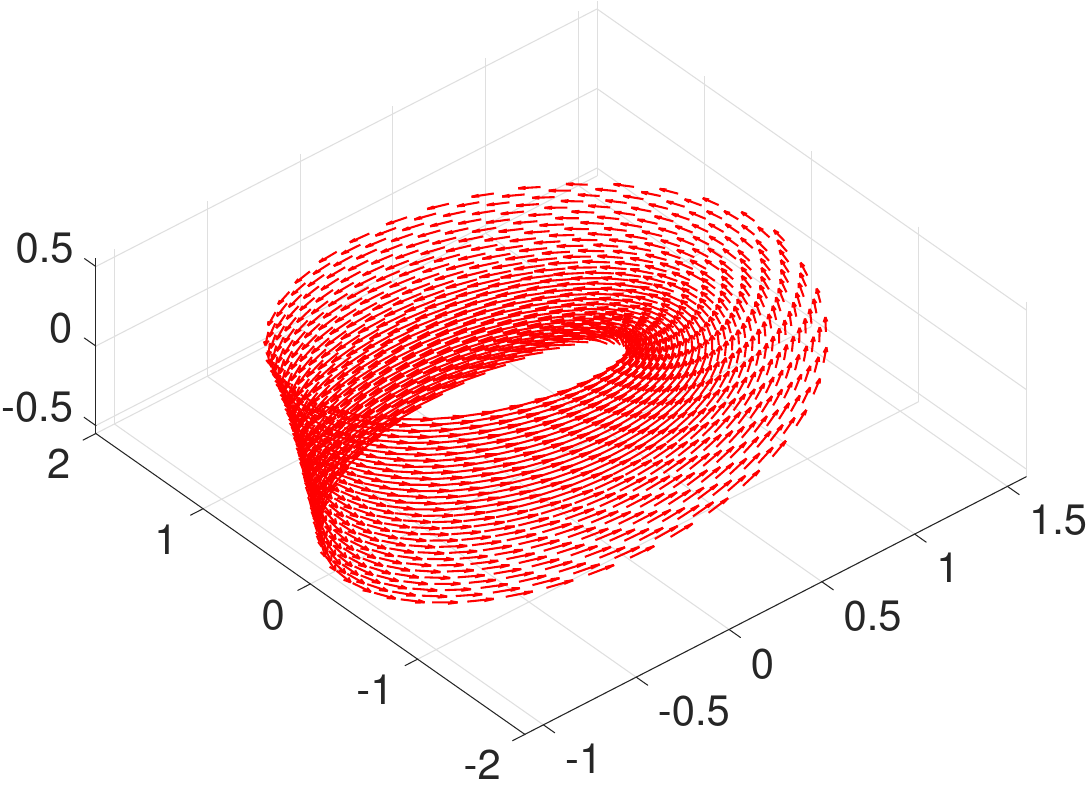}
\includegraphics[width=0.45\textwidth]{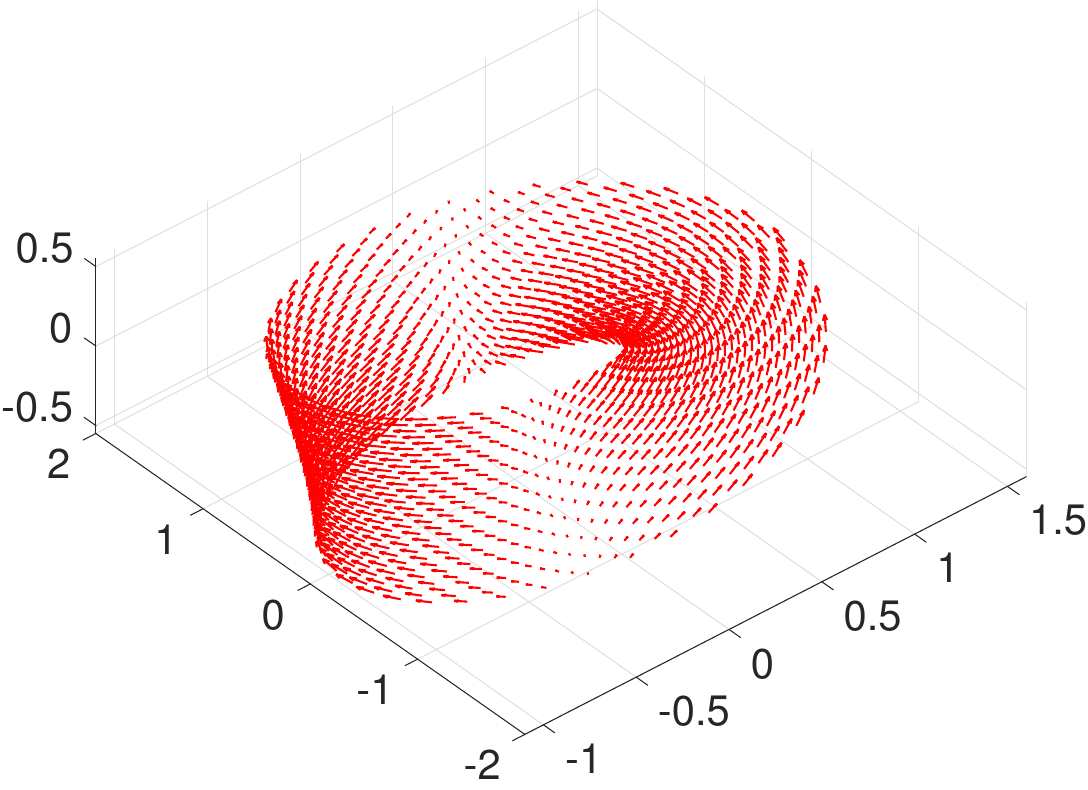}
\caption{\label{Meigenformplots} First two eigenforms on the M\"obius band. The first represents the one 1-homology class.}
\end{figure}

\begin{figure}
\centering
\includegraphics[width=0.45\textwidth]{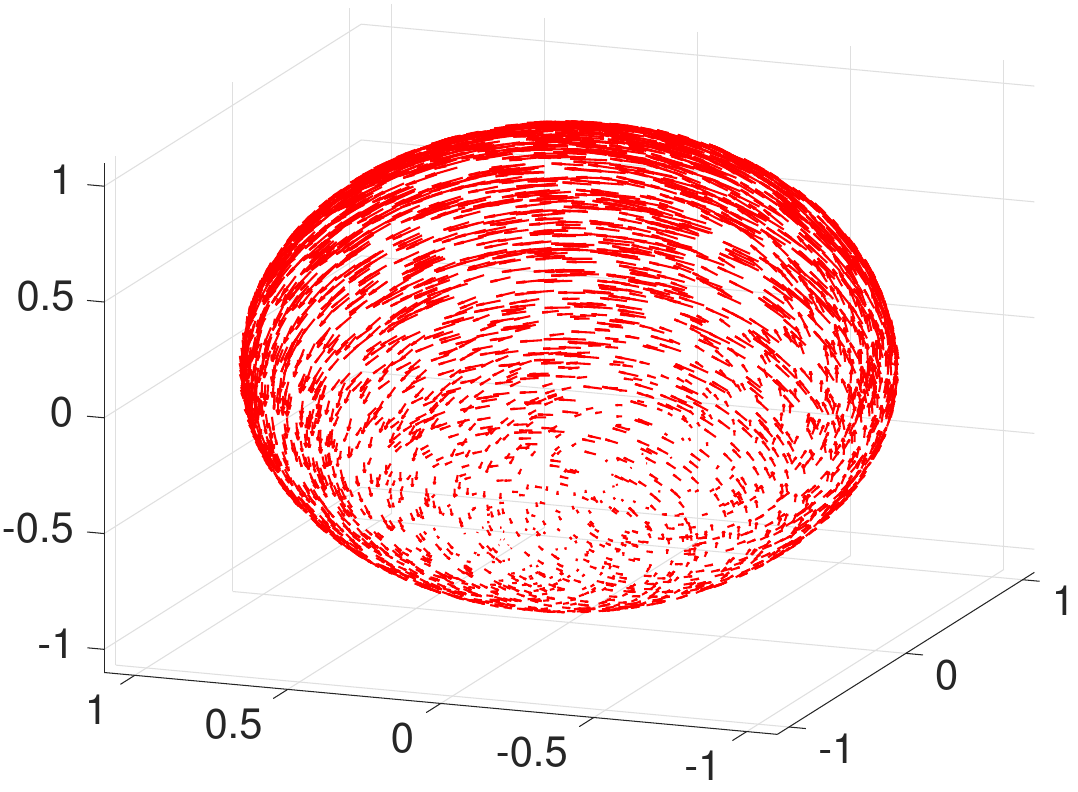}
\includegraphics[width=0.45\textwidth]{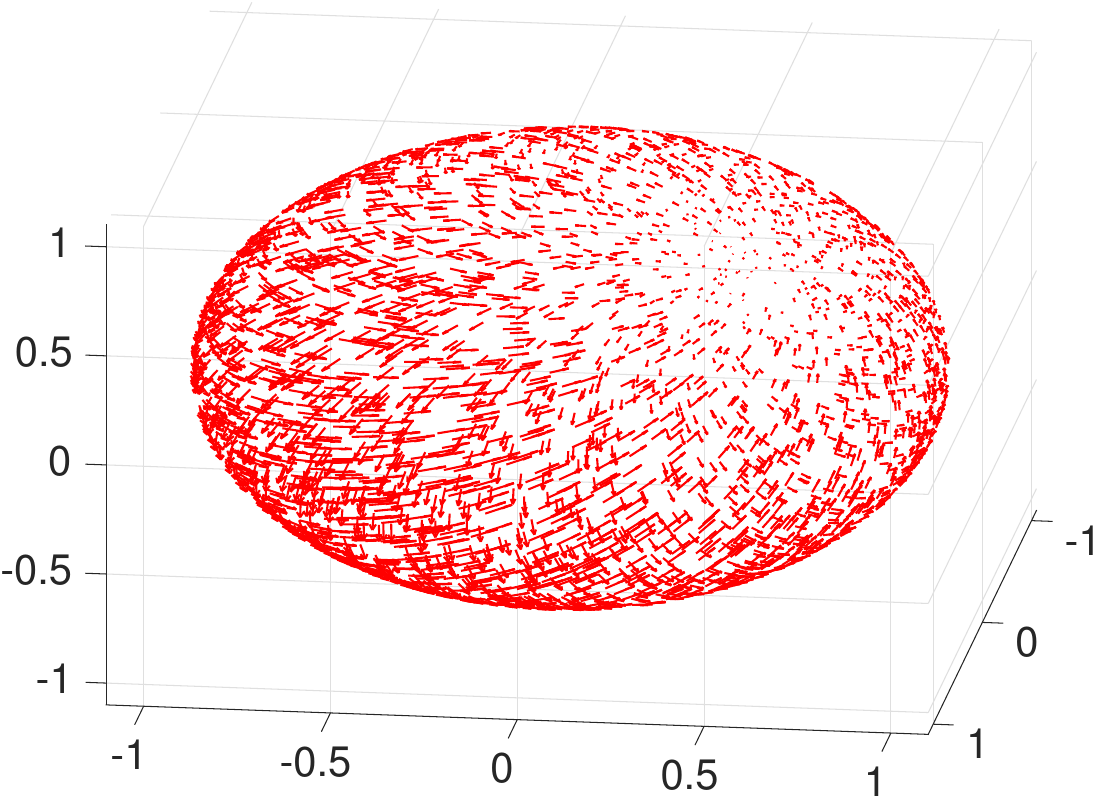}
\caption{\label{S2eigenformplots} First two eigenforms on the sphere $S^2$. Notice that the SEC constructs valid smooth vector fields which must vanish on the sphere.}
\end{figure}

\begin{figure}
\centering
\includegraphics[width=0.45\textwidth]{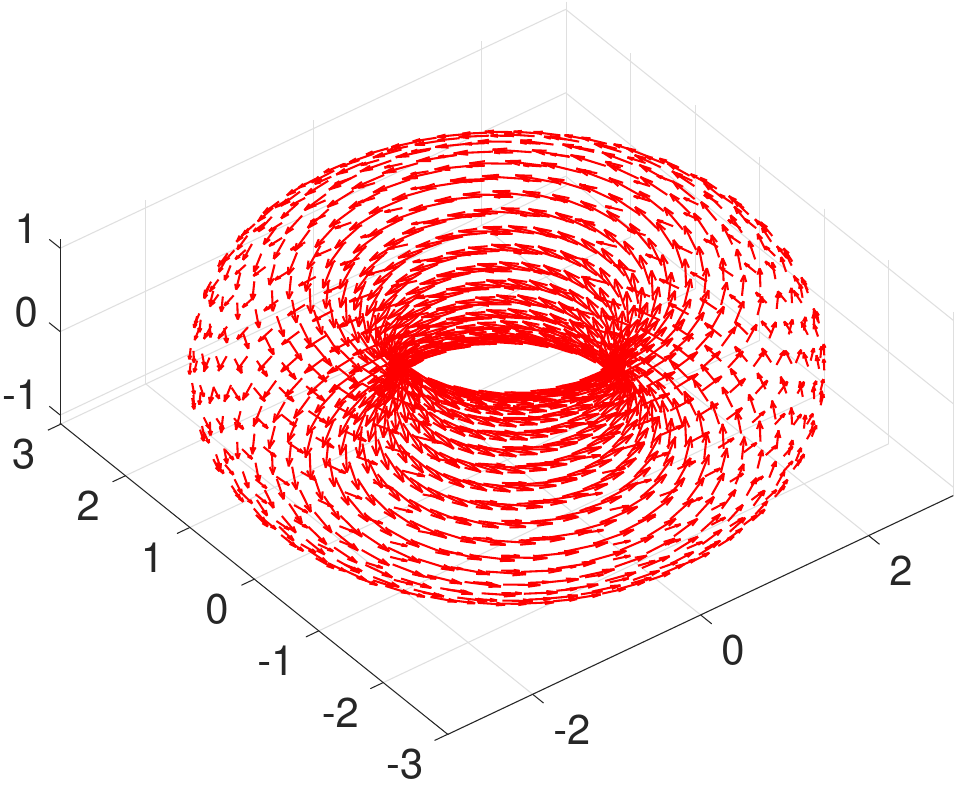}
\includegraphics[width=0.45\textwidth]{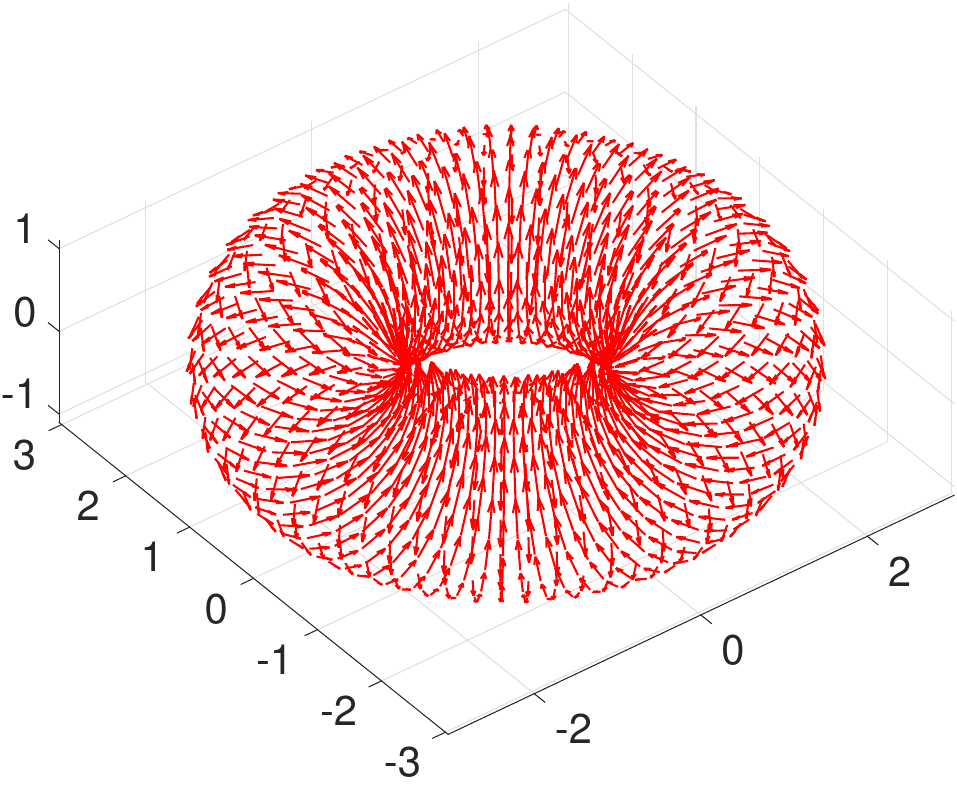}
\includegraphics[width=0.45\textwidth]{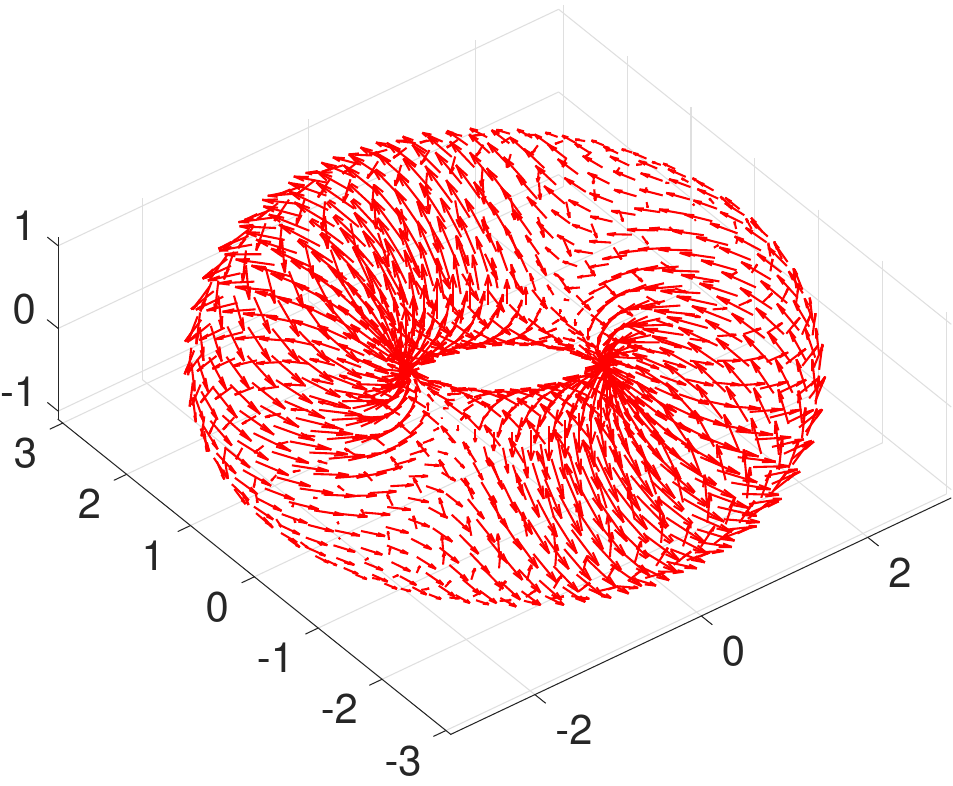}
\includegraphics[width=0.45\textwidth]{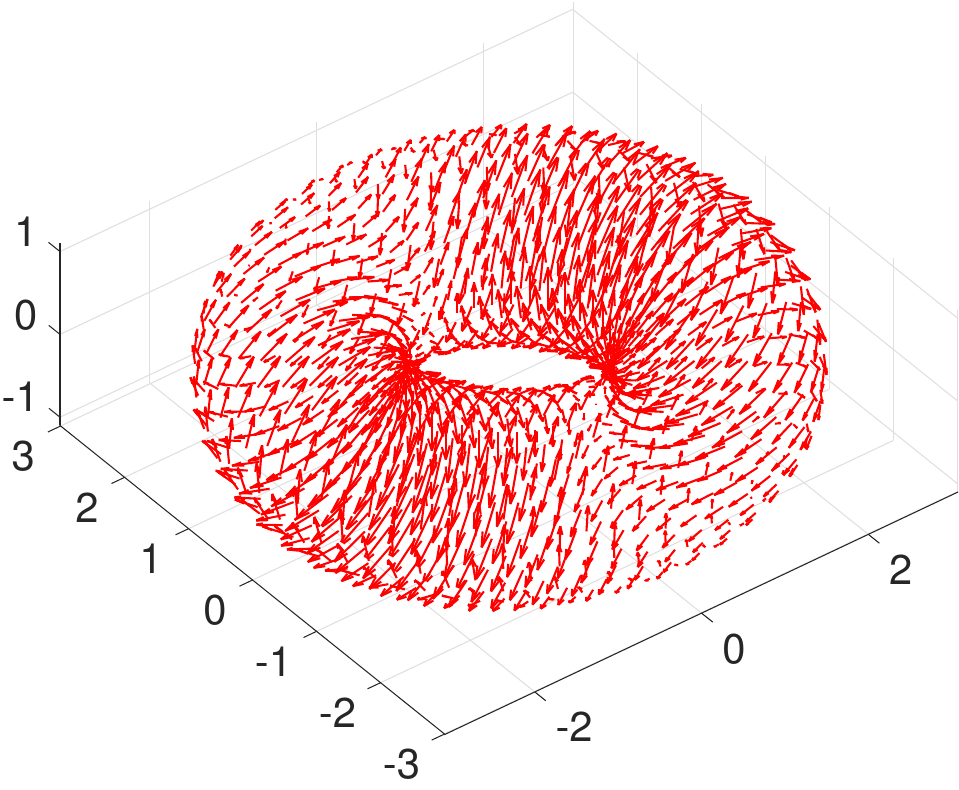}
\caption{\label{T2eigenformplots} First four eigenforms on the standard embedding of the torus in $\mathbb{R}^3$.  Notice that the first two represent the two 1-homology classes.}
\end{figure}

\begin{figure}
\centering
\includegraphics[width=0.45\textwidth]{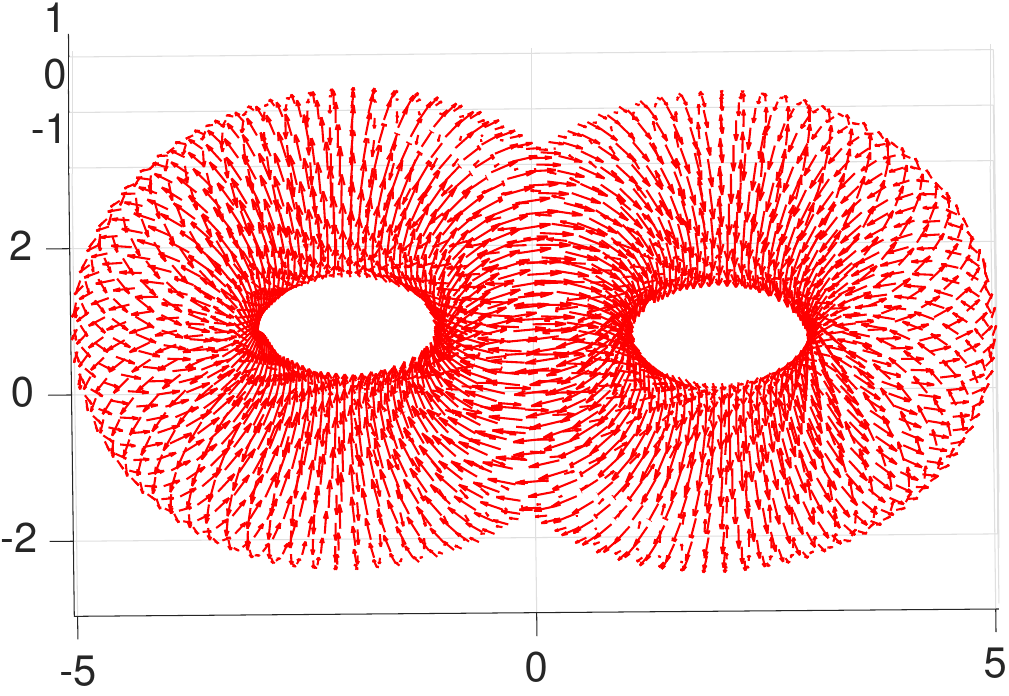}
\includegraphics[width=0.45\textwidth]{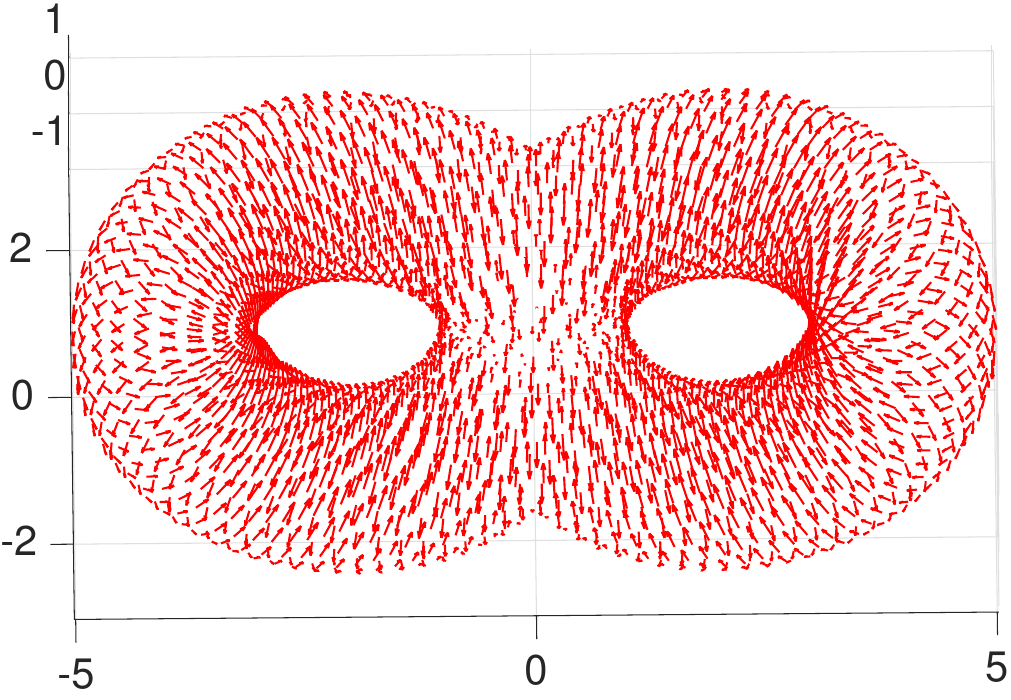}
\includegraphics[width=0.45\textwidth]{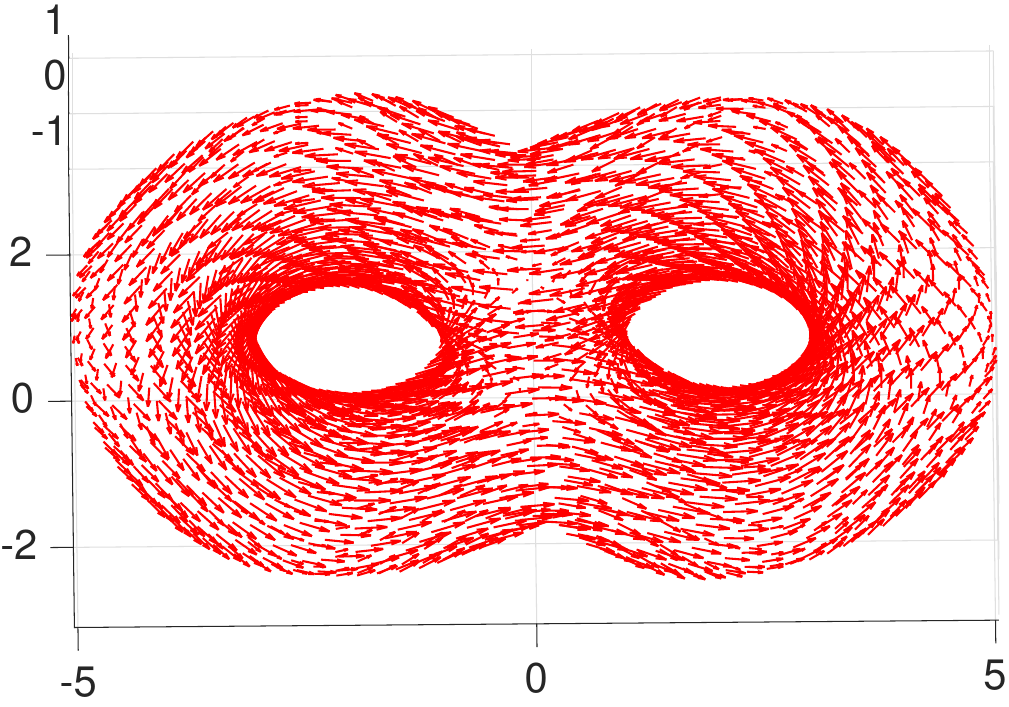}
\includegraphics[width=0.45\textwidth]{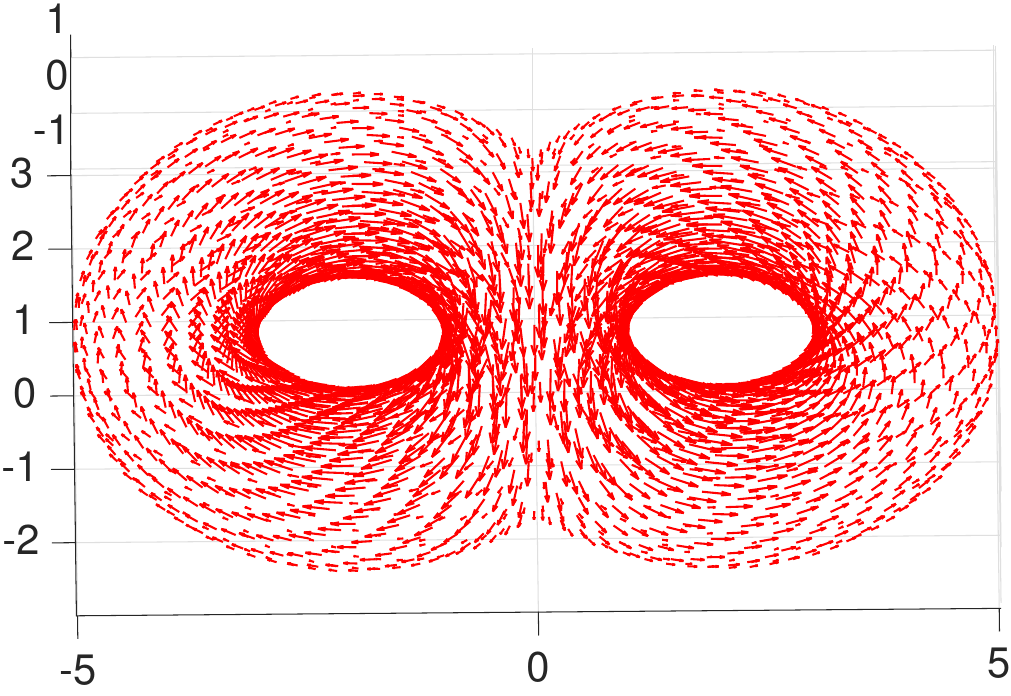}
\caption{\label{T22eigenformplots} First four eigenforms on the genus two surface, representing the four 1-homology classes.}
\end{figure}

We now apply the SEC on several surfaces embedded in $\mathbb{R}^3$ in order to demonstrate its connection to the manifold homology.  We also demonstrate the algorithm on a data set which is not sampled from a manifold to show that the SEC has potential applications even when the assumption of an underlying manifold does not hold.  While the results on the circle and flat torus above used large uniformly distributed data sets to validate the algorithm, in this case we work with much smaller data sets that may not be uniformly distributed with respect to the volume form. Details on the sampling procedure employed in each example can be found in the code in the online supplementary material.
 
We first consider four surfaces, namely a M\"obius band, sphere, torus, and a genus-2 surface.  Our goal for these examples is to show that the SEC correctly captures the coarse topological features of these smaller nonuniform data sets.  The $1$-homology of these manifolds corresponds to the kernel of the $1$-Laplacian so that the first Betti number should be equal to the multiplicity of the zero eigenvalue of $\Delta_1$. This is shown in Table~\ref{spectra} if we consider the eigenvalues closest to zero to represent $1$-homology.  We can also make this connection by visualizing the eigenforms via their corresponding vector fields as shown in Figs.~\ref{Meigenformplots}--\ref{T22eigenformplots}.  Following the vector field corresponding to a harmonic form (having eigenvalue zero) should generate a closed curve, which is a representative of a unique homology class.  We also show in Figs.~\ref{Meigenformplots}--\ref{T22eigenformplots} that the SEC approximations are smooth forms, demonstrated both by the smooth changes in the arrows, as well as the fact that the corresponding vector fields each vanish at some point on the sphere (since there are no smooth non-vanishing vector fields on a sphere). 

We should note that the M\"obius band is a manifold with boundary, which violates the assumption of manifold without boundary used in our theoretical derivations. Nevertheless, the algorithm produces reasonable results which suggest the theory may be able to be extended to manifolds with boundary.  A potential theoretical barrier is the requirement that the gradients of the eigenfunctions must span every tangent space (see Lemma~\ref{lemmaSpanning}), yet the estimator of the Laplacian from diffusion maps produces Neumann eigenfunctions \cite{CoifmanLafon06}, which will all have gradient orthogonal to the normal direction.  While this failure to span the tangent spaces may preclude the corresponding 1-forms $b^{ij}$ and $ \widetilde b^{ij}$ from providing frames for the $L^2$ space $H_1 $, it is still possible that they provide frames for higher-order Sobolev spaces associated with Neumann boundary conditions. In particular, it is possible that an analog of Theorem~\ref{thmFrameSob} holds for the Neumann $H_{1,1}$ Sobolev space, so that the corresponding Galerkin scheme from Definition~\ref{defnEigGalerkin} approximates the spectrum of the Neumann Laplacian on 1-forms.  While a rigorous study of the properties of SEC in the Neumann setting lies outside the scope of this work, the numerical results for the M\"obius band described above are consistent with this behavior. If approximation of the full $H_1 $ space, or Sobolev spaces associated with other boundary conditions is desired, a possible remedy would be to employ the normal direction estimator and/or the distance to boundary estimator which were recently developed in \cite{berry2017density}.  These estimators could be used to ensure that the full tangent space is spanned on the boundary (and near the boundary, which may be an issue for finite data sets). 
 
As a final example, we demonstrate the SEC on data sampled from the chaotic attractor of the Lorenz~63 dynamical system \cite{Lorenz63} (we refer to this set as ``L63''), a fractal set having no differential structure. While there is no exterior calculus defined on L63, it is a well-defined compact subset on $\mathbb{R}^3$ \cite{Tucker99} (with an induced metric topology), and exhibits certain coarse-grained topological features, most notably a hole in each of the two lobes of the attractor as shown in Fig.~\ref{l63}. Moreover, the diffusion maps and SEC algorithms can easily be applied to data sampled from this set (or indeed any data set in a metric space).  The SEC spectrum for the $1$-Laplacian is shown in Fig.~\ref{spectra}, and the first two eigenvalues are very close to zero, while the corresponding eigenforms, shown in Fig.~\ref{l63} seem to capture these two coarse topological features.  
 
In summary, the examples in this section show that the SEC can generate a collection of 1-forms (equivalently, vector fields) which can be used as a basis for vector fields defined on the data set. These vector fields are ordered by smoothness based on their corresponding eigenvalue (Dirichlet energy; see Section~\ref{laplacianOverview}), and higher eigenvalues correspond to more oscillatory vector fields.  Moreover, our theory shows that in the limit of large data and large $M$ (number of eigenfunctions of $\Delta_0$ used), the SEC basis corresponds to the natural basis for square integrable 1-forms.  However, even for small nonuniform data sets the SEC reflects coarse topological features of the underlying continuous space, which indicates that the SEC-derived approximations could be useful even outside of the large data limit.

  \begin{figure}
  \centering
\includegraphics[width=0.49\textwidth]{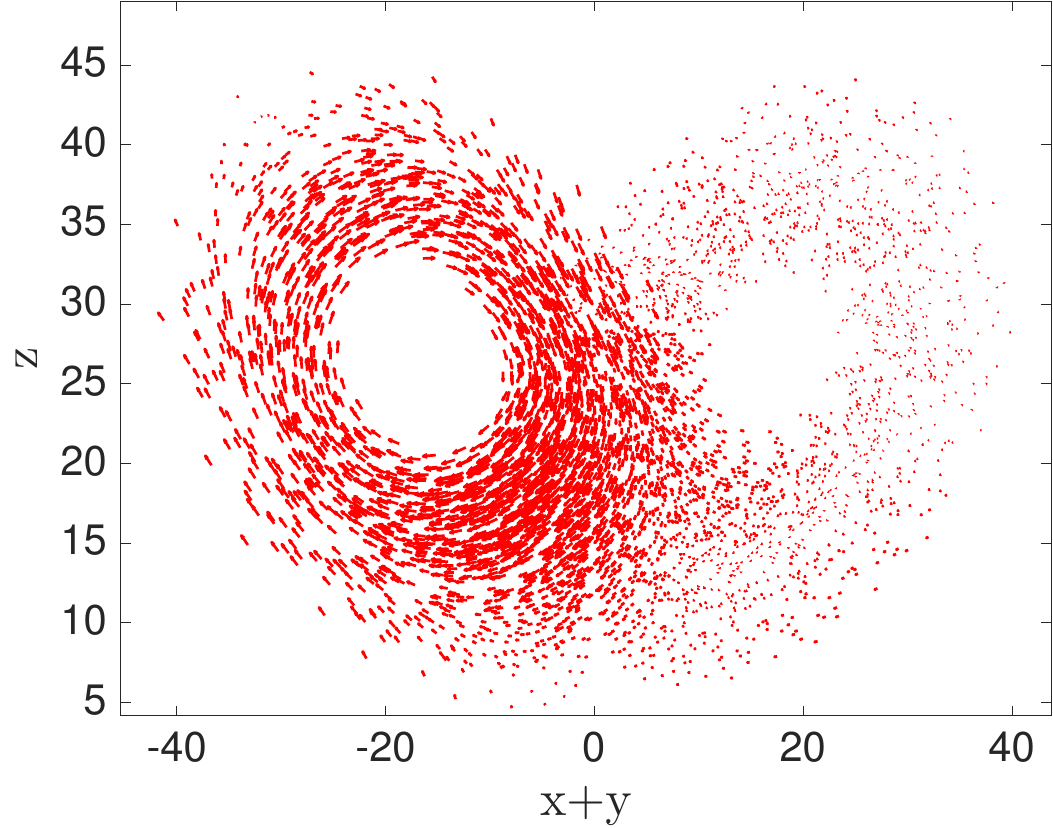}
\includegraphics[width=0.49\textwidth]{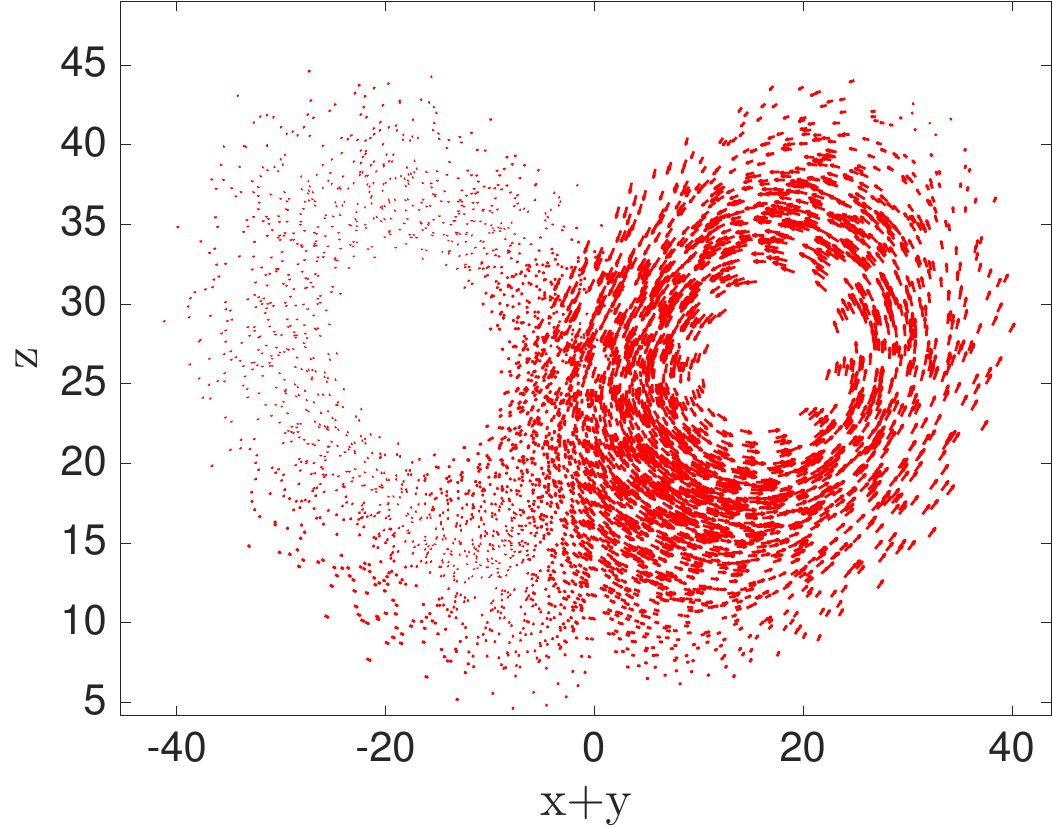}
\caption{\label{l63} First two eigenforms of the SEC $1$-Laplacian on 10000 data points sampled from the Lorenz~63 attractor.  Notice that closed integral curves of these vector fields correspond to independent representatives of the two ``holes'' in the attractor.}
\end{figure}

\section{\label{secDiscussion}Discussion}

In this paper, we have developed a spectral framework, called spectral exterior calculus (SEC), for the exterior calculus on Riemannian manifolds. A central underpinning of this approach is a family of frames for $L^2$ and higher-order Sobolev spaces of vector fields and differential forms, built entirely from the eigenvalues and eigenfunctions of the Laplacian on functions. By virtue of this construction, our framework lends itself well to data-driven approximation of the objects of interest in exterior calculus, such as vector fields and differential forms, as well as operators acting on these objects (e.g., the Laplacian on forms), requiring no additional information beyond point-clouds of data. In particular, SEC extends the applicability of the extensive array of graph-theoretic techniques for pointwise and spectral approximation of the Laplacian on functions \cite{BelkinNiyogi03,Singer06,CoifmanLafon06,VonLuxburgEtAl08,BerrySauer16b,HeinEtAl05,BerryHarlim16,belkin2007convergence,shi2015convergence,trillos2018error,trillos2018variational,BerrySauer16} to learning problems involving higher-order objects, with rigorous convergence guarantees in the asymptotic limit of large data. Crucially, by relying solely on approximations of the eigenvalues and eigenfunctions of the Laplacian on functions, SEC decouples the computational complexity of approximation of vector fields, forms, and related operators from the number of samples and ambient data space dimension.  

Another key aspect of SEC is its focus on $L^2$-convergent, as opposed to pointwise-convergent approximations, as the latter typically require additional structures such as simplicial complexes that are difficult to estimate from data alone. Here, we have shown that our frames for $H^1$ Sobolev spaces of 1-forms lead naturally to Galerkin approximation schemes for the eigenvalue problem of the Laplacian on 1-forms, which are provably well-posed by classical results on variational formulations of elliptic eigenvalue problems \cite{BabuskaOsborn91}. These techniques extend previously developed data-driven Galerkin approximation techniques for unbounded operators on functions \cite{GiannakisEtAl15,Giannakis19,DasGiannakis19} to the setting of 1-forms. These Galerkin methods for the Laplacian on one-forms, and more broadly approximation techniques for unbounded operators in exterior calculus, crucially depend on the availability of well-conditioned approximation spaces for $H^p$ spaces of higher regularity than $L^2$. In the framework of SEC, such spaces are naturally constructed through singular value decomposition of sparse Grammian matrices associated with $H^p$ frame elements, which are computable via closed-form expressions involving the eigenvalues of the Laplacian on functions, and inner-product relationships between products of the corresponding eigenfunctions. Once computed, the Galerkin-approximated Laplacian eigenforms can be visualized through their dual vector fields, reconstructed via a data-driven spectral approximation of the pushforward map. In addition to frame representations, SEC provides data-driven representations of vector fields as operators on functions, which is useful for tasks such as dynamical systems forecasting \cite{BerryEtAl15,Giannakis19}, among other applications.  

We have demonstrated the efficacy of SEC approximations to the eigenvalues and eigenforms of the 1-Laplacian, and their associated vector field duals, in a suite of examples involving orientable (circle, flat and curved 2-tori, 2-sphere, genus-2 surface) and non-orientable (M\"obius band) smooth manifolds, as well as a fractal set having no differentiable structure (the L63 attractor). In the circle and flat-torus examples, where analytical expressions for the 1-Laplacian eigenvalues and eigenforms are available, we found that SEC accurately approximates the leading 50 to 100 eigenvalues using a modest number (20) of 0-Laplacian eigenfunctions to build the SEC frames, and a moderate dimension ($40$ for circle and $107$ for the flat torus) of the corresponding Galerkin approximation spaces. In the curved torus, sphere, and M\"obius band examples (where analytical expressions for the eigenvalues and eigenforms of the 1-Laplacian are not readily available), we demonstrated that the SEC results are consistent with with the 1-homology of these manifolds. In particular, the number of eigenvalues of the SEC-approximated 1-Laplacian numerically close to zero was found to be equal to the first Betti number of the manifold under study, and the corresponding eigenforms were found to generate closed curves (through the integral curves of their vector field duals), representing 1-homology classes. It should be noted that due to the presence of a boundary, the M\"obius example lies outside the theoretical domain of applicability of the SEC formulation developed here, but the numerical results were found to be qualitatively consistent with the application of Neumann boundary conditions, implicitly enforced through diffusion maps. Moreover, even though our theory was developed in the smooth-manifold setting, the SEC-derived eigenforms for the L63 attractor were also found to be qualitatively consistent with the 1-homology of a coarse graining of the attractor. In this example, we obtained two eigenforms at small corresponding eigenvalue, generating closed curves around each of the ``holes'' in the lobes of the attractor (Fig.~\ref{l63}), which is what one would intuitively expect for the L63 topology.        

There are several avenues of future research stemming from this work. First, it would be of interest to develop SEC approximation schemes for other operators of interest in exterior calculus, such as the Hodge star operator and the Lie derivative. Among other applications, such approximation schemes are likely to be of interest in dynamical systems modeling. For instance, for a dynamical system on  a manifold generated by a vector field $ v $, there is an associated dynamical system acting on the tangent bundle, which can be thought of as a pointwise linearization of the system. In this setting, the Lie derivative $ \mathcal{ L }_v$  on vector fields generates the action of this system on sections of the tangent bundle, and is known to have useful spectral properties characterizing growth rates of perturbations \cite{ChiconeSwanson81}. Approximating these spectra from data would thus provide a computationally efficient empirical method for characterizing the principal modes of instability in datasets generated by dynamical systems; a topic of considerable current interest that has spurred the development of powerful geometrical approaches such as the theory of covariant Lyapunov vectors \cite{GinelliEtAl07}. 

Other research directions would involve extending SEC to to manifolds with boundary, or non-smooth topological spaces. In particular, our results for the L63 attractor suggest that a spectral formulation of exterior calculus, possibly with a very similar structure to the one presented here, may be possible provided one can construct an appropriate diffusion operator generalizing the notion of the Laplacian on Riemannian manifolds. To that end, it is worthwhile noting that the Markov integral operators employed here to approximate the heat semigroup and its associated Laplacian are, in fact, well defined as operators on $C^0 $ and $L^2 $ function spaces on Borel measure spaces without differentiable structure. Moreover, recent work \cite{DasGiannakis19} has shown that, at least in some scenarios, such operators can be consistently approximated from data under natural assumptions (e.g., data sampled from a dynamical system possessing a physical ergodic invariant measure, such as the L63 system studied here). Yet, a key step relevant to the SEC framework, namely the consistent approximation of the Laplacian from the heat semigroup, achieved by taking $ \epsilon \to 0 $ kernel bandwidth limits, has, to our knowledge, no known generalization to non-differentiable spaces. In such a setting, the challenge would be to approximate the generator of an appropriate $C^0$ diffusion semigroup based on kernel integral operators approximating that semigroup (akin to the operators $ \tilde P_\epsilon $ in Section~\ref{secKernelApprox} approximating the heat semigroup). Such a diffusion operator would lead to an analog of the product rule for the 0-Laplacian on smooth functions through its associated symmetric bilinear form \cite{BakryEtAl14}, providing one of the necessary ingredients to building the SEC framework. However, at present, the construction of this operator and its approximation remain elusive.   

Finally, while the present SEC formulation has focused heavily on $H^p$ Sobolev spaces and their associated notions of regularity based on Dirichlet energy, it should be noted that there exist analogs of the RKHSs associated with the heat kernel on functions in the setting of differential forms. In addition to providing well-defined notions of regularity through the corresponding RKHS norms, the fact that RKHS spaces have continuous pointwise evaluation functionals allows for pointwise- (as opposed to $L^2$-) convergent approximations when working in such spaces, with well-known applications in smoothing and interpolation. These properties motivate the construction of frames for RKHSs of forms, and their use in Galerkin approximation schemes for operators on forms, analogous to the schemes presented here utilizing Sobolev spaces.

%%      ---------------------------------------------------------------------
%%      ------------------------- APPENDIX (OPTIONAL) -----------------------
%%      ---------------------------------------------------------------------
        
%%      If you have one appendix, uncomment the line \appendix and add
%%      a \section{ *** APPENDIX TITLE ***}. If you have more than
%%      one, uncomment the line \appendices and add a \section{ ***
%%      APPENDIX TITLE ***} command for each appendix title.

%\appendix
\appendices

\section{Discrete incompatibility of product rule and Leibniz rule}\label{productRules}

Let $ f $, $ h $, and $ b $ be arbitrary smooth, real-valued functions on a Riemannian manifold. The product rule for the Laplacian operator can be written as 
\[ \Delta(fh) = f\Delta h + h\Delta f - 2 \grad f(h), \]
where $\grad f(h) = \grad f \cdot \grad h$.  The Leibniz rule for vector fields states that 
\[ v(fh) = fv(h) + hv(f), \]
and since $\grad f$ is a vector field, we can easily derive a triple product rule for the Laplacian
\begin{align} \Delta(fhb) &= f\Delta(hb) + hb\Delta f - 2 \grad f(hb) \nonumber \\ 
&= f\Delta(hb) + hb\Delta f - 2 h\grad f(b) - 2 b \grad f(h) \nonumber \\
&= f\Delta(hb) + hb\Delta f - h f\Delta b - h b\Delta f + h\Delta(fb) - bf\Delta h - bh\Delta f + b\Delta (hf). \nonumber
\end{align}
Rearranging the triple product rule, we have
\[ 0 = \Delta(fhb) -f\Delta(hb)-h\Delta(fb)-b\Delta(fh)+fh\Delta b +fb \Delta h + hb\Delta f. \]
Now assume that $L$ is a discrete Laplacian operator (meaning a matrix), without any assumption on the form of the representation of vector fields.  If we assume that gradient fields are represented in such a way that both the product rule for the Laplacian and the the Leibniz rule hold, then the same derivation will produce the triple product rule for the discrete Laplacian.  However, if the triple product rule holds on the standard basis vectors $\{e_i\}$, then we find that
\[ 0 = L(e_i e_j e_k) - e_i L (e_j e_k) - e_j L (e_i e_k) - e_k L(e_i e_j) + e_i e_j L e_k + e_i e_k L e_j + e_j e_k L e_i, \]
where vector-vector products are componentwise.  Setting $i=j=k$, we find
\[ 0 = L(e_i) - 3 e_i L (e_i e_i) + 3 e_i e_i L e_i = Le_i \]
and since $Le_i =0$ for all $i$, we conclude that the only matrix which satisfies the triple product rule is the zero matrix.  This shows that no discrete representation of vector fields can simultaneously satisfy both the Laplacian product rule and the Leibniz rule.

\section{Computations and derivations of formulas}\label{derivations}

In this appendix, we derive the formulas for the inner products and Dirichlet energies of the SEC frame elements for 1-forms listed in Table~\ref{fig3} (\ref{appLapl}). In addition, we discuss how these formulas can be extended to higher-order $k$-forms (\ref{appKForms}).  Since the SEC frame elements are real, in these derivations we will consider that all Hilbert spaces of functions, vector fields, and $ k$-forms are over the real numbers, i.e., there will be no complex conjugation in the corresponding inner products.  

\subsection{\label{appLapl}Inner products and Dirichlet forms for 1-forms}

 Recall that the Hodge inner product  $1$-forms $\omega,\nu$ is defined by 
 \[ \langle \omega,\nu\rangle_{H_1} =\int_{\mathcal{M}} \omega \wedge \star \nu=\langle 1,\eta(\omega,\nu)\rangle_{H}=\langle 1,\omega^\sharp \cdot \nu^\sharp\rangle_{H}, \]
 where we abbreviate the Riemannian inner product for vector fields $w,v$ by $w \cdot v \equiv g(w,v)$.  Similarly for 1-forms $\omega,\nu$ we will abbreviate the Riemannian inner product by $\omega \cdot \nu \equiv \eta(\omega,\nu) = g(\omega^\sharp,\nu^\sharp)$.  In particular, if $\omega=f\,dh$ and $\nu=\alpha\, d\beta$, then the Hodge inner product can be further simplified as
 \[ \langle f\,dh,\alpha\, d\beta \rangle_{H_1} = \langle d h \cdot d \beta,f\alpha \rangle_{H} = \frac{1}{2}\langle h\Delta \beta + \beta \Delta h - \Delta(h\beta),f \alpha \rangle_{H}, \]
 and substituting eigenfunctions $b^i$ of $\Delta$ into this formula, we define
 \begin{align} G_{ijkl} &\equiv \langle b^i\, db^j,b^k\, db^l \rangle_{H_1} = \frac{1}{2}\langle b^j \Delta b^l + b^l \Delta b^j - \Delta(b^jb^l),b^ib^k \rangle_{H} \nonumber \\ 
 &= \frac{1}{2}\langle \lambda_l b^jb^l + \lambda_j b^lb^j - \Delta\left(\sum_s c_{ljs}b^s \right),b^ib^k \rangle_{H} \nonumber \\ 
 &= \frac{1}{2}(\lambda_l+\lambda_j)c_{ijkl}-\sum_s \lambda_s c_{ljs}c_{iks}, \nonumber \\
 &= \frac{1}{2}\sum_s (\lambda_l+\lambda_j - \lambda_s) c_{ljs}c_{iks}, \end{align}
 where $c_{ijk}\equiv\langle b^ib^j,b^k\rangle_{H}$ and $c^0_{ijkl}\equiv\langle b^ib^j,b^kb^l\rangle_{H}=\sum_s c_{ijs}c_{kls}$ (note that these are invariant to permutations of indices).  Notice that the Gramm matrix $ \hat G_{ijkl} = \langle \widehat b^{ij}, \widehat b^{kl} \rangle_{H_1}$ of the antisymmetric elements is easy to compute from $ G_{ijkl}$ since
 \[ \hat G_{ijkl} = G_{ijkl}+G_{jilk}-G_{ijlk}-G_{jikl}, \]
 so next we will consider the computation of the values $ E_{ijkl} $ of the Dirichlet form.
  
Recalling the formula for the $1$-Laplacian $\Delta_1 = d \delta + \delta d$, for $1$-forms $\omega,\nu$ we can write
\begin{align}\label{del1sum} \langle \omega,\Delta_1\nu\rangle_{H_1} = \langle \delta \omega,\delta \nu\rangle_{H} + \langle d\omega,d\nu\rangle_{H_2} \end{align}
where $d$ is the exterior derivative, and $\delta$ its formal adjoint, the codifferential. For $1$-forms $\omega=f\,dh$ and $\nu = \alpha\, d\beta$, we have
\begin{align}\label{del1term1} \langle d(f\,dh),d(\alpha\, d\beta) \rangle_{H_2} &= \langle df \wedge dh,d\alpha \wedge d\beta \rangle_{H_2} = \langle 1,\eta(df \wedge dh,d\alpha \wedge d\beta)\rangle_{H_2} \nonumber \\
    &= \left\langle 1,\det\left[ \begin{array}{cc} d f \cdot d \alpha & d f \cdot d \beta \\ d h \cdot d \alpha & d h \cdot d \beta \end{array} \right]\right\rangle_{H}  \nonumber \\
&=  \langle d f \cdot d \alpha , d h \cdot d \beta \rangle_{H} - \langle  d f \cdot d \beta, d h \cdot d \alpha\rangle_{H}.
\end{align}

We recall that the codifferential acting on $k$-forms is $\delta = (-1)^{m(k-1)+1}\star d \, \star $, where $d$ is the exterior derivative acting on $m-k$ forms.  We use $m$ instead of $d$ for the dimension of the manifold in this section to avoid confusion with the exterior derivative. For a function $f$ and $1$-form $\omega$, we have
\begin{align} \nonumber \delta(f\omega) &= - \star  d \, \star  (f\omega) = - \star  d(f \star  \omega) \\
    \nonumber &= - \star  (df \wedge \star \omega + f d\star \omega) = -\star (df\wedge \star \omega) - f \star  d \, \star  \omega \nonumber \\ 
    &= -\eta(df,\omega) + f\, \delta \omega = -d f \cdot \omega + f\, \delta \omega, \label{eqAppCodiff}
\end{align}
so in particular if $\omega = dh$ we find (using $\Delta = \delta d$) 
\begin{align*}
    \delta(f\,dh) &= -d f \cdot d h + f \Delta h \\
    &= \frac{1}{2}(\Delta(fh) - f\Delta h - h\Delta f) + f\Delta h = \frac{1}{2}(\Delta(fh)+f\Delta h-h\Delta f). 
\end{align*}
Using this formula we can simplify the inner product
\begin{align}  \langle \delta (f\,dh),\delta (\alpha\, d\beta) \rangle_{H} &= \frac{1}{4}\langle \Delta(fh)+f\Delta h-h\Delta f,\Delta(\alpha \beta)+\alpha\Delta \beta-\beta\Delta \alpha\rangle_{H} \nonumber \\
    &= \langle -d f \cdot d h + f\Delta h,-d \alpha \cdot d \beta + \alpha\Delta \beta \rangle_{H} \nonumber \\
    \nonumber &= \langle d f \cdot d h,d \alpha \cdot d \beta \rangle - \langle d f \cdot d h,\alpha\Delta \beta\rangle_{H} - \langle d \alpha \cdot d \beta, f\Delta h\rangle_{H} \\
    \label{del1term2} & \quad + \langle f\Delta h,\alpha\Delta \beta\rangle_{H}.
  \end{align}
  Now that both summands in \eqref{del1sum} have been written in terms of $\Delta$, we can combine \eqref{del1term1} and \eqref{del1term2} and several cancellations yield,
  \begin{align*} 
  \langle f\,dh,\Delta_1(\alpha\, d\beta)\rangle_{H_1} &= \langle d f \cdot d \alpha , d h \cdot d \beta \rangle_{H} - \langle  d f \cdot d \beta, d h \cdot d \alpha\rangle_{H} \\
  & \quad + \langle d f \cdot d h,d \alpha \cdot d \beta \rangle_{H}  - \langle d f \cdot d h,\alpha\Delta \beta\rangle_{H} \\
  & \quad - \langle d \alpha \cdot d \beta, f\Delta h\rangle_{H} + \langle f\Delta h,\alpha\Delta \beta\rangle_{H}. \end{align*}
  When the functions are eigenfunctions of $\Delta$, we find
  \begin{align*} \nonumber d b^i \cdot d b^j &= \frac{1}{2}\left(b^i \Delta b^j + b^j\Delta b^i - \Delta(b^ib^j)\right) = \frac{1}{2}\left((\lambda_i+\lambda_j)b^ib^j -\Delta(b^ib^j)\right) \\
      &= \frac{1}{2}\sum_s (\lambda_i+\lambda_j-\lambda_s)c_{ijs}b^s, 
  \end{align*}
so that we can define
\begin{align*} F_{ijkl} &= \langle d b^i \cdot d b^j,d b^k \cdot d b^l\rangle_{H} = \frac{1}{4}\sum_s (\lambda_i+\lambda_j-\lambda_s)(\lambda_k +\lambda_l - \lambda_s)c_{ijs}c_{kls}  \nonumber \\
&= \frac{1}{4}\left[(\lambda_i+\lambda_j)(\lambda_k+\lambda_l)c_{ijkl} - (\lambda_i+\lambda_j+\lambda_k+\lambda_l)\sum_s \lambda_s c_{ijs}c_{kls} \right. \\
& \quad \left. + \sum_s \lambda_s^2 c_{ijs}c_{kls} \right].
\end{align*}
Moreover, recalling that
\[ G_{ijkl} = \frac{1}{2}\sum_s (\lambda_l+\lambda_j - \lambda_s) c_{ljs}c_{iks} = \frac{1}{2}(\lambda_l+\lambda_j)c^0_{ijkl} - \frac{1}{2}\sum_s \lambda_s c_{ljs}c_{iks}, \]
we simplify the above formulas as 
\[ \langle d(b^i db^j),d(b^k db^l) \rangle_{H_2} = F_{ikjl} - F_{iljk} \]
and 
\[ \langle \delta (b^idb^j),\delta (b^k db^l) \rangle_{H} = F_{ijkl} - \lambda_l G_{kilj} - \lambda_j G_{ikjl} + \lambda_j \lambda_l c^0_{ijkl}, \]
so that
\begin{align} E_{ijkl} &= \langle b^i db^j,\Delta_1 (b^k db^l) \rangle_{H_2} \\
    &= F_{ikjl} - F_{iljk} + F_{ijkl} - \lambda_l G_{kilj} - \lambda_j G_{ikjl} + \lambda_j \lambda_l c_{ijkl} \nonumber \\
&= c^0_{ijkl}\left(\frac{1}{4}(\lambda_i + \lambda_k)(\lambda_j +\lambda_l) - \frac{1}{4}(\lambda_i + \lambda_l)(\lambda_j +\lambda_k) + \frac{1}{4}(\lambda_i + \lambda_j)(\lambda_k +\lambda_l)\right) \nonumber \\
&\quad +c^0_{ijkl}\left(\lambda_j\lambda_l - \frac{1}{2}\lambda_l(\lambda_i+\lambda_j) - \frac{1}{2}\lambda_j(\lambda_k + \lambda_l) \right) \nonumber \\
&\quad + \frac{1}{4}(\lambda_i+\lambda_j+\lambda_k+\lambda_l)(c_{iljk}^{1} - c_{ikjl}^{1} - c_{ijkl}^{1}) + \frac{1}{2}(\lambda_j+\lambda_l)c_{ijkl}^{1} \nonumber \\
& \quad + \frac{1}{4}(c_{ikjl}^{2} + c_{ijkl}^{2} - c_{iljk}^{2}) \nonumber \\
&= \frac{1}{4}(\lambda_i+\lambda_j+\lambda_k+\lambda_l)(c_{iljk}^{1} - c_{ikjl}^{1} - c_{ijkl}^{1}) + \frac{1}{2}(\lambda_j+\lambda_l)c_{ijkl}^{1} \nonumber \\
& \quad + \frac{1}{4}(c_{ikjl}^{2} + c_{ijkl}^{2} - c_{iljk}^{2}) \nonumber \\
&= \frac{1}{4}\left[ (\lambda_i+\lambda_j+\lambda_k+\lambda_l)(c_{iljk}^{1} - c_{ikjl}^{1}) + (\lambda_j+\lambda_l-\lambda_i - \lambda_k)c_{ijkl}^{1} \right.\nonumber \\
& \quad + \left. (c_{ikjl}^{2} + c_{ijkl}^{2} - c_{iljk}^{2})\right] \nonumber
\end{align}
A straightforward computation then shows that the $c^0_{ijkl}$ coefficients exactly cancel, and we define 
\[ c_{ijkl}^{1} \equiv \sum_s \lambda_s c_{ijs}c_{skl} \hspace{20pt}\textup{and}\hspace{20pt} c_{ijkl}^{2} \equiv \sum_s \lambda_s^2 c_{ijs}c_{skl}, \]
which also allows us to write $G_{ijkl} = \frac{1}{2}((\lambda_l+\lambda_j)c_{ijkl} - c_{ikjl}^{1}).$
Finally, to compute the Dirichlet form for the antisymmetric elements, we note that 
\[ E_{ijkl} - E_{jikl} =  \frac{1}{2}\left[ (\lambda_i+\lambda_j+\lambda_k+\lambda_l)(c_{iljk}^{1} - c_{ikjl}^{1}) + (\lambda_j-\lambda_i)c_{ijkl}^{1}   + (c_{ikjl}^{2}- c_{iljk}^{2}) \right] \]
and
\begin{align*}
E_{ijlk} - E_{jilk} &= \frac{1}{2}\left[ -(\lambda_i+\lambda_j+\lambda_k+\lambda_l)(c_{iljk}^{1} - c_{ikjl}^{1}) + (\lambda_j-\lambda_i)c_{ijkl}^{1} \right.  \\
& \quad \left. - (c_{ikjl}^{\lambda^2}- c_{iljk}^{2}) \right],  
\end{align*}
leading to
\begin{align*} \hat E_{ijkl} &\equiv \langle b^i db^j - b^j db^i, \Delta_1(b^k db^l - b^l db^k )\rangle =  E_{ijkl}-E_{jikl} - (E_{ijlk} - E_{jilk}) \nonumber \\ &= (\lambda_i+\lambda_j+\lambda_k+\lambda_l)(c_{iljk}^{1} - c_{ikjl}^{1})  + (c_{ikjl}^{2}- c_{iljk}^{2}).  
\end{align*}
These formulas are summarized in Table~\ref{fig3} in Section \ref{laplacianOverview}.

\subsection{\label{appKForms}Extension to $k$-forms}

Next, we briefly summarize how we can extend the SEC formulas derived in \ref{appLapl} to higher-order forms.  Let  $b^I = b^{i_0} db^{\tilde I} \equiv b^{i_0} db^{i_1} \wedge \cdots \wedge db^{i_k}$ be a $k$-form frame element ($\tilde I = (i_1,...,i_k)$ and $I=(i_0,...,i_k)$).  As we will see below, the Hodge inner product of the exterior derivatives is easily computed as a determinant of inner products of 1-forms.  The more complex term is the Hodge inner product of the codifferential terms.  To understand this, we need to generalize the product rule for the codifferential as
\begin{align*} 
\delta (b^I) &= (-1)^{m(k-1)+1}\star d\star (b^{i_0} db^{\tilde I}) = (-1)^{m(k-1)+1}\star d(b^{i_0} \star db^{\tilde I})\\
 & = (-1)^{m(k-1)+1}\star (db^{i_0} \wedge \star db^{\tilde I} + b^{i_0} d \star  db^{\tilde I})  \nonumber \\ 
&=  (-1)^{m(k-1)+1}\star (db^{i_0} \wedge \star db^{\tilde I}) + b^{i_0} \delta(db^{\tilde I}). \nonumber \end{align*}
We can now reduce $\delta(db^{\tilde I})$ by rewriting $db^{\tilde I} = d(b^{i_1}db^{\hat I})$ ,where $\hat I = (i_2,...,i_k)$, so that
\[ \delta_k(db^{\tilde I}) = \delta_k d_{k-1} (b^{i_1}db^{\hat I}) = \Delta_{k-1}(b^{i_1}db^{\hat I}) - d_k\delta_{k-1}(b^{i_1}db^{\hat I}). \]
Notice that in the above formula we have reduced the problem of computing the $k$-codifferential to computing the $(k-1)$-Laplacian and the $(k-1)$-codifferential.  Since we have shown how to compute the codifferential on $1$-forms, this strategy can be used to lift the SEC formulas to higher-order forms, although the formulas become quite complicated.  We carried this out for $2$-forms, but the derivations are quite long and a closed formula for the general case remains elusive.

Thus, while in principle the iterative formula above allows us to lift our formulation to $k$-forms, we do not yet have a closed formula for the inner product
\[ \langle b^I, \Delta_k b^J \rangle_{H_{k}} = \langle d(b^I),d(b^J)\rangle_{H_{k+1}} + \langle \delta(b^I),\delta(b^J)\rangle_{H_{k-1}}. \]
However, the first term above is simply the integral of the determinant of the matrix of pairwise inner products $db^{i_s}\cdot db^{j_r}$ for $s,r \in \{0,...,k\}$.  The key to computing this term is the formula
\[ (db^i \cdot db^j)(db^k \cdot db^l) = \left(\sum_q g_{ijq}b^q\right)\left(\sum_n g_{kln}b^n\right) = \sum_{q,n} g_{ijq}g_{kln}b^q b^n, \]
and iterating the above we can expand the product as
\[ \prod_{s,r=0}^k db^{i_s}\cdot db^{j_r} = \sum_{n=0}^{k^2} \prod_{s,r=0}^k g_{i_s j_r q_n}b^{q_n}  . \]
The integral of these products can then be represented in terms of the integrals of products of eigenfunctions, which can be computed from the $c$ tensor. We briefly summarize these formulas for the SEC on general $k$-forms in Table~\ref{kformtable}.

\begin{table}
\caption{\label{kformtable} Table of formulas for the SEC on $k$-forms.  We abbreviate the multi-indices $I=(i_0,...,i_k)$ and $\tilde I = (i_1,...,i_k)$ and similarly $J=(j_0,...,j_k)$ and $\tilde J = (j_1,...,j_k)$. We use $[db^{i_q}\cdot db^{j_n}]$ to abbreviate the $k \times k$ matrix with $(q,n)$-th entry given by $db^{i_q}\cdot db^{j_n}$. Finally, note that $db^I = db^{i_0}\wedge \cdots \wedge db^{i_k}$.}
\begin{center}
    \footnotesize
{\renewcommand{\arraystretch}{1.5}
\begin{tabular*}{\linewidth}{@{\extracolsep{\fill}}l c c}
\hline
Object & Symbolic & Spectral  \\
\hline\hline
Multiple Product		& $c_{I}^0 = \langle b^{i_0}\cdots b^{i_k},1 \rangle _{H}$ & $c_I^0 = \sum_s c_{i_0 i_1s}c_{s i_2 \cdots i_k}^0$  \\
\hline
    Tensor Evaluation
    &
\begin{minipage}{.35\linewidth}
    \begin{displaymath}
        \begin{gathered}
            H^{\tilde I \tilde J} =(db^{i_1}\cdot db^{j_1})\cdots (db^{i_k}\cdot db^{j_k}) \\
             =\grad b^{i_1} \otimes \cdots \otimes \grad b^{i_k}(b^{j_1},...,b^{j_k})
        \end{gathered}
    \end{displaymath}
\end{minipage}
 & 
 \begin{minipage}{.33\linewidth}
     \begin{displaymath}
         \begin{gathered}
             \hat H^{\tilde I\tilde J}_l \equiv \langle H^{\tilde I\tilde J},b^l\rangle_{H} \\ 
              = \mbox{$\sum_{n=1}^{k^2} \prod_{s,r=1}^k$} g_{i_s j_r m_n}c_{l m_1\cdots m_{k^2}} 
         \end{gathered}
     \end{displaymath}
 \end{minipage}\\
\hline
Tensor Product & $b_{J} = b^{j_0} \grad b^{j_1} \otimes \cdots \otimes \grad b^{j_k}$ & $\langle b_J(b^{i_1},...,b^{i_k}),b^l\rangle  =\sum_s \hat H^{\tilde J\tilde I}_s c_{s j_0 l}$ \\
\hline
Frame Elements & $b^I = b^{i_0}db^{i_1} \wedge \cdots \wedge db^{i_k}$ & $\langle b^I(b_J),b^l\rangle _{H} = \langle b^I \cdot b^J,b^l\rangle _{H}$ \\
\hline
Riemannian Metric & $b^I \cdot b^J =b^{i_0}b^{j_0}\textup{det}([db^{i_q}\cdot db^{j_n}]) $ &
\begin{minipage}{.33\linewidth}
    \begin{displaymath}
        \begin{gathered}
            \langle b^I \cdot b^J,b_l\rangle_{H}  =  \\
             \mbox{$ \sum_s \sum_{\sigma \in S_k}$} \textup{sgn}(\sigma) c_{s i_0 j_0 l}\hat H_s^{\tilde I\sigma(\tilde J)} 
        \end{gathered}
    \end{displaymath}
\end{minipage}\\
\hline
Hodge Grammian & $G_{IJ} = \langle b^I,b^J\rangle _{H_k}$ & 
\begin{minipage}{.33\linewidth}
    \begin{displaymath}
        \begin{gathered}
            \langle b^I\cdot b^J,1\rangle_{H} = \\
            \sum_{s}\sum_{\sigma \in S_k}\textup{sgn}(\sigma) c_{s i_0 j_0} \hat H^{\tilde I \sigma(\tilde J)}_s  
        \end{gathered}
    \end{displaymath}
\end{minipage}\\
\hline
$d$-Energy & $E^d_{IJ} = \langle db^I,db^J\rangle _{H_{k+1}} $ & $\langle db^I \cdot db^J,1\rangle _{H_{k+1}}=\hat H^{IJ}_0$  \\
\hline
\end{tabular*}
}
\end{center}
\end{table}

%
%%      Type body of appendix/-ices here.

%%      ---------------------------------------------------------------------
%%      ---------------------------ACKNOWLEDGMENTS (OPTIONAL) ---------------
%%      ---------------------------------------------------------------------

%% ***** UNCOMMENT THE FOLLOWING LINE TO ADD ACKNOWLEDGMENTS.

 \ack 

%%      Type acknowledgments here.

Tyrus Berry acknowledges support by NSF grant DMS-1723175. Dimitrios Giannakis acknowledges support by ONR grant N00014-14-0150, YIP grant N00014-16-1-2649, NSF grant DMS-1521775, and DARPA grant HR0011-16-C-0116. We would like to thank Jeff Cheeger for stimulating conversations.
%%      ---------------------------------------------------------------------
%%      --------------------------- BIBLIOGRAPHY ----------------------------
%%      ---------------------------------------------------------------------

\frenchspacing
\bibliographystyle{cpam.bst}
\bibliography{bibliography}
\end{document}